\documentclass[11pt,a4paper]{article}
\usepackage[utf8]{inputenc}
\usepackage[T1]{fontenc}
\usepackage[english]{babel}
\usepackage{amsmath}
\usepackage{amsfonts}
\usepackage{amssymb}
\usepackage{graphicx}
\usepackage{makeidx}
\usepackage{dsfont}
\usepackage{xcolor}
\usepackage{lmodern}
\usepackage[left=2.5cm,right=2.5cm, top=3cm,bottom=2.5cm]{geometry}

\usepackage{amsthm}
\newtheorem{theointro}{Theorem}

\newtheorem{theo}{Theorem}
\theoremstyle{plain}
\newtheorem{prop}[theo]{Proposition}
\newtheorem*{prop*}{Proposition}
\newtheorem{coro}[theo]{Corollary}
\newtheorem{lemme}[theo]{Lemma}
\newtheorem*{lemme*}{Lemma}
\theoremstyle{remark}
\newtheorem{rem}[theo]{Remark}
\newtheorem*{rem*}{Remark}

\theoremstyle{definition}
\newtheorem{defin}[theo]{Definition}

\newtheorem*{defin*}{Definition}
\newtheorem*{prop_def*}{Proposition-Definition}

\newcommand{\eps}{\varepsilon}
\renewcommand{\P}[1]{\ensuremath{\mathbb P \left[ #1 \right]}}
\newcommand{\E}[1]{\ensuremath{\mathbb E \left[ #1 \right]}}

\newcommand{\Var}[1]{\ensuremath{\mathbb {V}\mathrm{ar} \left[ #1 \right]}}
\renewcommand{\leq}{\leqslant}
\renewcommand{\geq}{\geqslant}
\newcommand{\Leb}{\operatorname{Leb}}
\newcommand{\Q}{\mathds Q}
\newcommand{\R}{\mathds R}
\newcommand{\C}{\mathds C}
\newcommand{\N}{\mathds N}

\newcommand{\Nun}{\mathds N \backslash \{0\}}

\newcommand{\id}{\operatorname{id}}

\newcommand{\wR}{w^\Re}
\newcommand{\wI}{w^\Im}

\newcommand{\crochetk}{\langle k \rangle}
\newcommand{\dtv}{\mathrm d_{\mathrm{TV}}}

\makeatletter
\renewcommand{\cleardoublepage}{%
  \clearpage\thispagestyle{empty}%
  \if@twoside
    \ifodd\c@page
    \else
      \hbox{}\newpage
      \if@twocolumn
         \hbox{}\newpage
      \fi
    \fi
  \fi}
\makeatother

\usepackage{fancyhdr}
\usepackage{fancybox}

\usepackage{hyperref}
\hypersetup{hidelinks}

\title{Infinite-dimensional regularization of  McKean-Vlasov equation  with a Wasserstein diffusion}
\author{Victor \textsc{Marx} \thanks{
Université Côte d’Azur, CNRS, Laboratoire J.A. Dieudonné UMR 7351, France.       \newline
E-mail: marx@math.uni-leipzig.de }}
\date{}

\newcommand{\ZeroRoman}[1]{
  \ifcase\value{#1}\relax 0\else
  \Roman{#1}\fi}

\begin{document}
\maketitle

\begin{center}
\shadowbox{
\begin{minipage}{15cm}
\large
\textbf{Abstract.}

Much effort has been spent in recent years on restoring uniqueness of McKean-Vlasov SDEs with non-smooth coefficients. As a typical instance, the velocity field is assumed to be bounded and measurable in its space variable and Lipschitz-continuous with respect to the distance in total variation in its measure variable, see~\cite{jourdain97, mishuraverten, lacker18, chaudrufrikha18, roecknerzhang18}. 
In contrast with those works, we consider in this paper a Fokker-Planck equation driven by an infinite-dimensional noise, inspired by the diffusion models on the Wasserstein space studied in~\cite{konarovskyi17system, konarovskyivonrenesse18, marx18}. 
We prove that well-posedness of that equation holds for a drift function that might be only bounded and measurable in its measure argument, provided  that a trade-off is respected between the regularity in the finite-dimensional component and the regularity in 
the measure argument. In this regard, 
we show that the higher the regularity of~$b$ with respect to its space variable is, the lower regularity we have to assume on~$b$ with respect to its measure variable in order to restore uniqueness. 
\end{minipage}
}
\end{center}

 

\tableofcontents

\section{Introduction}

Let us denote by $\mathcal P_2(\R)$ the $L_2$-Wasserstein space, consisting in all probability measures~$\mu$ on~$\R$ such that $\int_\R x^2 \mu(\mathrm dx)$ is finite, and by $W_2$ the usual Wasserstein distance on~$\mathcal P_2(\R)$.

In this paper, we are interested in regularization by noise results for equations in infinite dimension perturbed by infinite-dimensional noises. 
More precisely, we will consider the following equation:
\begin{align}
\left\{
\begin{aligned}
\mathrm dy_t(u)&= b(y_t(u),\mu_t) \mathrm dt, \quad u \in [0,1], t \in [0,T], \\
\mu_t&=\Leb_{[0,1]}\circ y_t^{-1},
\end{aligned}
\right.
\label{eq:fp_quantile}
\end{align}
where the unknown $(y_t)_{t\in [0,T]}$ is a time-continuous process such that for each $t\in [0,T]$, $y_t$ takes values in the space $L_2^\uparrow[0,1]$ of non-decreasing square-integrable functions $f:[0,1] \to \R$. Then the measure-valued process $(\mu_t)_{t\in [0,T]}$ satisfies the following non-linear Fokker-Planck equation on the Wasserstein space $\mathcal P_2(\R)$:
\begin{align}
\partial_t \mu_t + \operatorname{div} ( b(\cdot, \mu_t)  \; \mu_t)=0. 
\label{eq:fp_mesure}
\end{align}
Remark that \eqref{eq:fp_quantile} and~\eqref{eq:fp_mesure} are deterministic equations. 
If the velocity field $b:\R \times \mathcal P_2(\R) \to \R$ is a Lipschitz-continuous function, then equation~\eqref{eq:fp_quantile} is well-posed: the proof is based on a fixed-point method and is similar to the proof of~\cite[Thm 1.1]{sznitman91}. 
Although existence might hold true in cases where $b$ is less regular, uniqueness often fails to be true when $b$ is not Lipschitz-continuous,  e.g. when for each $\mu \in \mathcal P_2(\R)$, $b(\mu):=2 \operatorname{sign}(m) \sqrt{|m|}$, where $m:= \int_\R x \mu (\mathrm dx)$ and $\operatorname{sign}(x):= \mathds 1_{\{ x \neq 0\}} \frac{x}{|x|}$. That  example is derived from the classical  Peano counter-example of ill-posedness of a one-dimensional transport equation.
Our interest is to restore uniqueness of equation~\eqref{eq:fp_quantile} for a certain class of velocity fields~$b$ by adding an infinite-dimensional diffusion. 

In the first part of this paper, we prove a  result of restoration of uniqueness for equation~\eqref{eq:fp_quantile} for a perturbative diffusion constructed on $\mathcal P_2(\R)$. 
That diffusion, which is an infinite-dimensional analogue of a  Brownian motion, is constructed as a  regularized variant of the Modified Massive Arratia flow introduced by Konarovskyi and von Renesse (see~\cite{konarovskyi17system,konarovskyivonrenesse18,marx18}). Interestingly enough,  diffusions on the Wasserstein space allow to observe averaging effects  in infinite dimension: to make it clear, we will assume in this first part that the velocity field $b:\R \times \mathcal P_2(\R) \to \R$ is $\mathcal C^2$-differentiable in the space variable (\textit{i.e.} the first variable) but only measurable and bounded in the measure variable (\textit{i.e.} the second variable). 
This comes in contrast with regularization results  in the case where the noise is of the same dimension as the ambiant space, obtained among others by Jourdain~\cite{jourdain97}, Mishura-Veretennikov~\cite{mishuraverten}, Lacker~\cite{lacker18}, Chaudru de Raynal-Frikha~\cite{chaudrufrikha18} and Röckner-Zhang~\cite{roecknerzhang18}. In those papers, the typical assumptions on the drift function is that $b$ should be bounded and measurable in the space variable and Lipschitz-continuous in total variation distance in the measure variable; in other words, finite-dimensional noises can only average the non-smoothness of a finite-dimensional argument of the drift function. 

In the second part of this paper, a connection is made between the result of the first part and the above-mentioned literature. With the aim to interpolate both aforementioned classes of assumptions on $b$, we observe a restoration of uniqueness phenomenon for a continuum of admissible drift functions~$b$, as long as a regularity condition is satisfied: roughly speaking, the assumption is $\eta+\frac{3}{2}\delta>\frac{3}{2}$, where $\eta$ is the Sobolev-regularity of $b$ in the space variable and $\delta$ is the Hölder-regularity of $b$ in the measure variable. 
 It should be already noticed at this stage that the results of this second part are obtained at the price of relaxing the related notion of weak solution and of modifying the structure of the noise, by adding an idiosyncratic Brownian motion,  as we will explain hereafter. 

Before stating the theorems proved in this paper, let us briefly recall important results on restoration of uniqueness for McKean-Valsov equations on the one hand and on construction of diffusions on the Wasserstein space on the other hand. 

\subsection{Restoration of uniqueness results for McKean-Vlasov equations}

As a matter of fact, restoration of uniqueness is now a well-understood phenomenon for classical Itô's SDEs in finite dimension; let us distinguish weak well-posedness results obtained in the aftermath of pioneer work by Stroock and Varadhan (see~\cite{stroockvaradhan69,stroockvaradhan79}) and strong well-posedness results, meaning that the solution is adapted to the filtration generated by the noise and that two solutions are almost surely indistinguishable (see Zvonkin~\cite{zvonkin74}, Veretennikov~\cite{veretennikov80}, Krylov-Röckner~\cite{krylovrockner05}). 
More recently,  restoration of uniqueness of PDEs has become an active topic of research. 
In~\cite{flandoli_gub_priola_10}, Flandoli, Gubinelli and Priola have shown that
the following transport equation with multiplicative noise
\begin{align*}
\mathrm d_t u(t,x)= (b(t,x) \cdot Du(t,x)) \mathrm dt + \sum_{i=1}^d e_i \cdot Du(t,x) \circ dW_t^i; \quad u(0,x)=u_0(x),
\end{align*}
is well-posed for Hölder-continuous drift functions~$b$, whereas the transport equation without noise is not necessarily well-posed, see e.g. the counter-example (given in~\cite{flandoli_gub_priola_10}) $b(t,x)= \frac{1}{1-\gamma}([x| \wedge R)^\gamma$ for  fixed $R>0$ and $\gamma \in (0,1)$.
Many further investigations have been made for SDEs on Hilbert spaces. In a series of papers~\cite{daprato_flandoli_10, daprato_flandoli_priola_roeckner_13,  daprato_flandoli_priola_roeckner_15}, Da Prato, Flandoli, Priola and Röckner proved that pathwise uniqueness holds for an SDE on a Hilbert space $H$ of the form
\begin{align}
\label{eq:SPDE:A:B:W}
d X_t= (A X_t + B(X_t)) \mathrm dt + \mathrm dW_t,
\end{align}
for a  certain class of self-adjoint, negative definite operators  $A: D(A) \subset H \to H$, for $W$  a cylindrical Wiener process on $H$ and for $B:H \to H$ only measurable and locally bounded. 
For an interesting introduction and a survey of results on regularization by noise phenomena, see also Flandoli's seminal lecture notes~\cite{flandoli11}. 
Various other equations in infinite-dimension have also been studied, like e.g. kinetic equations~\cite{fedrizzi_flandoli_priola_vovelle}. 
Interestingly enough, we quote in this context the recent result of Delarue \cite{Delarue2019} in which some of the above results are 
used to restore uniqueness to a mean-field game by means of an infinite dimensional common noise of Ornstein-Uhlenbeck type. Although this work 
shares some motivation with ours, it must be stressed that the dynamics 
of the particle therein
obey an operator $A$ similar to the one that appears in 
\eqref{eq:SPDE:A:B:W}. Equivalently, this says that uniqueness is restored but at the price of an extra layer of interactions
which is, in contrast to the mean-field one, purely local, arising from the Ornstein-Uhlenbeck noise.
The models that we address in the rest of the paper do not have the latter feature.

Obviously, an extensive description of restoration of uniqueness results is out of reach of this introduction, but let us focus  in a more detailed fashion on a certain class of equations, namely  McKean-Vlasov equations. 
Let $b:   \R^d \times \mathcal P_2(\R^d) \to \R^d$ be a drift function, $\sigma: \R^d \times \mathcal P_2(\R^d) \to \R^{d\times m}$ be a diffusion matrix and $({B}_t)_{t \in [0,T]}$ be a Brownian motion in $\R^m$. McKean-Vlasov equation reads as follows
\begin{align}
\label{intro:McKeanVlasov}
\left\{ 
\begin{aligned}
\mathrm dX_t &= b( X_t, \mu_t) \mathrm dt + \sigma( X_t, \mu_t)  \mathrm d{B}_t, \\
\mu_t &= \mathcal L (X_t),
\end{aligned}
\right.
\end{align}
where $\mathcal L(X_t)$ denotes the law of $X_t$. The coefficients in the stochastic differential equation~\eqref{intro:McKeanVlasov} depend on the distribution of the solution $X_t$. That dependence is called mean-field interaction, due to the link with a particle system. Indeed, equation~\eqref{intro:McKeanVlasov} should be regarded as the limit when $N \to +\infty$ of a system of particles of the following form:
\begin{align}
\label{intro:McKeanVlasov:particle}
\mathrm dX_t^i = b(X_t^i, \overline{\mu}_t^N) \mathrm dt+ \sigma (X_t^i, \overline{\mu}_t^N) \mathrm d {B}_t^i, \quad i=1, \dots, N, 
\end{align}
where $\overline{\mu}_t^N=\frac{1}{N}\sum_{j=1}^N \delta_{X_t^j}$ and $(B_t^i)_{t\in [0,T],\; 1\leq i\leq N}$ are independent Brownian motions, the latter being usually referred to as \textit{idiosyncratic noises} in order to stress the fact that there are somehow proper to a given particle. The trajectory of each particle $(X_t^i)_{t\in [0,T]}$ depends  on both the current position of the particle and  the positions of the other particles, but only via the empirical distribution $\overline{\mu}_t^N$; that is why this system is called mean-field.

Well-posedness of McKean-Vlasov SDEs has been widely studied. We here provide a tiny example
of all the existing references in the field.  Generally speaking, existence and uniqueness may be proved by a 
Picard fixed point argument on the process $(\mu_{t})_{t \in [0,T]}$ provided that the coefficients are sufficiently regular, 
say for instance that they are Lipschitz-continuous in both variables, Lipschitz-continuity 
with respect to the measure argument being understood with respect to the $L_2$-Wasserstein distance. 
This strategy is made clear in the seminal lecture notes of 
Sznitman~\cite{sznitman91}. Variants may be found, see for example (to quote earlier ones) the works of
Funaki~\cite{funaki84}, 
Gärtner \cite{gartner88} or Oelschläger~\cite{oelschlager84}. 
Interestingly enough, the proof of existence and uniqueness extends to
models with a common noise of the form
\begin{align}
\label{intro:McKeanVlasov:common}
\begin{aligned}
\mathrm dX_t &= b( X_t, \mu_t) \mathrm dt + \sigma(X_t, \mu_t)  \mathrm d{B}_t + 
\sigma_{0}(t,X_{t},\mu_{t}) {\mathrm d} W_{t},
\end{aligned}
\end{align}
with the constraint that 
$\mu_t$ now matches the conditional law of $X_{t}$ given the realization of $W$, 
where $(W_{t})_{t \in [0,T]}$ is a new Brownian motion, independent of $(B_{t})_{t \in [0,T]}$ and of dimension 
$m_{0}$, and $\sigma_{0}$ stands for a new volatility coefficient defined in the same manner as $\sigma$. 
Importantly, $\mu_{t}$ becomes random under the presence of $W$. The terminology \textit{common noise}
is better understood when we write down the analogue of \eqref{intro:McKeanVlasov:particle}, which reads:
\begin{align}
\label{intro:McKeanVlasov:particle:common}
\mathrm dX_t^i = b(X_t^i, \overline{\mu}_t^N) \mathrm dt+ \sigma (X_t^i, \overline{\mu}_t^N) \mathrm d {B}_t^i+ \sigma_{0} (X_t^i, \overline{\mu}_t^N) \mathrm d W_{t}, \quad i=1, \dots, N. 
\end{align}
The key fact here is that all the particles are driven by the same noise $(W_{t})_{t \in [0,T]}$, which is of course assumed to be independent of the collection $(B_{t}^i)_{t \in [0,T],\; 1\leq i\leq N}$. 
The reader may have a look at 
the works of Vaillancourt
\cite{vaillancourt}, Dawson and Vaillancourt \cite{dawson:vaillancourt},
Kurtz and Xiong \cite{kurtz:xiong:1,kurtz:xiong:2} 
or 
Coghi and Flandoli \cite{coghi}
for more details on 
(\ref{intro:McKeanVlasov:common}) and (\ref{intro:McKeanVlasov:particle:common}).

Let us describe  restoration of uniqueness phenomena for McKean-Vlasov equations. 
Well-posedness may  fail to be true for the  "deterministic" equation (i.e.
$\sigma \equiv 0$ in~\eqref{intro:McKeanVlasov}
in which case the randomness only comes from the initial condition)
when the drift term~$b=b(x,\mu)$ is not regular enough. 
Existing results in the field show that it is possible to require 
$b$ to be merely measurable and bounded in the space variable~$x$ and
Lipschitz-continuous in the measure variable~$\mu$, with respect to the topology generated by the total variation
distance $\dtv$, defined by 
\begin{align*}
\mathrm d_{\operatorname{TV}} (\mu,\nu):=2 \sup \left\{ \left| \int_{\R^d} f \mathrm d\mu - \int_{\R^d} f \mathrm d\nu  \right| \;; f:\R^d \to \R \text{ measurable}, \|f\|_{L_\infty} \leq 1  \right\},
\end{align*}
which is finer than the topology generated by the Wasserstein distance $W_2$. 
In particular, Jourdain~\cite{jourdain97} has proved that restoration of uniqueness holds in a weak sense for McKean-Vlasov equation~\eqref{intro:McKeanVlasov} in a case where $\sigma \equiv 1$ and $b$ is bounded, measurable and Lipschitz-continuous in its measure variable with respect to~$\dtv$.
Recently, several papers have improved the results, proving well-posedness for more general coefficients $\sigma$ in cases where $\sigma$ does not depend on $\mu$~\cite{mishuraverten,lacker18,chaudrufrikha18,roecknerzhang18}. 
In~\cite{mishuraverten}, Mishura and Veretennikov have in particular shown pathwise uniqueness under Lipschitz-continuity assumptions on $b$  with respect to the measure variable and on $\sigma$ with respect to the space variable. In~\cite{lacker18}, Lacker gives a short proof of well-posedness relying on a fixed-point argument.
Röckner and Zhang~\cite{roecknerzhang18} have extended the results to the case of
unbounded coefficients with suitable integrability properties, in the sense of Krylov-Röckner~\cite{krylovrockner05}. 
Let us emphasize once more that $b$ is assumed to be at least Lipschitz-continuous with respect to the measure-variable in  total variation distance.
This assumption might presumably be explained by the fact that the finite dimensional noise~$B$ cannot have a regularizing effect on  the infinitely many directions of the measure argument of~$b$; that is one of the reasons that drives us to study more precisely the effect of a noise defined on the Wasserstein space $\mathcal P_2(\R)$. 

\subsection{Diffusions on the Wasserstein space}

Before defining the diffusion model that we use in this text, let us briefly introduce the pre-existing models that have inspired our construction. In~\cite{vrenessesturm09}, von Renesse and Sturm constructed a so-called Wasserstein diffusion  on the space of probability measures on $[0,1]$, that is a Markovian stochastic process $(\mu_t)_{t \in [0,T]}$ with a reversibility property with respect to an entropic measure on $\mathcal P_2([0,1])$. Interestingly, the dynamics of  $(\mu_t)_{t \in [0,T]}$ are similar to the dynamics of a standard Brownian motion, in the sense that the large deviations in small time are given by the Wasserstein distance~$W_2$ and the martingale term that arises when expanding any smooth function~$\varphi$ of the measure argument along the process has exactly the square norm of the Wasserstein gradient of~$\varphi$ as local quadratic variation. 
Stochastic processes owning those diffusive features are various and several were studied in recent years. We decide in this paper to construct a diffusion inspired by the nice model of coalescing particles called Modified  Massive Arratia flow: in~\cite{konarovskii11, konarovskyi17system}, Konarovskyi introduces a diffusion model on $\mathcal P_2(\R)$ consisting in a modification of Arratia's system of coalescing particles on the real line. To wit, in Konarovskyi's model, each particle carries a mass determining its quadratic variation and moves independently of the other particles as long as it does not collide with another. To make it clear, at each collision between two particles, both particles stick together and form a unique new particle with a mass equal to the sum of the masses of both incident particles. At each time, the quadratic variation increment of a particle is given by the inverse of its mass. 
That model satisfies interesting properties, studied by Konarovskyi and von Renesse in~\cite{konarovskyi17system, konarovskyi17behavior, konarovskyivonrenesse18}, including an Itô-like formula and a Varadhan-like formula, with the Wasserstein distance~$W_2$ playing the analogous role of the Euclidean metric for the standard Brownian motion. 
Moreover, those dynamics have a canonical representation as a process of quantile functions (or increasing rearrangement functions) $(y_t)_{t \in [0,T]}$: $\forall u \in [0,1]$, $y_t(u):= \sup\{ x \in \R: \mu_t ((-\infty, x]) \leq u \}$. 

Despite a simple construction and the diffusive properties described above, the question of the uniqueness of Konarovskyi's model (not only pathwise uniqueness but also uniqueness in law) remains - as far as the author knows - open. In particular, it has a singularity at time $t=0^+$:
if $\mu_0$ has a density with respect to the Lebesgue measure, then almost surely for every $t >0$, the probability measure $\mu_t$ is a finite weighted sum of Dirac masses or, in other words, the quantile function $y_t$ is a step function. 
In~\cite{marx18}, the author overcomes lack of uniqueness by modifying Konarovskyi's model, replacing the coalescing procedure by a system of particles interacting at short range; among others, whenever the initial condition has a regular density, the solution itself remains an absolutely continuous measure. The author has proved in~\cite{marx18} the convergence of that mollified model to the Modified Massive Arratia flow. The diffusion used in this work to regularize Fokker-Planck equation is directly inspired from the works~\cite{konarovskyi17system,marx18}.

\subsection{Main results of this work}

This paper is divided into two parts and gives two complementary results of well-posedness, in a weak sense, of perturbed Fokker-Planck equations. 
First, we will address the case of a drift function~$b$ with low regularity in its measure variable but $C^2$-regularity in space. Second, we will treat a continuum  of admissible velocity fields~$b$ that somehow interpolates the assumptions of the first part and those of~\cite{jourdain97, mishuraverten, lacker18, chaudrufrikha18, roecknerzhang18}. The structure of the equations are almost similar in both parts, up to the addition of an idiosyncratic noise in the second part, which comes up with a slightly more general notion of weak solution.

\subsubsection{Restoration of uniqueness for a  velocity field merely measurable in its measure argument}

Let $T \in (0,+\infty)$ be a fixed  time.

Let us consider the Fokker-Planck equation~\eqref{eq:fp_quantile} perturbed by a diffusive noise: for each $t \in [0,T]$ and each $u \in[0,1]$, 
\begin{align}
\label{eq:fp_with_diffusion}
\left\{
\begin{aligned}
\mathrm dy_t(u)&= b(y_t(u),\mu_t) \mathrm dt 
+\frac{1}{\big( \int_0^1 \varphi(y_t(u)-y_t(v) )\mathrm dv \big)^{1/2}} \int_\R  f(k)  \Re(  e^{-ik y_t(u)} \mathrm dw(k,t) ) ; \\
\mu_t&=\Leb_{[0,1]}\circ y_t^{-1},
\end{aligned}
\right.
\end{align}
with initial condition $y_0=g$.
That equation describes a system of particles in interaction; for every $u \in [0,1]$, $(y_t(u))_{t\in [0,T]}$ denotes the trajectory of the particle indexed by $u$ and for every $t \in [0,T]$, $\mu_t$ is the distribution of the cloud of particles. There is a mean-field interaction in SDE~\eqref{eq:fp_with_diffusion}, both  through a drift term~$b$ which takes as an argument the probability measure $\mu_t$, and through the diffusion term, since the denominator $ (\int_0^1 \varphi(y_t(u)-y_t(v) )\mathrm dv)^{1/2}$ also depends on the  distribution on the real line  of the cloud of particles.
 
Let us briefly describe the different terms appearing in equation~\eqref{eq:fp_with_diffusion}. 
\begin{itemize}
\item[-] as in~\cite{konarovskyi17system,marx18}, the unknown $(y_t)_{t\in [0,T]}$ is a time-continuous stochastic process such that for each $t\in [0,T]$, $y_t$ is a random variable with values in the space $L_2^\uparrow[0,1]$.
Recall that for each $t\in [0,T]$, $u \in [0,1] \mapsto y_t(u)$ can be seen as the quantile function associated to the measure $\mu_t=\Leb_{[0,1]}\circ y_t^{-1}$ belonging to $\mathcal P_2(\R)$. Importantly, it means that we are studying  stochastic processes $(\mu_t)_{t \in [0,T]}$ in $\mathcal P_2(\R)$ that admit a canonical representation in the form of a tractable process of quantile functions. 

\item[-] the function $b: \R \times L_2[0,1] \to \R$ will be called the \emph{drift function} or velocity field. 

\item[-] the second term appearing in~\eqref{eq:fp_with_diffusion} is the \emph{diffusion term}. 
This term looks like the process introduced by the author in~\cite{marx18}. We refer to Remark~\ref{rem:comparaison avec modele marx18} below to explain the reasons why we slightly modify the shape of the diffusion here. 
It consists of several parts:

\subitem $\triangleright$ the complex-valued Brownian sheet $w$ is defined by $w:= \wR +i \wI$, where $\wR$ and $\wI$ are two independent real Brownian sheets on $\R \times [0,T]$ and $i= \sqrt{-1}$. To make it clear,    $\Re(  e^{-ik y_t(u)} \mathrm dw(k,t) )
= \cos(k y_t(u)) \mathrm d\wR(k,t) + \sin(k y_t(u)) \mathrm d\wI(k,t)$. The definition of Brownian sheets will be recalled at the beginning of Part~\ref{part:regularization} of this paper. 

\subitem $\triangleright$ the function $f$ will typically be of the form $f_\alpha(k) =\frac{1}{(1+k^2)^{\alpha/2}}$. The higher $\alpha$ is, the smoother the diffusion term is with respect to the space variable~$u$. 

\subitem $\triangleright$ we denote by $m_t(u):=  \int_0^1 \varphi(y_t(u)-y_t(v) )\mathrm dv$ the mass function. In order to avoid problems of cancellation of the mass, we will only consider in this text functions $\varphi$ which are positive everywhere on $\R$. A typical example will be the Gaussian density. Our results will also include the case where the mass is constant, \textit{i.e.} when $\varphi \equiv 1$. 
\end{itemize}

In words, the role of~$\varphi$ is to tune the local variance of the particle. This is similar to the models presented by Konarovskyi~\cite{konarovskyi17system} and by the author~\cite{marx18}, where the quadratic variation of $(y_t(u))_{t\in [0,T]}$ is proportional to $\int_0^t \frac{\mathrm ds}{m_s(u)}$. 
In order to make the comparison more precise, we may compute the local covariation field of the martingale component in 
\eqref{eq:fp_with_diffusion}, namely, for any two $u,u' \in [0,1]$,
\begin{align}
{\mathrm d} \bigl\langle y_{\cdot}(u),y_{\cdot}(u') \bigr\rangle_t 
&= \frac1{m_t(u)^{1/2} m_t(u')^{1/2}} \int_{\R} f^2(k) \cos \bigl( k (y_{t}(u) - y_{t}(u')) \bigr) {\mathrm dk} \; \mathrm dt \notag
\\
&= \frac{\Re \bigl( \mathcal F(f^2) \bigr) \bigl( y_{t}(u) - y_{t}(u') \bigr)}{m_t(u)^{1/2} m_t(u')^{1/2}} \; \mathrm dt, 
\label{correlation_deux_points}
\end{align}
where $\mathcal F(f^2)$ stands for the Fourier transform of $f^2$. 
Interestingly enough,  formula~\eqref{correlation_deux_points} may be compared with the covariation of the Modified Massive Arratia flow (see~\cite{konarovskyi17system}):
\begin{align*}
 \bigl\langle y_{\cdot}(u),y_{\cdot}(u') \bigr\rangle_t 
&= \int_{\tau_{u,u'}}^{t \wedge \tau_{u,u'}} \frac{1}{m_s(u)} \mathrm ds,
\end{align*}
where $\tau_{u,u'}:= \inf \{ t \geq 0: y_t(u)=y_t(u') \} \wedge T$ stands for the collision stopping time and $m_s(u):= \int_0^1 \mathds 1_{ \{  \tau_{u,v} \leq s  \}  } \mathrm dv$ stands for the Lebesgue measure of the particles which have already coalesced with particle $u$ at time $s$.
For instance, whenever (say to make it simple) $f(k)=f_1(k)=\frac1{(1+k^2)^{1/2}}$, 
$ \mathcal F(f^2)(x)$ behaves like $\exp(-\vert x \vert)$, which shows that the range of interaction
in~\eqref{eq:fp_with_diffusion}
 is infinite but decays exponentially fast.
By computing the Fourier transform with a residue formula, the latter may be shown to remain true whenever
$f(k)=f_n(k)$ for any integer $n \geq 1$, which proves that this new model shares 
some of the features of the approximation introduced in~\cite{marx18}
but has a longer interaction range.

The first main result of this paper is the  following theorem, stating well-posedness of equation~\eqref{eq:fp_with_diffusion}. The assumptions on the velocity field $b$ are simplified here, we refer to Definition~\ref{def:B hyp} and to Theorem~\ref{theo:wp_general_case} for more details. We denote by $\partial_1^{(j)} b$, $j=1,2$, the first two derivatives of $x \mapsto b(x,\mu)$ at fixed $\mu$. 
\begin{theointro}
\label{theointro_partie1}
Let $g$ be a strictly increasing $\mathcal C^1$-function. 
Let $f:\R \to \R$ be defined by $f(k)=f_\alpha(k):=\frac{1}{(1+k^2)^{\alpha/2}}$, with $\alpha> \frac{3}{2}$. Let $b:\R \times \mathcal P_2(\R) \to \R$ be a bounded measurable function such that for each $\mu \in \mathcal P_2(\R)$, $x\mapsto b(x,\mu)$ is twice continuously differentiable and $\partial_1 b$ and $\partial_1^{(2)} b$ are uniformly bounded on  $\R \times \mathcal P_2(\R)$. Then there is a unique weak solution to equation~\eqref{eq:fp_with_diffusion}. 
\end{theointro}
The proof of Theorem~\ref{theointro_partie1} relies on Girsanov's Theorem. The main issue is to write the drift term $b$ as a perturbation of the noise. To achieve this goal, we have to invert the diffusion coefficient; more precisely, we will resolve the following equation: find a complex-valued  process $(h_t)_{t\in [0,T]}$ satisfying for every $x \in \R$ and $t \in [0,T]$
\begin{align}
\label{eq:drift_a_inverser}
b(x,\mu_t)
=\frac{1}{\big( \int_0^1\varphi(x-y_t(v)) \mathrm dv\big)^{1/2}} \int_\R  f_\alpha(k)     \Re(e^{-ik x} h_t(k)) \mathrm dk . 
\end{align}
Thanks to the fact (and this is our rationale for it) that we chose an interaction kernel in the diffusion term 
of~\eqref{eq:fp_with_diffusion}
in a Fourier-like shape, $h$ can be defined as an inverse Fourier transform:
\begin{align*}
h_t(k)= \frac{1}{f_\alpha(k)} \mathcal F^{-1} \left( b(\cdot, \mu_t) \cdot (\varphi \ast p_t)^{-1/2}\right) (k),
\end{align*}
where $p_t$ denotes the density of the measure~$\mu_t$. 
To apply Girsanov's Theorem, $h$ should belong to $L_2(\R; \C)$; that is why we  assume in Theorem~\ref{theointro_partie1} some regularity of $b$ with respect to the variable~$x$. Remark that the higher $\alpha$ is, the more difficult it is to invert the kernel. It highlights a balance between the regularity of the process $y$ and the integrability of the Fourier inverse of~$h$. 

\begin{rem}
\label{rem:comparaison avec modele marx18}
Let us explain what happens when we perturb Fokker-Planck equation~\eqref{eq:fp_quantile} with the diffusion of~\cite{marx18}:
\begin{align*}
\mathrm dy_t(u)&= b(y_t(u),\mu_t) \mathrm dt 
+ \frac{1}{\int_0^1\varphi^2(y_t(u)-y_t(v)) \mathrm dv} \int_0^1  \varphi(y_t(u)-y_t(u'))    \mathrm dw(u',t).
\end{align*}
Then, the inversion problem  consists in finding an $L_2[0,1]$-process $(h_t)_{t\in [0,T]}$ satisfying for every $x \in \R$ and $t \in [0,T]$
\begin{align*}
b(x,\mu_t)
=\frac{1}{\int_0^1\varphi^2(x-y_t(v)) \mathrm dv} \int_0^1  \varphi(x-y_t(u')) h_t(u') \mathrm du',
\end{align*}
and equivalently in solving for every $x \in \R$ and $t \in [0,T]$
\begin{align}
\label{inverseh_ancien_modele}
(h_t \circ F_t) (x)=\frac{1}{p_t(x)} \mathcal F^{-1} \left(\frac{\mathcal F(b(\cdot ,\mu_t) \cdot (\varphi^2 \ast p_t) )}{\mathcal F(\varphi)} \right) (x),
\end{align}
where $F_t$ (resp. $p_t$) stands for the c.d.f. (resp. the density) associated to $\mu_t$. 
There are two major hindrances with  equation~\eqref{inverseh_ancien_modele}. The first one is the division by the density~$p_t$: $p_t$ is equal to zero outside the support of $\mu_t$ and $\frac{\mathcal F(b(\cdot ,\mu_t) \cdot ((\varphi^2) \ast p_t) )}{\mathcal F(\varphi)}$ has no chance to be smooth enough so that its inverse Fourier transform has compact support. It led us to change the model so that the integral in the right-hand side of~\eqref{eq:drift_a_inverser} is written as a Fourier transform. 
The second problem with~\eqref{inverseh_ancien_modele} is the division by $\mathcal F(\varphi)$. 
In the case where $\varphi$ is a Gaussian density, $\frac{1}{\mathcal F(\varphi)(k)}$ behaves like $e^{k^2/2}$. 
Even if $b$ is $\mathcal C^\infty$ with respect to its first variable, this would not be sufficient to obtain $L_2$-integrability of $h$. 
Let us try to reduce the regularity of $\varphi$: if $\varphi(x)=e^{-|x|}$, then $\frac{1}{\mathcal F(\varphi)(k)}=1+k^2$. 
Nevertheless, the density $p_t$ cannot be of class $\mathcal C^1$ with this choice of function~$\varphi$ (we refer to Remark~\ref{rem:cauchy}). 
Thus even with smooth functions $b$, the regularity of $\varphi^2 \ast p_t$ is not sufficient to compensate for the term $\frac{1}{\mathcal F(\varphi)}$ and to
insure that $h_t$ belongs to $L_2$. 
In order to solve this problem, we chose to consider two different functions $f_\alpha$ and $\varphi$ respectively at numerator and denominator of the diffusive part of~\eqref{eq:fp_with_diffusion}; this trick allows us to choose different regularities on  $f_\alpha$ and on $\varphi$. 
\end{rem}

\subsubsection{Restoration of uniqueness under a regularity assumption in both arguments of the velocity field}

The second main result of this text is a well-posedness result for a continuum of admissible drift functions~$b$ that interpolates the assumptions of Theorem~\ref{theointro_partie1} and the assumptions usually made for McKean-Vlasov equations with finite-dimensional noise, namely $b$ Lipschitz-continuous with respect to $\dtv$ in its measure variable and $b$ bounded and measurable in its space argument (see~\cite{jourdain97,mishuraverten, chaudrufrikha18, lacker18, roecknerzhang18}). 
Importantly, we succeed to do so at the price of relaxing in a dramatic manner the structure 
of the noise in hand and of the related notion of solution, adding in particular a new idiosyncratic noise denoted by~$\beta$.
In particular, as we explain below, the diffusion used to obtain the latter interpolation result does not fit the main features of the models in~\cite{konarovskyi17system,marx18} and in the first part. Among others, we lose here the underlying property of monotonicity, meaning that the solution can no longer be seen as the quantile function of the associated measure-valued process. Nevertheless, we feel that this interpolation argument is important to make the connection between the regularity assumptions of Theorem~\ref{theointro_partie1} and those used in the pre-existing literature.

Here is our new model. Let $\mu_0 \in \mathcal P_2(\R)$ be an initial condition and $\xi$ be a random variable with law $\mu_0$. 
Let $w:= \wR +i \wI$ be a complex-valued Brownian sheet defined as in equation~\eqref{eq:fp_with_diffusion}. 
Let $(\beta_t)_{t\in [0,T]}$ be a Brownian motion independent of $(w,\xi)$.  Let us consider the following McKean-Vlasov SDE with constant mass:
\begin{align}
\label{intro:interpolation_equation}
\left\{
\begin{aligned}
\mathrm dz_t&= b( z_t, \mu_t) \mathrm dt + \textstyle\int_\R f(k) \Re (  e^{-ik z_t} d w (k,t)) + \mathrm d\beta_t , \\
\mu_t &= \mathcal L^{\mathbb P}(z_t |\mathcal G_t^{\mu, W}),  \quad (\mu,w)  \perp \!\!\! \perp (\beta, \xi)\\
z_0 &=\xi, \quad    \mathcal L^{\mathbb P} (\xi)=\mu_0,
\end{aligned}
\right.
\end{align}
where $(\mathcal G_t^{\mu,W})_{t\in [0,T]}$ is the filtration generated by the Brownian sheet $w$ and by the measure-valued process $(\mu_t)_{t\in [0,T]}$ itself. 
Whereas $w$ is seen as a common noise, the Brownian motion~$\beta$ is seen here as an idiosyncratic source of randomness and $\mu_t$ can be seen as the law of $z_t$ with respect to the randomness carrying both the initial condition and the idiosyncratic noise. The addition of the new source of randomness $\beta$ is 
easily understood: similar to the Brownian motion in standard SDEs, it allows to mollify the drift in 
the space variable $x$. 
As for the conditioning 
in the identity
$\mu_t = \mathcal L^{\mathbb P}(z_t |\mathcal G_t^{\mu, W})$,
it must be compared with our presentation 
of McKean-Vlasov equations with a common noise, 
see 
\eqref{intro:McKeanVlasov:common}. The main difference between both is that the conditioning now involves 
$\mu$ itself: this comes from the fact we will allow for weak solutions, namely for solutions 
for which $\mu$ may not be adapted with respect to the common noise $w$. 
In fact, the latter causes some technical difficulties in the proofs. In particular, it requires to work with 
solutions that satisfy an additional assumption:
the observation of $z$ cannot bias the future realizations of $\mu,w$ and~$\beta$. That new requirement is known as the compatibility condition and has been often used in the study of weak solutions to stochastic equations (see~\cite{kurtz07,kurtz14}). We refer to Section~\ref{section:definition_solution} for a complete definition of the notion of weak compatible solution.

Here is our result. Let $\eta >0$ and $\delta \in [0,1]$. 
Let us consider the drift function $b: \R \times \mathcal P_2(\R) \to \R$ in the class $(H^\eta, \mathcal C^\delta)$. The definition of that class of admissible drift functions is given in Section~\ref{section:description_admissible_drift}, but roughly speaking, it contains functions $b$ such that for every fixed~$\mu$, $x\mapsto b(x,\mu)$ belongs to the Sobolev space $H^\eta(\R)$ with a Sobolev norm uniform in $\mu$, and for every fixed $x$, $\mu\mapsto b(x,\mu)$ is $\delta$-Hölder continuous in $\dtv$.
the  Hölder norm being uniform in $x$. Then, we have the following statement:
\begin{theointro}
\label{theointro_partie2}
Let $\eta >0$ and $\delta \in [0,1]$ be such that $\eta > \frac{3}{2}(1-\delta)$ and let $b$ be of class $(H^\eta, \mathcal C^\delta)$. 
Let $f:\R \to \R$ be defined by $f(k)=f_\alpha(k):=\frac{1}{(1+k^2)^{\alpha/2}}$, with $\frac{3}{2}< \alpha \leq \frac{\eta}{1-\delta}$. Then existence and uniqueness of a weak compatible solution to equation~\eqref{intro:interpolation_equation}  hold. 
\end{theointro}
The 
condition 
$\eta > \frac{3}{2}(1-\delta)$
 quantifies the minimal regularity that is needed, with our approach, to restore uniqueness. If $b$ is Lipschitz-continuous in total variation distance with respect to $\mu$ ($\delta=1$), then almost no regularity of $b$ in $x$ is needed ($\eta>0$): it is close to the assumptions of~\cite{jourdain97,mishuraverten, chaudrufrikha18, lacker18, roecknerzhang18}. If $b$ is only uniformly bounded in $\mu$ ($\delta=0$), then $x\mapsto b(x,\mu)$ should belong to $H^\eta(\R)$ for some $\eta>\frac{3}{2}$, which is slightly stronger than the assumption made in Theorem~\ref{theointro_partie1}.
In particular, it holds if $\eta=\delta=\frac{2}{3}$, in a case where $b$ is not Lipschitz-continuous in any variable.

\paragraph{Organisation of the paper.}
The part~\ref{part:regularization} of this work will be devoted to the construction of the above-mentioned variant of a diffusive model on the Wasserstein space and to the proof of the regularization result stated in Theorem~\ref{theointro_partie1}. 
In the part~\ref{part:continuum_admissible_drift}, we describe the trade-off between regularity in the space and in the measure variable and we prove Theorem~\ref{theointro_partie2}.

\paragraph{Notations.}Throughout this paper, we will always denote by $C_M$ the constants depending on~$M$, even if they change from one line to the next. 
We will also denote by $\crochetk := (1+k^2)^{1/2}$.

\section{Regularization of an ill-posed Fokker-Planck equation}
\label{part:regularization}

Let $(\Omega, \mathcal G, (\mathcal G_t)_{t \in [0,T]}, \mathbb P)$ be a filtered probability space. Assume that $\mathcal G_0$ contains all the $\mathbb P$-null sets. 

Let us recall the definition, given by Walsh~\cite[p.269]{walsh86},  of a real-valued Brownian sheet on $\R \times [0,T]$. We call \textit{white noise} on $\R \times [0,T]$ any random set function $W$ defined on the set of Borel subsets of $\R \times [0,T]$ with finite Lebesgue measure such that
\begin{itemize}
\item[-] for any $A \in \mathcal B(\R \times [0,T])$ with finite Lebesgue measure, $W(A)$ is a normally distributed random variable  with zero mean and with variance equal to $\Leb (A)$;
\item[-] for any disjoint subsets $A$ and $B \in \mathcal B(\R \times [0,T])$ with finite Lebesgue measures, $W(A)$ and $W(B)$ are independent and $W(A \cup B)=W(A)+W(B)$. 
\end{itemize}
The random function $w:\R \times [0,T] \to R$, defined for every $t \in [0,T]$ by $w(k,t)=W([0,k]\times[0,t])$ if $k\geq 0$ and $w(k,t)=W([k,0]\times[0,t])$ if $k < 0$, is called \textit{Brownian sheet} on $\R \times[0,T]$.
 Let us fix two  independent $(\mathcal G_t)_{t \in [0,T]}$-adapted Brownian sheets  $\wR$ and $\wI$   on $\R \times[0,T]$. The process $w=\wR+i \wI$ is called complex-valued Brownian sheet on $\R \times[0,T]$. 
We refer to~\cite[Theorem 1.1]{marx18} for an explanation as to how Brownian sheets are naturally related to Konarovskyi's model.

Let us rewrite hereafter equation~\eqref{eq:fp_with_diffusion}: we are looking for a solution $(y_t)_{t\in [0,T]}$ with values in $L^\uparrow_2[0,1]$ such that for any $u \in [0,1]$ and any $t \in [0,T]$, 
\begin{align}
\label{eq:definition_FP_perturbee}
\left\{ 
\begin{aligned}
dy_t(u)&=  b(y_t(u),\mu_t) \mathrm dt +\frac{1}{m_t(u)^{1/2}} \int_\R f(k)  \left(  \cos(k y_t(u))  \mathrm d\wR(k,t)+  \sin(k y_t(u))  \mathrm d\wI(k,t)\right),\\
\mu_t&= \Leb_{[0,1]} \circ y_t^{-1},
\\
y_0&=g,
\end{aligned}
\right.
\end{align}
with $m_t(u)= \int_0^1 \varphi(y_t(u)-y_t(v))\mathrm dv$. Let us define and comment the different terms appearing in that equation. First, the space $L^\uparrow_2[0,1]$ is the set of square-integrable non-decreasing functions from $[0,1]$ to $\R$. To wit, we are looking for solutions to~\eqref{eq:definition_FP_perturbee} such that for each time $t\in [0,T]$, the map $u\mapsto y_t(u)$ is non-decreasing; therefore, it is the quantile function associated to the measure $\mu_t$. In other words, equation~\eqref{eq:definition_FP_perturbee} is describing the random dynamics of the process $(\mu_t)_{t \in [0,T]}$ via its canonical representation in terms of a quantile function process $(y_t)_{t \in [0,T]}$.
We will assume that the initial condition belongs to $L_{2+}^\uparrow [0,1]$, the set of non-decreasing  functions $g: [0,1] \to \R $ such that there is $p>2$ satisfying $\int_0^1 |g(u)|^p \mathrm du <+\infty$.  
The map $\varphi:\R \to \R$ is an even function of class $\mathcal C^\infty$, decreasing on $[0,+\infty)$ and such that for every $x\in [0,+\infty)$, $\varphi(x) >0$. Typical examples of functions $\varphi$ are  the constant function $\varphi \equiv 1$ and the Gaussian density $\varphi(x)=\frac{1}{\sqrt{2\pi}} e^{-x^2/2}$. 
The map $f: \R \to \R$ is an even and square integrable function. 
The precise assumptions on the drift function $b: \R \times \mathcal P_2(\R) \to \R$ will be given later.

It should be once more emphasized that, due to the presence of the noise $w$, the process $(\mu_t)_{t \in [0,T]}$ is random. More precisely, by a straightforward computation of Itô's formula, it can be shown that the process $(\mu_t)_{t \in [0,T]}$ satisfies the following SPDE: 
\begin{align}
\label{SPDE:mu_t}
\mathrm d \mu_t +  \partial_x \Big( b(\cdot,\mu_t) \; \mu_t \Big) \mathrm dt  &= \frac{1}{2} \|f\|_{L_2(\R)}^2  \;\partial_{xx}^2\left( \frac{\mu_t}{\varphi \ast \mu_t } \right) \mathrm dt \notag \\
&\quad - \partial_x \left( \frac{\mu_t}{(\varphi \ast \mu_t)^{1/2}} \int_\R f(k) \Re \left(e^{-ik \;\cdot \; }  \mathrm dw(k,t) \right)   \right).
\end{align}
We recognize on the left-hand side of equation~\eqref{SPDE:mu_t}  a Fokker-Planck equation, with a diffusive perturbation appearing on the right-hand side due to the addition of a noise. 
Here, $\varphi \ast \mu_t := \int_\R \varphi(\cdot -x) \mathrm d\mu_t(x)$ represents the mass function. 
If $\varphi$ is close to the indicator function $\mathds 1_0$ and $b \equiv 0$, then
 equation~\eqref{SPDE:mu_t} becomes very similar  to the SPDE obtained by Konarovskyi and von Renesse for their model~\cite{konarovskyivonrenesse18}:
\begin{align*}
\mathrm d\mu_t = \Gamma(\mu_t) \mathrm dt+\operatorname{div}  (\sqrt{\mu_t} d W_t),
\end{align*}
where $\Gamma$ is defined as $\langle f, \Gamma(\nu) \rangle:= \frac{1}{2} \sum_{ x\in \operatorname{Supp}(\nu)} f''(x)$.

Let us first, in Section~\ref{section:regularity properties}, construct the diffusion, \textit{i.e.} solve equation~\eqref{eq:definition_FP_perturbee} when $b\equiv 0$. Then, in Section~\ref{section:well-posed_fokker-planck}, we will prove well-posedness of equation~\eqref{eq:definition_FP_perturbee} under the assumptions given in Theorem~\ref{theointro_partie1}.

\subsection{Construction of the diffusion without drift term}
\label{section:regularity properties}

The aim of this section is to study the solvability of the equation without drift, \textit{i.e.} equation~\eqref{eq:definition_FP_perturbee} when $b \equiv 0$: 
\begin{align}
\label{eq:real equation on y}
y_t(u)&=g(u)+ \int_0^t \frac{1}{m_s(u)^{1/2}} \int_\R \cos(k y_s(u)) f(k) \mathrm d\wR(k,s) \notag\\
&\quad\quad\quad +\int_0^t \frac{1}{m_s(u)^{1/2}} \int_\R \sin(k y_s(u)) f(k) \mathrm d\wI(k,s),
\end{align}
with $m_s(u)=\int_0^1 \varphi(y_s(u)-y_s(v)) \mathrm dv$.

In Paragraph~\ref{parag:existence, uniqueness, continuity}, we will introduce an auxiliary equation where the function $\varphi$ is replaced by a truncated function $\varphi_M$ so that the diffusion coefficient is  bounded. We will prove strong well-posedness of that equation, continuity and monotonicity with respect to the space variable $u$ of the solution. In Paragraph~\ref{parag:non-blowing}, we will deduce  existence and uniqueness of a strong solution to equation~\eqref{eq:real equation on y}.

\subsubsection{Existence, uniqueness and continuity of the diffusion}
\label{parag:existence, uniqueness, continuity}

Let $M \in \Nun$. 
Recall that $\varphi$ is even and decreasing on $[0,+\infty)$. 
Let us define $\varphi_M (x):=\varphi(|x|\wedge M)$. The interest in replacing $\varphi$ by $\varphi_M$ is that $\varphi_M$ is now bounded below by a positive constant: for each $x \in \R$, $\varphi_M(x) \geq  \varphi(M) >0$. Let us consider the following equation
\begin{align}
\label{eq:def y^M}
y^M_t(u)&=g(u)+ \int_0^t \frac{1}{m_s^M(u)^{1/2}} \int_\R \cos(k y^M_s(u)) f(k) \mathrm d\wR(k,s) \notag\\
&\quad\quad\quad +\int_0^t \frac{1}{m_s^M(u)^{1/2}} \int_\R \sin(k y^M_s(u)) f(k) \mathrm d\wI(k,s).
\end{align}
where $m_s^M(u)=\int_0^1 \varphi_M(y^M_s(u)-y^M_s(v)) \mathrm dv$. 
Since the mass function $m^M_s$ is uniformly bounded  below by~$\varphi(M)$, this equation is easier to resolve and we expect that the solution also satisfies equation~\eqref{eq:real equation on y} up to a certain stopping time. 

Following~\cite{gawareckimandrekar11}, we give the following definition:
\begin{defin}
A $(\mathcal G_t)_{t \in [0,T]}$-adapted process $(M_t)_{t\in [0,T]}$ is said to be an $L_2$-valued $(\mathcal G_t)_{t \in [0,T]}$-\textit{martingale} if for each time $t \in [0,T]$, $M_t$ belongs to $L_2([0,1], \R)$ and $\E{\|M_t\|_{L_2}}<+\infty$ and if for each $h \in L_2([0,1], \R)$, the scalar product $(M_t,h)_{L_2}$ is a real-valued $(\mathcal G_t)_{t \in [0,T]}$-martingale.
\end{defin}

Recall that $\crochetk := (1+k^2)^{1/2}$. 
The next proposition states well-posedness for equation~\eqref{eq:def y^M}:
\begin{prop}
\label{prop:sde M ex&uniq}
Let $g \in  L_{2+}^\uparrow [0,1]$. Assume that $k \mapsto \crochetk f(k)$ is square integrable. 
There exists a unique solution $y^M$ in $\mathcal C([0,T],L_2[0,1])$ to equation~\eqref{eq:def y^M}. Furthermore, the process $(y^M_t)_{t \in [0,T]}$ is an $L_2$-valued continuous $(\mathcal G_t)_{t \in [0,T]}$-martingale. 
\end{prop}

\begin{rem}
In this proposition and in every following result, we assume, at least, that $k \mapsto \crochetk f(k)$ is square integrable on $\R$. In the particular case of $f_\alpha(k)=\frac{1}{\crochetk^\alpha}=\frac{1}{(1+k^2)^{\alpha/2}}$, this assumption is equivalent to the condition $\alpha> \frac{3}{2}$. 
\label{rem:comparaison Cauchy}
\end{rem}

\begin{proof} 
The proof is based on a fixed-point argument, very similar to Proposition~3.5 in~\cite{marx18}.  
Define $(\mathcal M,\|\cdot\|_{\mathcal M})$ the space of all $z \in L_2(\Omega,C([0,T],L_2[0,1]))$ such that $(z(\omega)_t)_{t \in [0,T]}$ is a $(\mathcal G_t)_{t \in [0,T]}$-adapted process with values in $L_2[0,1]$. The definition of $\|\cdot\|_{\mathcal M}$ is given by $\|z\|_{\mathcal M}:=\E{ \sup_{t \leq T} \int_0^1 |z_t(u)|^2 \mathrm du }^{1/2}  $.  Define 
\begin{align*}
\psi(z)_t(u)&:= g(u) +\int_0^t \frac{1}{m^z_s(u)^{1/2}}  \int_\R \cos(k z_s(u)) f(k) \mathrm d\wR(k,s) \\
&\quad \quad \quad +\int_0^t \frac{1}{m^z_s(u)^{1/2}}\int_\R \sin(k z_s(u)) f(k) \mathrm d\wI(k,s) , 
\end{align*}
where $m^z_s(u)= \int_0^1 \varphi_M (z_s(u)-z_s(v)) \mathrm dv$. 
For each $z \in \mathcal M$, $\psi(z)$ belongs to $\mathcal M$, since by Burkholder-Davis-Gundy inequality, there is $C>0$ such that
\begin{align*}
\E{\sup_{t \leq T} \int_0^1 |\psi(z)_t(u)|^2 \mathrm du} 
&\leq 3\|g\|_{L_2}^2 + C \E{\int_0^1 \!\! \int_0^T \!\! \int_\R \frac{\cos^2(kz_s(u))f^2(k)}{m_s^z(u)}
 \mathrm dk \mathrm ds \mathrm du} \\
& \quad
+C \E{\int_0^1 \!\! \int_0^T \!\! \int_\R \frac{\sin^2(kz_s(u))f^2(k)}{m_s^z(u)}
 \mathrm dk \mathrm ds \mathrm du}  \\
&\leq  3\|g\|^2_{L_2} + C  \|f\|_{L_2}^2  \; \E{\int_0^1 \!\! \int_0^T \!\! \frac{1}{m_s^z(u)} \mathrm ds \mathrm du}  \leq  3\|g\|^2_{L_2} + C_M \|f\|_{L_2}^2,
\end{align*}
because $m_s^z \geq \varphi(M)>0$. 
Moreover, $(\psi(z)_t)_{t \in [0,T]}$ is an $L_2$-valued martingale and for each $t \in [0,T]$
\begin{multline*}
\E{\sup_{s \leq t} \int_0^1 |\psi(z^1)_s-\psi(z^2)_s|^2(u) \mathrm du}
\\ \leq C \E{\int_0^1 \! \! \int_0^t \!\!\int_\R \left| \frac{\cos(k z^1_s(u)) f(k)}{m^{z^1}_s(u)^{1/2}}-\frac{\cos(k z^2_s(u)) f(k)}{m^{z^2}_s(u)^{1/2}}\right|^2 \mathrm dk \mathrm ds \mathrm du} \\
+ C \E{\int_0^1 \! \! \int_0^t \!\!\int_\R \left| \frac{\sin(k z^1_s(u)) f(k)}{m^{z^1}_s(u)^{1/2}}-\frac{\sin(k z^2_s(u)) f(k)}{m^{z^2}_s(u)^{1/2}}\right|^2 \mathrm dk \mathrm ds \mathrm du}.
\end{multline*}
For every $u \in [0,1]$ and every $s \in [0,T]$, $| \cos(kz^2_s(u))-\cos(kz^1_s(u))   |\leq k |z^2_s(u)-z^1_s(u)|$ and the same Lipschitz estimate holds for the sine function. Furthermore, $\varphi_M$ is bounded  below and Lipschitz-continuous, since $\varphi_M$ is $\mathcal C^\infty$ on $(-M,M)$, continuous on $\R$ and constant on $[M,+\infty)$. Thus we have:
\begin{align}
\label{comparaison masses}
\bigg| \frac{1}{\sqrt{m_s^{z^1} (u)}}-\frac{1}{\sqrt{m_s^{z^2} (u)}}\bigg|
&= \frac{1}{\sqrt{m_s^{z^1} (u)}\sqrt{m_s^{z^2} (u)}} \frac{1}{ \sqrt{m_s^{z^1} (u)}+\sqrt{m_s^{z^2} (u)}   } \left| m_s^{z^1} (u) - m_s^{z^2} (u) \right| \notag\\
&\leq \frac{1}{2 \varphi(M)^{3/2}}  \int_0^1   \left| \varphi_M (z^1_s(u)-z^1_s(v)) - \varphi_M (z^2_s(u)-z^2_s(v)) \right| \mathrm dv \notag\\
&\leq C_M \left( | z_s^1(u)-z_s^2(u)| + \int_0^1 | z_s^1(v)-z_s^2(v)|  \mathrm dv \right). 
\end{align}
It follows that:
\begin{align*}
\E{\sup_{s \leq t} \int_0^1 |\psi(z^1)_s-\psi(z^2)_s|^2(u) \mathrm du}
& \leq C_M \E{\int_0^1 \! \! \int_0^t \!\! \int_\R |z^1_s(u)-z^2_s(u)|^2 (1+|k|^2) |f(k)|^2 \mathrm dk\mathrm ds \mathrm du}\\
&\leq C_M \int_\R \crochetk^2 |f(k)|^2 \mathrm dk \int_0^t\E{\sup_{r\leq s}\int_0^1 |z^1_r(u)-z^2_r(u)|^2  \mathrm du}\mathrm ds.
\end{align*}
Define $h_n(t):=\E{\sup_{s \leq t} \int_0^1 |\psi^{\circ n} (z^1)_s-\psi^{\circ n} (z^2)_s|^2(u)\mathrm du}$. 
There is a constant $C_{M,f}$ depending on $M$ and on $\int_\R \crochetk^2 |f(k)|^2 \mathrm dk$ such that for all $n \in \N$ and for all $t \in [0,T]$, we have $h_{n+1}(t) \leq C_{M,f} \int_0^t h_n(s) \mathrm ds$. Therefore, $h_n(t) \leq \frac{C_{M,f}^n t^n }{n!}h_0(t)$ and we deduce that $\| \psi^{\circ n} (z^1)-\psi^{\circ n} (z^2)\|^2_{\mathcal M} \leq \frac{C_{M,f} T^n }{n!}\| z^1-z^2\|^2_{\mathcal M}$. 
Let $n$ be large enough so that $\frac{C_{M,f}^n  T^n }{n!}<1$, \textit{i.e.} so that $\psi^{\circ n}$ is a contraction. Then $\psi$ admits a unique fixed point, which  we denote by $y^M$. Since $y^M=\psi(y^M)$, it is an $L_2$-valued continuous $(\mathcal G_t)_{t \in [0,T]}$-martingale. 
\end{proof}

In the following two propositions, we prove that the process $(y^M_t)_{t \in [0,T]}$ preserves continuity and monotonicity of the initial condition, under the same integrability assumption on~$f$ than in Proposition~\ref{prop:sde M ex&uniq}. 
\begin{prop}
\label{prop:continuité}
Let $g \in  L_{2+}^\uparrow[0,1]$ such that $g$ is $\delta$-Hölder for some $\delta >0$. Assume that $k \mapsto \crochetk f(k)$ is square integrable. 
There exists a version of $y^M$ in $\mathcal C([0,1]\times[0,T])$.
\end{prop}
\begin{proof}
Let $u_1,u_2 \in [0,1]$. Let $p\geq 2$ such that $p > \frac{1}{\delta}$. 
For every $t \in [0,T]$, by Burkholder-Davis-Gundy inequality, 
\begin{multline*}
\mathbb E \Big[\sup_{s \leq t} |y^M_s(u_1)-y^M_s(u_2)|^p\Big] \\
\leq C_p |g(u_1)-g(u_2)|^p 
+ C_{p,M} \mathbb E \Big[\Big( \int_0^t \!\!  \int_\R  \crochetk^2 f(k)^2  \mathrm dk |y^M_s(u_1)-y^M_s(u_2)|^2  \mathrm ds \Big)^\frac{p}{2}\Big].
\end{multline*}
It follows that
\begin{multline*}
\mathbb E \Big[\sup_{s \leq t} |y^M_s(u_1)-y^M_s(u_2)|^p\Big] \\
\begin{aligned}
&\leq C_p |g(u_1)-g(u_2)|^p 
+ C_{p,M,f}t^{p/2-1}  \mathbb E \Big[ \int_0^t |y^M_s(u_1)-y^M_s(u_2)|^p  \mathrm ds \Big]\\
&\leq C_p |g(u_1)-g(u_2)|^p 
+ C_{p,M,f}t^{p/2-1}  \int_0^t \E{\sup_{r\leq s} |y^M_r(u_1)-y^M_r(u_2)|^p}  \mathrm ds .  
\end{aligned}
\end{multline*}
By Gronwall's Lemma, and using the $\delta$-Hölder regularity of $g$, we have:
\begin{align*}
\E{\sup_{t \leq T} |y^M_t(u_1)-y^M_t(u_2)|^p}
\leq C_{M,p,f} |u_1-u_2|^{p \delta}. 
\end{align*}
Remark that $p\delta-1>0$. Let us apply Kolmogorov's Lemma (e.g in~\cite[Theorem I.2.1, p.26]{revuzyor13} with $d=1$, $\gamma=p$ and $\eps=p\delta-1$). Thus there  exists a version $\widetilde{y}^M$ of $y^M$ in $\mathcal C([0,1] \times [0,T])$. 
\end{proof}

\begin{prop}
Let $g \in  L_{2+}^\uparrow[0,1]$. Let $u_1<u_2 \in [0,1]$ be such that $g(u_1)<g(u_2)$. 
Assume that $k \mapsto \crochetk f(k)$ is square integrable. 
Let $y^M$ be the solution to equation~\eqref{eq:def y^M}. Then almost surely and for every $t \in [0,T]$, $y^M_t(u_1)<y^M_t(u_2)$. 
\label{prop:growth}
\end{prop}

\begin{proof}
Let $u_1<u_2 \in [0,1]$ be such that $g(u_1)<g(u_2)$. Thus the process $Y_t=y^M_t(u_2)-y^M_t(u_1)$ satisfies 
\begin{align}
\label{eq:growth}
Y_t=g(u_2)-g(u_1) +\int_0^t  Y_s \mathrm dN^M_s, 
\end{align}
where we denote
\begin{align*}
N^M_t&=\int_0^t \!\!  \int_\R \mathds 1_{\{y^M_s(u_1)\neq y^M_s(u_2) \}}
\frac{\theta^\Re_M(y^M_s(u_2),k,s)-\theta^\Re_M(y^M_s(u_1),k,s)}
{y^M_s(u_2)-y^M_s(u_1)} f(k) \mathrm d\wR(k,s) \\
& \quad + \int_0^t \!\!  \int_\R \mathds 1_{\{y^M_s(u_1)\neq y^M_s(u_2) \}}
\frac{\theta^\Im_M(y^M_s(u_2),k,s)-\theta^\Im_M(y^M_s(u_1),k,s)}
{y^M_s(u_2)-y^M_s(u_1)} f(k) \mathrm d\wI(k,s)
\end{align*} 
and $\theta^\Re_M(x,k,s)=\frac{\cos(kx)}{(\int_0^1 \varphi_M(x-y^M_s(v))\mathrm dv)^{1/2}}$
and $\theta^\Im_M(x,k,s)=\frac{\sin(kx)}{(\int_0^1 \varphi_M(x-y^M_s(v))\mathrm dv)^{1/2}}$. 
Thus we have
\begin{align*}
\langle N^M,N^M \rangle_t&=   \int_0^t \!\! \int_\R \mathds 1_{\{y^M_s(u_1)\neq y^M_s(u_2) \}}
\left|\frac{\theta^\Re_M(y^M_s(u_2),k,s)-\theta^\Re_M(y^M_s(u_1),k,s)}
{y^M_s(u_2)-y^M_s(u_1)}\right|^2 f(k)^2 \mathrm dk \mathrm ds \\
&\quad + \int_0^t \!\! \int_\R \mathds 1_{\{y^M_s(u_1)\neq y^M_s(u_2) \}}
\left|\frac{\theta^\Im_M(y^M_s(u_2),k,s)-\theta^\Im_M(y^M_s(u_1),k,s)}
{y^M_s(u_2)-y^M_s(u_1)}\right|^2 f(k)^2 \mathrm dk \mathrm ds. 
\end{align*}
For every $x_1$, $x_2 \in \R$, for every $k\in \R$ and for every $s \in [0,T]$, we have the following two estimates:
\begin{align*}
\vert \cos(kx_2)-\cos(kx_1)   \vert
&\leq k |x_2-x_1|,
\\
\left|\int_0^1 \varphi_M(x_2-y^M_s(v))\mathrm dv-\int_0^1 \varphi_M(x_1-y^M_s(v))\mathrm dv  \right| &\leq \operatorname{Lip}(\varphi_M) |x_2-x_1|. 
\end{align*}
It follows, by the same computation as~\eqref{comparaison masses}, that
\begin{align*}
\left|\frac{1}{\big(\int_0^1 \varphi_M(x_2-y^M_s(v))\mathrm dv\big)^{1/2}}-\frac{1}{\big(\int_0^1 \varphi_M(x_1-y^M_s(v))\mathrm dv\big)^{1/2}}  \right| &\leq C_M |x_2-x_1|.
\end{align*}
Thus for every $x_1$, $x_2 \in \R$, 
\begin{multline*}
\left| \theta^\Re_M(x_2,k,s)-\theta^\Re_M(x_1,k,s) \right| \\
\begin{aligned}
&\leq \left| \frac{\cos(kx_2)-\cos(kx_1)}{(\int_0^1 \varphi_M(x_2-y^M_s(v))\mathrm dv)^{1/2}} \right| \\
&\quad  + \left| \cos(kx_1)  \right|
\left|\frac{1}{\big(\int_0^1 \varphi_M(x_2-y^M_s(v))\mathrm dv\big)^{1/2}}-\frac{1}{\big(\int_0^1 \varphi_M(x_1-y^M_s(v))\mathrm dv\big)^{1/2}}  \right|\\
&\leq  C_M \crochetk |x_2-x_1|,
\end{aligned}
\end{multline*}
Therefore, for every $s \in [0,T]$, 
\begin{align*}
 \mathds 1_{\{y^M_s(u_1)\neq y^M_s(u_2) \}}
\left|\frac{\theta^\Re_M(y^M_s(u_2),k,s)-\theta^\Re_M(y^M_s(u_1),k,s)}
{y^M_s(u_2)-y^M_s(u_1)}\right|^2 
\leq C_M \crochetk^2.
\end{align*}
We have a similar bound on $\theta^\Im$. 
We deduce that $\frac{\mathrm d\langle N^M,N^M \rangle_s}{\mathrm ds} \leq C_M \int_\R \crochetk^2 f(k)^2 \mathrm dk$ for each $s \in [0,T]$. Hence the stochastic differential equation~\eqref{eq:growth} has a unique solution and it is $Y_t=(g(u_2)-g(u_1)) \exp\left(N^M_t-\frac1{2}\langle N^M,N^M \rangle_t \right)$.  In particular
\begin{align*}
y^M_t(u_2)-y^M_t(u_1)=(g(u_2)-g(u_1)) \exp\left(N^M_t-\frac1{2}\langle N^M,N^M \rangle_t \right).
\end{align*}
Since $g(u_1)<g(u_2)$, we deduce that for every $t \in [0,T]$, $Y_t >0$. Thus for every $t \in [0,T]$, $y^M_t(u_1)<y^M_t(u_2)$. 
\end{proof}

\begin{coro}
\label{coro:growth}
Let $g \in  L_{2+}^\uparrow[0,1]$ such that $g$ is $\delta$-Hölder for some $\delta >\frac{1}{2}$. Assume that $k \mapsto \crochetk f(k)$ is square integrable. 
Then there is a  version $y^M$ of the solution to equation~\eqref{eq:def y^M} in $\mathcal C([0,1] \times [0,T])$ such that almost surely, for each $t\in [0,T]$, $u\in [0,1] \mapsto y^M_t(u)$ is strictly increasing. 
\end{coro}
\begin{proof}
By Proposition~\ref{prop:continuité}, we know that there is a version $y^M$ of the solution to~\eqref{eq:def y^M} jointly continuous in time and space. 

Furthermore, by Proposition~\ref{prop:growth}, there exists an almost sure event $\widetilde{\Omega}$ under which $y^M$ belongs to $\mathcal C([0,1] \times [0,T])$ and for every $t\in [0,T]$ and for every $u_1,u_2 \in \Q \cap [0,1]$ such that $u_1<u_2$, we have $y_t^M(u_1)<y_t^M(u_2)$. 
Since $u \mapsto y_t^M(u)$ is continuous under the event $\widetilde{\Omega}$, we deduce that $y_t^M(u_1)<y_t^M(u_2)$ holds with every $u_1<u_2 \in [0,1]$. 
\end{proof}

\subsubsection{Construction of a non-blowing solution on the global time interval \texorpdfstring{$[0,T]$}{[0,T]}}
\label{parag:non-blowing}

In this paragraph, we build a solution to equation~\eqref{eq:real equation on y}, provided that the initial condition~$g$ is smooth enough.

\begin{defin}
\label{def:g_un}
Let $\mathbf G^1$ denote the set of $\mathcal C^1$-functions $g:[0,1]\to \R$ such that for all $u<v$, $g(u)<g(v)$.  
\end{defin}

Remark that every $g$ in $\mathbf G^1$ is the quantile function of a measure $\mu_0$, which is absolutely continuous with respect to the Lebesgue measure on $\R$. Indeed, let $F_0$ be the inverse map of~$g$, \textit{i.e.} the unique map $F_0:[g(0),g(1)] \to [0,1]$ such that $F_0 \circ g= \id_{[0,1]}$, and let $p_0$ the first derivative of~$F_0$. Then $F_0$ and $g$ are respectively the cumulative distribution function (c.d.f) and the quantile function of the measure $\mu_0$ with density~$p_0$. Furthermore, $p_0$ is continuous and has a compact support equal to $[g(0),g(1)]$. 
 
Let $g \in \mathbf G^1$. 
Let us fix $M_0$ an integer such that $M_0> g(1)-g(0)$. 
We want to construct a solution to equation~\eqref{eq:real equation on y} starting at $g$, well-defined and continuous on the whole interval $[0,T]$. We will construct it on the basis of the family $(y^M)_{M \geq M_0}$ of solutions to equation~\eqref{eq:def y^M} for each $M \geq M_0$. 
Since $g$ belongs to $\mathbf G^1$, the assumptions made in Propositions~\ref{prop:sde M ex&uniq} and \ref{prop:continuité} and Corollary~\ref{coro:growth}  can be applied. Thus for every $u,v \in [0,1]$ and for every $t \in [0,T]$, 
$|y^M_t(u)-y^M_t(v)|\leq y^M_t(1)-y^M_t(0)$. 
For every $M, M' \geq M_0$, define
\begin{align*}
\tau_{M}(y^{M'}):= \inf \left\{t \geq 0: y^{M'}_t(1)-y^{M'}_t(0) \geq M\right\} \wedge T. 
\end{align*}
Since $M> g(1)-g(0)$ and since the process $y^{M'}_\cdot(1)-y^{M'}_\cdot(0)$ is continuous, $\tau_M(y^{M'})>0$ almost surely for every $M,M' \geq M_0$. 
Assume that $M \leq M'$. Then for every $s \leq \tau_M(y^{M'})$, for every $u,v \in [0,1]$, $|y^{M'}_s(u)-y^{M'}_s(v)| \leq M \leq M'$ and thus 
\begin{align*}
 \varphi_{M'}(y^{M'}_s(u)-y^{M'}_s(v))= \varphi(y^{M'}_s(u)-y^{M'}_s(v))=\varphi_M(y^{M'}_s(u)-y^{M'}_s(v)). 
\end{align*}
Let $\sigma= \tau_M(y^M) \wedge \tau_M(y^{M'})$. We deduce from the latter equality that the processes $(y^M_{ t \wedge \sigma})_{t\in [0,T]}$ and $(y^{M'}_{ t \wedge \sigma})_{t\in [0,T]}$ are both solutions to the same stochastic differential equation:
\begin{align}
\label{eq:sde M,M'}
z_t(u)=g(u)+\int_0^{t\wedge \sigma} \frac{1}{m^z_s(u)^{1/2}}  \left(\int_\R \cos(k z_s(u)) f(k) \mathrm d\wR(k,s) + \int_\R \sin(k z_s(u)) f(k) \mathrm d\wI(k,s)\right),
\end{align}
where $m^z_s(u)=\int_0^1 \varphi_M(z_s(u)-z_s(v)) \mathrm dv$. 

Assume that $k \mapsto \crochetk f(k)$ is square integrable. 
Therefore, by pathwise uniqueness of the solution to equation~\eqref{eq:sde M,M'}, which follows from the same argument as in Proposition~\ref{prop:sde M ex&uniq}, we have for all $u \in [0,1]$, for all $t \in [0,T]$, $y^M_{t \wedge \sigma}(u)=y^{M'}_{t \wedge \sigma}(u)$, whence $\tau_M(y^M)=\tau_M(y^{M'})$. From now on, we will denote that stopping time by $\tau_M$. The sequence of stopping times $(\tau_M)_{M \geq 1}$ is non-decreasing.

Setting $\tau_{M_0-1}=0$, we define $y_t(u):= \mathds 1_{\{t=0\}} g(u) +\sum_{M=M_0}^{+\infty} \mathds 1_{\{t \in (\tau_{M-1},\tau_M]\} }y^M_t(u)$ for every $t\in [0,T]$ and $u \in [0,1]$. 
Let $\tau_\infty:=\sup_{M\geq M_0} \tau_M$. Clearly, $\tau_\infty>0$ almost surely. Since $\tau_M \leq T$ for every $M \geq M_0$, we have $\tau_\infty \leq T$.  
Furthermore, for each $M \geq M_0$, $y=y^M$ on $[0,\tau_M]$ and on the interval $[0,\tau_\infty)$, $(y_t)_{t\in [0,T]}$ is solution to equation~\eqref{eq:real equation on y}.

Let us remark that $\mathbb P$-almost surely,  $u \mapsto y_t(u)$ is strictly increasing for every $t\in [0,T]$. Moreover, the following proposition states that it is the unique solution in $\mathcal C([0,1] \times [0,T])$ to equation~\eqref{eq:real equation on y}. 

\begin{prop}
\label{prop:exist and uniq of eq1}
Let $g \in \mathbf G^1$. Assume that $k \mapsto \crochetk f(k)$ is square integrable.  There exists a unique solution $y$ in $\mathcal C([0,1]\times [0,T])$ to equation~\eqref{eq:real equation on y}  and this solution is defined on $[0,T]$. Furthermore, the process $(y_t)_{t \in [0,T]}$ is $(\mathcal G_t)_{t \in [0,T]}$-adapted. 
\end{prop}

\begin{proof}
First, we prove that $\tau_\infty$ defined above is almost surely equal to $T$. Let $M \geq M_0$. Let us estimate $\P{\tau_M<T}$. Define $z^M_t:=y^M_t(1)-y^M_t(0)$. Then $(z^M_t)_{t\in[0,T]}$ is a continuous and square integrable local martingale on $[0,T]$ and thus there is a standard $\mathbb P$-Brownian motion~$\beta$ such that $z^M_t=g(1)-g(0) +\beta_{\langle z^M,z^M \rangle_t}$. Moreover, $\tau_M=\inf \{ t :z^M_t \geq M\} \wedge T$. Under the event $\{ \tau_M <T\}$, there is a random time $t_0 \in [0,T)$ such that $z^M_{t_0} \geq M$ whereas for all $t \in [0,T]$, $z^M_t >0$ by Proposition~\ref{prop:growth}. 
Let us define the process $(\gamma_t)_{t \geq 0}$ by $\gamma_t:=g(1)-g(0)+\beta_t$. Under the measure $\mathbb P$, it is a Brownian motion starting at $g(1)-g(0) \in (0,M)$. Moreover, under the event $\{ \tau_M <T\}$, $(\gamma_t)_{t \geq 0}$ reaches the level $M$ before it reaches the level $0$. 
Therefore, $\P{\tau_M <T} \leq \P{(\gamma_t)_{t \geq 0} \text{ reaches $M$ before $0$}}=\frac{g(1)-g(0)}{M}$. 
Since $\{ \tau_M <T \}_{M \geq M_0}$ is a non-increasing sequence of events, we deduce that $\P{\bigcap_{\{M\geq M_0\}}\{ \tau_M <T\}}=0$. 
Thus $\mathbb P$-almost surely, there exists  $M\geq M_0$ such that $\tau_M = T$, whence $y=y^M$. It follows that $\tau_\infty=T$ almost surely. Thus $y$ is a continuous solution to equation~\eqref{eq:real equation on y} defined on $[0,T]$.

Let us now prove pathwise uniqueness. 
Let $x^1$ and $x^2$ be two solutions on $(\Omega, \mathcal G, \mathbb P)$ to equation~\eqref{eq:real equation on y} in  $\mathcal C([0,1]\times [0,T])$. Let $\eps >0$. For every $M \geq M_0$, let us define the following event: $A^i_M:= \{\omega \in \Omega: \sup_{u \in [0,1], t\in [0,T]} |x_t^i(u)|(\omega) \leq \frac{M}{2}\}$, $i=1,2$. Let $A_M:= A_M^1\cap A_M^2$. The sequence of events $(A_M)_{M \geq M_0}$ is non-decreasing and it follows from the fact that $x^1$ and $x^2$ are continuous that $\P{\bigcup_{M \geq M_0} A_M}=1$. Thus there is $M$ such that $\P{A_M}> 1-\eps$.

Let $M$ be such that $\P{A_M}> 1-\eps$. Let $\tau_M^i:= \inf \left\{t \geq 0: x^i_t(1)-x^i_t(0) \geq M\right\} \wedge T$ and $\tau_M=\tau_M^1 \wedge \tau_M^2$. 
For $i=1,2$, the same argument as the one given in Corollary~\ref{coro:growth} implies that almost surely for each $t \in [0,T]$, $u \mapsto x^i_{t \wedge \tau_M}(u)$ is strictly increasing. Therefore, under the event $A_M$, the equality $\tau_M=T$ holds. 
Moreover, the processes $( x^1_{t \wedge \tau_M})_{t\in [0,T]}$ and $( x^2_{t \wedge \tau_M})_{t\in [0,T]}$ satisfy equation~\eqref{eq:def y^M} up to the same stopping time $\tau_M$. By Proposition~\ref{prop:sde M ex&uniq}, pathwise uniqueness holds for equation~\eqref{eq:def y^M}, so  $\P{x^1_{\cdot \wedge \tau_M} \neq x^2_{\cdot \wedge \tau_M}}=0$. In particular, $0=\P{\{x^1_{\cdot \wedge \tau_M} \neq x^2_{\cdot \wedge \tau_M}\} \cap A_M}=\P{\{x^1 \neq x^2\} \cap A_M}$. It follows that $\P{x^1 \neq x^2}< \eps$ for every $\eps >0$. Since $\eps>0$ is arbitrary, we conclude that $\P{x^1 \neq x^2}=0$ and pathwise uniqueness holds for~\eqref{eq:real equation on y}. 
\end{proof}

\subsubsection{Higher regularity  of the solution map}
\label{parag:higher regularity}

Let us remark that there is a strong relation between the regularity, for each fixed $t$, of the map $u \mapsto y_t(u)$ and the rate of decay at infinity of $f$. We have already seen in Proposition~\ref{prop:exist and uniq of eq1} that the afore-mentioned map is continuous for every $t\in [0,T]$ if $k \mapsto \crochetk  f(k)$ belongs to $L_2(\R)$. 
By differentiating formally $y$ with respect to $u$, we expect that the derivative of $y$ is  a solution to  the following linear stochastic differential equation for every $u \in [0,1]$:
\begin{align}
\label{eq:def z}
z_t(u)=g'(u)+\!\!\int_0^t \!\! z_s(u) \int_\R \!\! \phi^\Re (u,k,s) f(k) \mathrm d\wR(k,s)
+\!\!\int_0^t \!\! z_s(u) \int_\R \!\!\phi^\Im (u,k,s) f(k)\mathrm d\wI(k,s),
\end{align}
where
\begin{align*}
\phi^\Re (u,k,s)&:= \frac{-k\sin(ky_s(u))}{(\int_0^1 \varphi(y_s(u)-y_s(v)) \mathrm dv)^{1/2}} -  \frac{\cos(ky_s(u))\int_0^1 \varphi'(y_s(u)-y_s(v))\mathrm dv}{2(\int_0^1 \varphi(y_s(u)-y_s(v)) \mathrm dv)^{3/2}}  ;
\\
\phi^\Im (u,k,s)&:= \frac{k\cos(ky_s(u))}{(\int_0^1 \varphi(y_s(u)-y_s(v)) \mathrm dv)^{1/2}} - \frac{\sin(ky_s(u))\int_0^1 \varphi'(y_s(u)-y_s(v))\mathrm dv}{2(\int_0^1 \varphi(y_s(u)-y_s(v)) \mathrm dv)^{3/2}}.
\end{align*}

For every $j \in \N$ and every $\theta \in [0,1)$, let $\mathbf G^{j+\theta}$ denote the set of functions $g \in \mathbf G^{1}$ which are $j$-times differentiable and such that $g^{(j)}$ is $\theta$-Hölder continuous.

\begin{prop}
\label{prop:derivee de y}
Let $\theta \in (0,1)$. 
Let $g \in \mathbf G^{1+\theta}$. 
Assume that $k \mapsto \crochetk^{1+\theta} f(k)$ is square integrable. 
Almost surely, for every $t \in [0,T]$, the map $u \mapsto y_t(u)$ belongs to $\mathbf G^{1+\theta'}$ for  every $0\leq \theta' < \theta$ and $(\partial_u y_t)_{t \in [0,T]}$ satisfies equation~\eqref{eq:def z}.
Moreover, the derivative has the following explicit form:
\begin{multline*}
\partial_u y_t(u)=g'(u) \exp\bigg(\int_0^t \!\! \int_\R \phi^\Re(u,k,s) f(k) \mathrm d\wR(k,s) +\int_0^t \!\! \int_\R \phi^\Im(u,k,s) f(k) \mathrm d\wI(k,s)\\
 -\frac{1}{2} \int_0^t \!\! \int_\R (\phi^\Re(u,k,s)^2+\phi^\Im(u,k,s)^2) f(k)^2 \mathrm dk \mathrm ds \bigg).
\end{multline*}
More generally, if for an integer $j\geq 1$,  $g$ belongs to $\mathbf G^{j+\theta}$ and $k\mapsto \crochetk^{j+\theta} f(k)$ is square integrable, then almost surely, for every $t \in [0,T]$, the map $u \mapsto y_t(u)$ belongs to~$\mathbf G^{j+\theta'}$ for every $0\leq \theta'< \theta$.
\end{prop}

\begin{rem}
\label{rem:cauchy}
Let us consider the case of $f_\alpha(k)=\frac{1}{\crochetk^\alpha}$. The assumption $\crochetk^{j+\theta} f_\alpha(k) \in L_2(\R)$ is  equivalent to the condition $\alpha >j+\theta + \frac{1}{2}$. If $f$ is  the Cauchy density $f(k)=\frac{1}{1+k^2}$, then the process $u\mapsto y_t(u)$ is differentiable and its derivative is $\theta'$-Hölder continuous for every $\theta' < \frac{1}{2}$. 

By the property of monotonicity of $y_t$, we deduce that almost surely, for every $t\in [0,T]$, $u\mapsto \partial_u y_t(u) >0$. Recall that the c.d.f. $F_t$ associated to $y_t$ is equal to $F_t= (y_t)^{-1}$ and that the density of $p_t$ is the first derivative of $F_t$. Therefore, for every $u \in [0,1]$, $F_t(y_t(u))=u$ and $p_t(y_t(u)) \partial_u y_t(u)=1$. Thus for every $x\in [y_t(0), y_t(1)]$, 
\begin{align*}
p_t(x)= \frac{1}{\partial_u y_t(F_t(x))}.
\end{align*}
It follows that $p_t$ has the same regularity than $\partial_u y_t$. If $f(k)=\frac{1}{1+k^2}$, then $p_t$ is $\theta'$-Hölder continuous for every $\theta' < \frac{1}{2}$. 
\end{rem}

In order to prove Proposition~\ref{prop:derivee de y}, we first replace $\varphi$ by a function $\varphi_M$ bounded below as previously and prove the result for the corresponding equation: since every coefficient is now Lipschitz-continuous, the computations are classical. Finally, we recover the result of Proposition~\ref{prop:derivee de y} for a non-truncated function $\varphi$ by using the fact that almost surely for each time $t$, the solution $y$ is equal to $y_{\cdot \wedge \tau_M}$ for $M$ large enough. The interest reader may find a comprehensive proof of Proposition~\ref{prop:derivee de y} in~\cite[Prop. II.9]{marx_thesis}.

\subsection{Well-posedness of a perturbed Fokker-Planck equation}
\label{section:well-posed_fokker-planck}

We denote by $\mathcal F \phi$ the Fourier transform of a function~$\phi$;  if $\phi$ belongs to $L_1(\R)$, $\mathcal F \phi(x)= \frac{1}{\sqrt{2 \pi}}\int_\R e^{-ixy}\phi(y) \mathrm dy$. Recall that Plancherel's formula states that: $\|\phi\|_{L_2}=\|\mathcal F \phi\|_{L_2}$.  We denote by $\mathcal F^{-1}$ the inverse Fourier transform.

Let us recall equation~\eqref{eq:definition_FP_perturbee}:
\begin{align*}
\left\{ 
\begin{aligned}
dy_t(u)&=  b(y_t(u),\mu_t) \mathrm dt +\frac{1}{m_t(u)^{1/2}} \int_\R f(k) \Re \left(  e^{-ik y_t(u)}  \mathrm dw(k,t)\right),\\
\mu_t&= \Leb_{[0,1]} \circ y_t^{-1},
\\
y_0&=g,
\end{aligned}
\right.
\end{align*}
with $m_t(u)= \int_0^1 \varphi(y_t(u)-y_t(v))\mathrm dv$. 

In the previous section, we studied equation~\eqref{eq:definition_FP_perturbee} in the case where the drift $b$ is zero. As explained in the introduction, the well-posedness of equation~\eqref{eq:definition_FP_perturbee} will be deduced from the well-posedness of the diffusion without drift by a Girsanov transformation. Therefore, we will have to construct an $L_2$-valued process $(h_t)_{t \in [0,T]}$ satisfying equation~\eqref{eq:drift_a_inverser}. 

Importantly, we will assume the following assumption on~$b: \R \times \mathcal P_2(\R) \to \R$.

\begin{defin}
A measurable function $b: \R \times \mathcal P_2(\R) \to \R$ is said to satisfy the $b$-hypotheses of order $j\in \Nun$ if:
\begin{itemize}
\item [$(B1)$] for every $\mu \in \mathcal P_2(\R)$, $x \mapsto b(x,\mu)$ is continuous and $j$-times differentiable on $\R$;
\item [$(B2)$]  for every $i \in \{0,1,\dots,j\}$, there is a sequence $(C_i(M))_{M \geq M_0}$ such that the inequality $\left| \partial_1^{(i)} b (x,\mu) \right | \leq C_i(M)$  holds for every $x \in \R$ and for every $\mu \in \mathcal P_2(\R)$ with compact support satisfying $|\operatorname{Supp} \mu|\leq M$. 
\item[$(B3)$] the sequence $(C_0(M))_{M \geq M_0}$ satisfies $\frac{C_0(M)}{M} \underset{M \to +\infty}{\longrightarrow} 0$.
\end{itemize}
We say that $B:(x,z) \in \R \times L_2[0,1] \mapsto b(x,\Leb_{[0,1]} \circ z^{-1})$ satisfies the $B$-hypotheses of order~$j$ if the associated $b$ satisfies the $b$-hypotheses of order $j$. 
\label{def:B hyp}
\end{defin}
Of course, every bounded function $b$ such that $(B1)$  holds true satisfies the $b$-hypotheses of order $j$.
Moreover, Definition~\ref{def:B hyp} also allows  us also to consider unbounded functions $b$, for which $x \mapsto b(x,\mu)$ is uniformly bounded when the support of $\mu$ is controled, in the sense of assumption~$(B2)$. Assumption~$(B3)$ is here to ensure that the solution to the drifted equation  almost surely does not blow up before final time~$T$, as we will explain hereafter.

\begin{rem}
\label{exemples de drifts}
Let us give a few  examples of admissible drift functions:
\begin{itemize}

\item[-]  Let $b_1(x,\mu):= \mathbb E_\mu[a(x,Y)]=\int_\R a(x,y) \mathrm d\mu (y)$ or equivalently $B_1(x,z):=\int_0^1 a(x,z(u)) \mathrm du$. If $a:\R^2 \to R$ is bounded and $x \mapsto a(x,y)$ is $j$-times differentiable with bounded derivatives, then $b_1$ satisfies the $b$-hypotheses of order $j$. 
\item[-] Let $b_2(x,\mu):=a( \mathbb E_\mu [Y])=a(\int_\R y \mathrm d\mu(y))$. If $a$ is bounded, then $b_2$ satisfies the $b$-hypotheses of every order. 
\item [-]  Let $b_3(x,\mu):=a(x, \mathbb E_\mu [\psi(Y)])=a(x,\int_\R \psi(y) \mathrm d\mu(y))$. If $a$ is bounded and $j$-times partially  differentiable in its first argument with bounded derivatives and if $\psi$ is measurable, then $b_3$ satisfies the $b$-hypotheses of order $j$. 

\item[-] Let $b_4(x,\mu)= a(x)\operatorname{\mathbb Var}_\mu[Y]^\eta $, where $\eta<\frac{1}{2}$. If $a$ is bounded with $j$ bounded derivatives, then $b_4$ satisfies the $b$-hypotheses of order~$j$. 
Indeed, if $\mu$ has a compact support with $|\operatorname{Supp} \mu|\leq M$, then $\Var{Y}^\eta \leq M^{2\eta}$; thus $\frac{C_0(M)}{M}\leq\|a\|_{L_\infty}\frac{M^{2\eta}}{M} \to 0$.
\end{itemize}
Let us emphasize the fact that, in the first example above, $\mu \mapsto b_1(x,\mu)$ is Lipschitz-continuous in total variation distance, with a Lipschitz constant uniform in $x$ given by the $L_\infty$-norm of $a$. This means that restoration of uniqueness for the Fokker-Planck equation associated with $b_1$ can also be obtained with finite-dimensional noise, since $b_1$ satisfies the assumptions of~\cite{jourdain97}. 
Actually, Jourdain proved those results even in cases where $a$ is only bounded. Therefore, in the case of the first example above, our assumptions are more restrictive than previous existing litterature. 
The main interest of the study conducted here is that our result applies for examples $b_2$, $b_3$ and $b_4$, which do not satisfy the assumptions  of~\cite{jourdain97,mishuraverten}.
\end{rem}

The regularity assumptions on the $x$-dependence of the drift function $b$ depend on the decay rate  of $f$ at infinity. 
Recall that the faster $f$ decays at infinity, the higher regularity we can expect on the solution process; nevertheless, the drawback is that we  have to assume higher regularity on the drift function $b$ to be able to invert it. 
Therefore, the choice of the  decay rate~$\alpha$ of $f$ is crucial to obtain well-posedness for classes of drift functions of low regularity. 

\begin{defin}
\label{def:order_alpha}
We say that $f:\R \to \R$ is of \textit{order $\alpha>0$} if there exist two constants $C$ and $c >0$ such that 
$\frac{c}{\crochetk^\alpha} \leq f(k) \leq  \frac{C}{\crochetk^\alpha}$ for every $k \in \R$. Recall that $\crochetk:= (1+k^2)^{1/2}$. 
\end{defin}

In order to make clear this relation between regularity of $b$ and decay rate of $f$, we will  prove in Paragraph~\ref{parag:simple_case}  well-posedness for equation~\eqref{eq:definition_FP_perturbee} in a simplified case: we will assume that the mass is constant, namely that $\varphi \equiv 1$, and that for each $\mu \in \mathcal P_2(\R)$, $b(\cdot, \mu)$ belongs to a Sobolev space with a Sobolev norm uniform in $\mu$. In Paragraph~\ref{parag:general_case}, we will then give a general statement for more general functions $\varphi$ and $b$, but the idea of proof is the same up to technicalities.

\subsubsection{Simple case with constant mass and bounded drift function}
\label{parag:simple_case}

Let us assume in this paragraph that $\varphi$ is the constant function equal to one. In other words, we are studying the following equation:
\begin{align}
\label{eq:constant_mass}
\left\{ 
\begin{aligned}
dy_t(u)&=  b(y_t(u),\mu_t) \mathrm dt +\int_\R f(k) \Re \left(  e^{-ik y_t(u)}  \mathrm dw(k,t)\right),\\
\mu_t&= \Leb_{[0,1]} \circ y_t^{-1},
\\
y_0&=g,
\end{aligned}
\right.
\end{align}

Let us fix $\alpha, \eta>0$ and let us assume that $f:\R \to \R$ is of order $\alpha$, according to Definition~\ref{def:order_alpha}. 
Let $b: \R \times \mathcal P_2(\R) \to \R$ be a measurable function such that for each $\mu \in \mathcal P_2(\R)$, the map $x \mapsto b(x,\mu) $ belongs to the Sobolev space $H^\eta(\R)$ uniformly in $\mu$, that is there is a constant $C$ such that for every $\mu \in \mathcal P_2(\R)$, $\| b(\cdot, \mu) \|_{H^\eta} \leq C$, where 
\begin{align*}
\| \phi \|_{H^\eta} := \| k \mapsto \crochetk^\eta \mathcal F \phi (k) \|_{L_2}= \left( \int_\R (1+k^2)^\eta |\mathcal F \phi (k)|^2  \mathrm dk\right)^{1/2}.
\end{align*}
We also denote by $B: \R \times L_2[0,1] \to \R$ the function  $B(x,z):=b(x,\Leb_{[0,1]} \circ z^{-1})$. Of course,  for every $z \in L_2[0,1]$, inequality $\| B(\cdot, z) \|_{H^\eta} \leq C$ also holds with the same constant $C$ as above. 

The following lemma is the key step in order to apply a Girsanov transformation in equation~\eqref{eq:constant_mass}. Let us fix $f$, $B$ and a $(\mathcal G_t)_{t \in [0,T]}$-adapted process $(x_t)_{t\in [0,T]}$ with values in $\mathcal C([0,1], \R)$. Then we are looking for 
an  $L_2(\R, \C)$-valued   $(\mathcal G_t)_{t \in [0,T]}$-adapted process $(h_t)_{t\in[0,T]}=(h^\Re_t+i h^\Im_t)_{t\in[0,T]}$ such that for every $ t\in[0,T]$ and for every $u \in [0,1]$
\begin{align}
\label{eq:a inverser}
B(x_t(u),x_t) =   \int_\R e^{-ik x_t(u)}  f(k)  h_t(k) \mathrm dk,
\end{align}
or equivalently, taking the real part of~\eqref{eq:a inverser} (and using that $B$ and $f$ are real-valued), to find two $L_2(\R, \R)$-valued $(\mathcal G_t)_{t \in [0,T]}$-adapted processes $(h^\Re_t)_{t\in[0,T]}$ and $(h^\Im_t)_{t\in[0,T]}$ such that
\begin{align*}
B(x_t(u),x_t) =   \int_\R  {\cos(k x_t(u))f(k)}  h^\Re_t(k) \mathrm dk +\int_\R  {\sin(k x_t(u))f(k)}  h^\Im_t(k) \mathrm dk. 
\end{align*}

\begin{lemme}
\label{lemme:fourier_masseconstante}
Let $\alpha,\eta>0$. Let $f$ be  of order $\alpha$ and $b: \R \times \mathcal P_2(\R) \to \R$ be a measurable function such that for each $\mu \in \mathcal P_2(\R)$, $b(\cdot,\mu)$ belongs  to $H^\eta(\R)$ with a uniform $H^\eta$-norm. Let $(x_t)_{t\in [0,T]}$ be a  $(\mathcal G_t)_{t \in [0,T]}$-adapted process taking values in $\mathcal C([0,1], \R)$.  

If $\eta \geq \alpha$, then
there is a $(\mathcal G_t)_{t\in [0,T]}$-adapted process $(h_t)_{t\in [0,T]}$ which is solution, for every $t\in [0,T]$, to equation~\eqref{eq:a inverser} and such that   there exists $C>0$ depending only on $b$ and  $f$ for which  $\int_0^T\int_\R |h_t(k)|^2 \mathrm dk \mathrm dt \leq C$ holds almost surely.
\end{lemme}

\begin{proof}
By the substitution $y=x_t(u)$, equation~\eqref{eq:a inverser} is equivalent to
\begin{align}
\label{apres chgt de var}
B(y,x_t) =\int_\R e^{-ik y}  f(k)  h_t(k) \mathrm dk
\end{align}
for every $y\in \operatorname{Im}(x_t)$. In particular, if a process $(h_t)_{t\in [0,T]}$ satisfies~\eqref{apres chgt de var} for every $y \in \R$, then  it satisfies~\eqref{eq:a inverser} for every $u \in [0,1]$. 
Computing the Fourier transform on each side of equation~\eqref{apres chgt de var}, we have
$\mathcal F (f h_t) = \frac{1}{\sqrt{2\pi}}  B(\cdot,x_t)$. 
Therefore, the process defined by 
\begin{align}
h_t(k):=\frac{1}{\sqrt{2\pi}f(k)} \mathcal F^{-1} \left( B(\cdot,x_t)  \right) (k)
\label{ht formule d'inversion}
\end{align}
is solution to equation~\eqref{eq:a inverser}, provided that $h_t$ is square integrable for every $t\in [0,T]$. 
Let us compute the $L_2$-norm of $h$: there are $C_1$ and $C_2$ such that
\begin{align*}
\int_0^T \!\! \int_\R |h_t(k)|^2 \mathrm dk \mathrm dt 
&\leq  C_1 \int_0^T\!\! \int_\R  \crochetk^{2\alpha} |\mathcal F^{-1} \left( B(\cdot,x_t)  \right) (k)|^2 \mathrm dk\mathrm dt \\
&\leq  C_1 \int_0^T\!\! \int_\R  \crochetk^{2\eta} |\mathcal F^{-1} \left( B(\cdot,x_t)  \right) (k)|^2 \mathrm dk\mathrm dt \\
&= C_1 \int_0^T\!\! \int_\R  \crochetk^{2\eta} |\mathcal F \left( B(\cdot,x_t)  \right) (k)|^2 \mathrm dk\mathrm dt = C_1 \int_0^T \| B(\cdot,x_t) \|_{H^\eta}\mathrm dt \leq C_2,
\end{align*}
where we used the fact that $f$ is of order $\alpha$, that $\alpha \leq \eta$ and that $\mathcal F^{-1} (\phi) (\cdot) = \mathcal F (\phi) (-\; \cdot)$ for each $\phi \in L_2(\R)$. 
\end{proof}

Let us give,  in accordance with~\cite{karatzasshreve91},   the following sense to a weak solution to  SDE~\eqref{eq:constant_mass}. 

\begin{defin}
\label{def:weak solution}
A sextuple $(\Omega, \mathcal G, (\mathcal G_t)_{t\in [0,T]},\mathbb P,  z,w)$ is said to be a \emph{weak solution} to equation~\eqref{eq:constant_mass} if
\begin{itemize}
\item[-] $(\Omega, \mathcal G, (\mathcal G_t)_{t\in [0,T]} ,\mathbb P)$ is a filtered probability space satisfying usual conditions,
\item[-] $(z_t)_{t\in [0,T]}$ is a continuous $(\mathcal G_t)_{t\in [0,T]}$-adapted $ \mathcal C[0,1]$-valued process, 
\item[-] $w=(\wR,\wI)$, where $(\wR(k,t))_{k\in \R,t\in [0,T]}$ and $(\wI(k,t))_{k\in \R,t\in [0,T]}$ are two independent $(\mathcal G_t)_{t\in [0,T]}$-Brownian sheets under $\mathbb P$, 
\item[-] $\mathbb P$-almost surely, for every $t\in [0,T]$, 
\begin{align}
\label{eq:intégrée_z}
z_t(u)&=g(u)+\int_0^t \! \int_\R \cos( kz_s (u)) f(k) \mathrm d\wR(k,s) \notag
\\&\quad +\int_0^t \! \int_\R \sin( kz_s (u)) f(k) \mathrm d\wI(k,s)
 +\int_0^t  B(z_s(u),z_s) \mathrm ds,
\end{align}
where $B(x,z):=b(x,\Leb_{[0,1]} \circ z^{-1})$.
\end{itemize}
\end{defin}

\begin{theo}
Let $g \in \mathbf G^1$.
Let $f$ be  of order $\alpha> \frac{3}{2}$ and $b: \R \times \mathcal P_2(\R) \to \R$ be a measurable function such that for each $\mu \in \mathcal P_2(\R)$, $b(\cdot,\mu)$ belongs  to $H^\eta(\R)$ with a uniform $H^\eta$-norm.
If $\eta \geq \alpha$, there exists a unique weak solution to equation~\eqref{eq:constant_mass}.

Moreover, if $(\Omega^i, \mathcal G^i, (\mathcal G^i_t)_{t\in [0,T]}, \mathbb P^i, z^i,w^i)$, $i=1,2$, are two weak solutions to equation~\eqref{eq:constant_mass}, then the laws of $(z^1,w^1)$ and $(z^2,w^2)$ are equal in $\mathcal C([0,1] \times [0,T])\times \mathcal C(\R \times [0,T],\R^2)$.
\label{theo:weak_wp_hypotheses_simplifiees}
\end{theo}

Let us remark that the Brownian sheets $w^1$ and $w^2$ are seen here as taking values in $\R^2$, by an identification of $\R^2$ with $\C$.

\begin{proof}[Proof (Theorem~\ref{theo:weak_wp_hypotheses_simplifiees}, existence part)]
Let $(\Omega, \mathcal G,(\mathcal G_t)_{t\in [0,T]}, \mathbb P )$ be a filtered probability space and $(w(u,t))_{u \in [0,1],t\in [0,T]}$ be a $(\mathcal G_t)_{t\in [0,T]}$-Brownian sheet. 
Since $\alpha > \frac{3}{2}$, the map $k \mapsto \crochetk f(k)$ is square integrable. 
Let us consider equation~\eqref{eq:real equation on y} with $\varphi \equiv 1$ or equivalently equation~\eqref{eq:intégrée_z} with $B\equiv 0$: 
\begin{align}
\label{sans la masse}
y_t(u)=g(u)+\int_0^t \! \int_\R \cos( k y_s (u)) f(k) \mathrm d\wR(k,s)
+\int_0^t \! \int_\R \sin( k y_s (u)) f(k) \mathrm d\wI(k,s).
\end{align}
By Proposition~\ref{prop:exist and uniq of eq1}, there is a unique process $(y_t)_{t\in [0,T]}$ satisfying equation~\eqref{sans la masse}  for every $u \in [0,1]$. Moreover, $(y_t)_{t\in [0,T]}$ is a  $(\mathcal G_t)_{t \in [0,T]}$-adapted process taking values in $\mathcal C([0,1], \R)$.

Therefore, by Lemma~\ref{lemme:fourier_masseconstante}, there is a process $(h_t)_{t \in[0,T]}=(h^\Re_t+i h^\Im_t)_{t \in[0,T]}$ with values in $L_2(\R,\C)$ satisfying for every $t \in [0,T]$ and $u \in [0,1]$:
\begin{align*}
B(y_t(u),y_t) 
&=   \int_\R  \cos(k y_t(u))f(k)  h^\Re_t(k) \mathrm dk \\
&\quad+\int_\R  \sin(k y_t(u))f(k) h^\Im_t(k) \mathrm dk
\end{align*}
and such that there exists a constant $C$ such that almost surely, 
\begin{align}
\label{estimée_ht}
\int_0^T \!\!\int_\R |h_t(k)|^2 \mathrm dk \mathrm dt \leq C.
\end{align}
Therefore, we can rewrite equation~\eqref{sans la masse} as follows:
\begin{align*}
y_t(u)
&=g(u)+ \int_0^t \!\! \int_\R \cos(k y_s(u)) f(k) \mathrm d\wR(k,s)  
+\int_0^t \!\! \int_\R \sin(k y_s(u)) f(k) \mathrm d\wI(k,s)  \notag \\
&\quad +\int_0^t \!\! B(y_s(u),y_s)\mathrm ds  - \int_0^t \!\! \int_\R  \cos(k y_s(u))f(k)   h^\Re_s(k) \mathrm dk 
- \int_0^t \!\! \int_\R  \sin(k y_s(u))f(k)   h^\Im_s(k) \mathrm dk \notag \\
&=g(u)+ \int_0^t \!\! \int_\R \cos(k y_s(u)) f(k) \mathrm d\widetilde{w}^\Re(k,s)  
+\int_0^t \!\! \int_\R \sin(k y_s(u)) f(k) \mathrm d\widetilde{w}^\Im(k,s)  \notag \\
&\quad +\int_0^t \!\! B(y_s(u),y_s)\mathrm ds,
\end{align*}
where we define for every $k \in \R$ and for every $s\in [0,T]$
\begin{align*}
\mathrm d\widetilde{w}^\Re(k,s) &:=\mathrm d\wR(k,s)  - h^\Re_s(k) \mathrm dk  \mathrm ds, \\
\mathrm d\widetilde{w}^\Im(k,s) &:=\mathrm d\wI(k,s)  - h^\Im_s(k) \mathrm dk  \mathrm ds.
\end{align*}
Let us consider the process $(G_t)_{t \in [0,T]}$ defined by:
\begin{align*}
G_t:= \exp \left( \int_0^t\!\! \int_\R  h^\Re_s(k) \mathrm d\wR(k,s) +  \int_0^t \!\! \int_\R  h^\Im_s(k) \mathrm d\wI(k,s) - \frac{1}{2} \int_0^t \!\! \int_\R |h_s(k)|^2 \mathrm dk \mathrm ds \right). 
\end{align*}
By~\eqref{estimée_ht}, there is $C>0$ such that $\exp(\frac{1}{2} \int_0^t  \int_\R |h_s(k)|^2 \mathrm dk \mathrm ds) \leq C$ almost surely. Thus Novikov's condition holds and the process $(G_t)_{t \in [0,T]}$ is a $\mathbb P$-martingale.
Let us define the probability measure $\mathbb Q$ by the absolutely continuous measure with respect to~$\mathbb P$ with density $\frac{\mathrm d\mathbb Q}{\mathrm d\mathbb P}=G_T$. 
By Girsanov's Theorem, under the probability measure $\mathbb Q$, $(\widetilde{w}^\Re(k,t),\widetilde{w}^\Im(k,t))_{k\in \R, t \in [0,T]}$ are two independent Brownian sheets on $\R \times [0,T]$ and the couple $(y,\widetilde{w})$ satisfies equation~\eqref{eq:intégrée_z}. 
Thus $(\Omega,\mathcal G, (\mathcal G_t)_{t\in [0,T]},\mathbb Q, y, \widetilde{w})$ is a weak solution of equation~\eqref{eq:constant_mass}. 
\end{proof}

Let us start by proving the uniqueness part of Theorem~\ref{theo:weak_wp_hypotheses_simplifiees} in the case where the drift function $b \equiv 0$ in equation~\eqref{eq:constant_mass}, namely in the case of equation~\eqref{sans la masse}.

\begin{lemme}
\label{lemme:cas_sans_drift}
Let us assume that  $(\Omega^i, \mathcal G^i, (\mathcal G^i_t)_{t\in [0,T]}, \mathbb P^i, z^i,w^i)$, $i=1,2$, are two weak solutions to equation~\eqref{sans la masse}. Then $(z^1,w^1)$ and $(z^2,w^2)$ have same law in $\mathcal C([0,1] \times [0,T])\times \mathcal C(\R \times [0,T],\R^2)$.  
\end{lemme}

\begin{proof}
Recall that by Proposition~\ref{prop:exist and uniq of eq1}, equation~\eqref{sans la masse} has a unique pathwise solution. By an infinite-dimensional version of Yamada-Watanabe result  (see~\cite[Prop 5.3.20]{karatzasshreve91}), it implies that the law of $(z^1,w^1)$ under $\mathbb P^1$ is equal to the law of $(z^2,w^2)$ under $\mathbb P^2$.
\end{proof}

The proof of uniqueness in law for equation~\eqref{eq:constant_mass} is based on Girsanov's Theorem. As in the proof of the existence part, we will apply Lemma~\ref{lemme:fourier_masseconstante} to the drift function $B$ and to a weak solution to equation~\eqref{eq:constant_mass}.

\begin{proof}[Proof (Theorem~\ref{theo:weak_wp_hypotheses_simplifiees}, uniqueness part)]
Let us consider  $(\Omega^i, \mathcal G^i, (\mathcal G^i_t)_{t\in [0,T]}, \mathbb P^i, z^i,w^i)$, for $i=1,2$,  two weak solutions to equation~\eqref{eq:constant_mass}. 
Let $i=1$ or $2$. In particular, $(z^i_t)_{t \in [0,T]}$ is a   $(\mathcal G^i_t)_{t \in [0,T]}$-adapted process taking values in $\mathcal C([0,1], \R)$.  
Thus by Lemma~\ref{lemme:fourier_masseconstante},  there is a $(\mathcal G^i_t)_{t\in [0,T]}$-adapted process $(h^i_t)_{t\in[0,T]}$ such that  
$\int_0^T \int_\R |h^i_t(k)|^2 \mathrm dk \mathrm ds \leq C$ almost surely and for all $t\in[0,T]$ and $u\in [0,1]$, 
\begin{align*}
B(z^i_t(u),z^i_t) 
&=   \int_\R  \cos(k z^i_t(u))f(k) h^{i,\Re}_t(k) \mathrm dk +\int_\R  \sin(k z^i_t(u))f(k) h^{i,\Im}_t(k) \mathrm dk.
\end{align*}
Furthermore, by equation~\eqref{ht formule d'inversion}, there is a measurable map $\mathcal H: \mathcal C[0,1] \to L_2(\R, \C)$ such that $h_t^i= \mathcal H(z_t^i)$ for every $t\in [0,T]$ and for $i=1,2$. The map $\mathcal H$ is defined by:
\begin{align*}
\mathcal H(\mathbf x): k \mapsto  \frac{1}{\sqrt{2\pi} f(k)} \mathcal F^{-1} (B(\cdot, \mathbf x)) (k)
\end{align*}
for every $\mathbf x \in \mathcal C[0,1]$.

Since $(z^i_t)_{t\in [0,T]}$ is solution to equation~\eqref{eq:intégrée_z}, we have $\mathbb P^i$-almost surely for every $t\in [0,T]$ and for every $u\in [0,1]$:
\begin{align*}
z^i_t(u)&=g(u)+\int_0^t \int_\R \cos( kz^i_s (u)) f(k) (\mathrm dw^{i,\Re}(k,s) + h^{i,\Re}_s(k) \mathrm dk\mathrm ds) 
\\&\quad +\int_0^t  \int_\R \sin( kz^i_s (u)) f(k) (\mathrm dw^{i,\Im}(k,s) + h^{i,\Im}_s(k) \mathrm dk\mathrm ds).
\end{align*}
Let us define for every $k\in \R$ and every $s\in [0,T]$
\begin{align*}
\mathrm d\widetilde{w}^{i,\Re}(k,s) &:=\mathrm dw^{i,\Re}(k,s)  + h^{i,\Re}_s(k) \mathrm dk  \mathrm ds, \\
\mathrm d\widetilde{w}^{i,\Im}(k,s) &:=\mathrm dw^{i,\Im}(k,s)  + h^{i,\Im}_s(k) \mathrm dk  \mathrm ds.
\end{align*}
Let us consider the process $(G^i_t)_{t \in [0,T]}$ defined by:
\begin{align*}
G^i_t:= \exp  \left(- \int_0^t \!\! \int_\R h^{i,\Re}_s(k) \mathrm dw^{i,\Re}(k,s) - \int_0^t \!\! \int_\R  h^{i,\Im}_s(k) \mathrm dw^{i,\Im}(k,s) - \frac{1}{2} \int_0^t \!\! \int_\R | h^i_s(k)|^2 \mathrm dk \mathrm ds \right). 
\end{align*} 
Novikov's condition applies because $\int_0^T \int_\R |h^i_t(k)|^2 \mathrm dk \mathrm ds \leq C$ almost surely and  the process $(G^i_t)_{t\in [0,T]}$ is a $\mathbb P^i$-martingale.
We define the probability measure $\mathbb Q^i$ by the absolutely continuous measure with respect to $\mathbb P^i$ with density $\frac{\mathrm d\mathbb Q^i}{\mathrm d\mathbb P^i}=G^i_T$.
By Girsanov's Theorem, under $\mathbb Q^i$,  
 $\widetilde{w}^i=(\widetilde{w}^{i,\Re},\widetilde{w}^{i,\Im})$ is a couple of two independent Brownian sheets  and $\mathbb Q^i$-almost surely, for every $t\in [0,T]$, 
\begin{align*}
z^i_t(u)&=g(u)+\int_0^t  \!\!  \int_\R f(k) \left( \cos( kz^i_s (u))  \mathrm d\widetilde{w}^{i,\Re}(k,s) + \sin( kz^i_s (u))  \mathrm d\widetilde{w}^{i,\Im}(k,s)\right).  
\end{align*}
Thus $(\Omega^i, \mathcal G^i, (\mathcal G^i_t)_{t\in [0,T]}, \mathbb Q^i, z^i,\widetilde{w}^i)$,  for $i=1,2$,  are two weak solutions to equation~\eqref{eq:constant_mass} in the case where $B\equiv 0$. 
By Lemma~\ref{lemme:cas_sans_drift}, it follows that for every measurable function $\psi : \mathcal C([0,1] \times [0,T])\times \mathcal C(\R \times [0,T],\R^2) \to \R$ such that $\mathbb E^{\mathbb Q^i} \left[ |\psi(z^i,\widetilde{w}^i)| \right]<+\infty$ for $i=1,2$, we have
\begin{align}
\mathbb E^{\mathbb Q^1} \left[ \psi(z^1,\widetilde{w}^1) \right]
=\mathbb E^{\mathbb Q^2} \left[ \psi(z^2,\widetilde{w}^2) \right].
\label{psi_equality}
\end{align}

Let $\phi: \mathcal C([0,1] \times [0,T])\times \mathcal C(\R \times [0,T],\R^2) \to \R$ be a bounded and measurable function. 
We have
\begin{align}
\mathbb E^{\mathbb P^i}\left[ \phi(z^i,w^i) \right]
&= \mathbb E^{\mathbb Q^i}\left[ \phi(z^i,w^i) (G_T^i)^{-1}\right] \notag\\
&= \mathbb E^{\mathbb Q^i}\Bigg[ \phi(z^i,w^i)  \exp \Bigg( \int_0^T \!\! \int_\R  h^{i,\Re}_s(k) \mathrm dw^{i,\Re}(k,s) \notag \\ 
&\quad\quad\quad\quad\quad\quad\quad\quad\quad\quad
+\int_0^T \!\! \int_\R  h^{i,\Im}_s(k) \mathrm dw^{i,\Im}(k,s)
+\frac{1}{2} \int_0^T \!\! \int_\R | h^i_s(k)|^2 \mathrm dk \mathrm ds \Bigg)\Bigg]\notag
\\
&= \mathbb E^{\mathbb Q^i}\Bigg[ \phi(z^i,\widetilde{w}^i +\textstyle \int_0^\cdot \int_0^\cdot  h_s^i(k) \mathrm dk \mathrm ds ) \displaystyle \exp \Bigg( \int_0^T \!\! \int_\R  h^{i,\Re}_s(k) \mathrm d\widetilde{w}^{i,\Re}(k,s)\notag \\ 
&\quad\quad\quad\quad\quad\quad\quad\quad\quad\quad
+ \int_0^T \!\! \int_\R  h^{i,\Im}_s(k) \mathrm d\widetilde{w}^{i,\Im}(k,s)
-\frac{1}{2} \int_0^T \!\! \int_\R | h^i_s(k)|^2 \mathrm dk \mathrm ds \Bigg)\Bigg]\notag\\
&=\mathbb E^{\mathbb Q^i} \left[ \psi(z^i,\widetilde{w}^i) \right],
\label{loi_de_zi_wi}
\end{align}
where $\psi :\mathcal C([0,1] \times [0,T])\times \mathcal C(\R \times [0,T],\R^2) \to \R$ is a measurable function, because for each $t\in [0,T]$, $h_t^i= \mathcal H(z_t^i)$ with $\mathcal H: \mathcal C[0,1] \to L_2(\R, \C)$ a measurable function. By equality~\eqref{psi_equality}, we deduce that $\mathbb E^{\mathbb P^1}\left[ \phi(z^1,w^1) \right]=\mathbb E^{\mathbb P^2}\left[ \phi(z^2,w^2) \right]$. 
Thus $(z^1,w^1)$ and $(z^2,w^2)$ have the same law and this completes the proof of the theorem. 
\end{proof}

\subsubsection{General case}
\label{parag:general_case}

In the previous paragraph, our assumptions on $b$ were rather restrictive: for instance, the inversion statement of Lemma~\ref{lemme:fourier_masseconstante} does not apply for $b\equiv 1$ because it does not belong to $H^\eta(\R)$ for any positive $\eta$. 
In this paragraph, we  explain briefly how we can extend the well-posedness result for a larger class of  drift functions $b$  or general mass functions $\varphi$. 
Because the proofs are very similar to the particular case seen above, the  statements of this paragraph will be explained shortly  without the detailled proofs: the interest reader can find the complete proofs of the results stated below in~\cite[Parag. II.4]{marx_thesis}. 

Let us recall that we consider $\varphi: \R \to \R$ an even $\mathcal C^\infty$-function, such that $\varphi$ is positive and decreasing on $[0,+\infty)$. 
For every fixed $M >0$, we define the following assumptions:
\begin{defin}
A process $(x_t)_{t \in [0,T]}$ with values in $\mathcal C[0,1]$ is said to satisfy the $X_M$-hypotheses if:
\begin{itemize}
\item[$(X1)$] $(x_t)_{t \in [0,T]}$ is $(\mathcal G_t)_{t \in [0,T]}$-adapted.
\item[$(X2)$] almost surely, for every $t \in [0,T]$, $u \mapsto x_t(u)$ is strictly increasing. 
\item[$(X3)$] almost surely, for every $t \in [0,T]$, $|x_t(1)-x_t(0)|\leq M$. 
\end{itemize}
\label{def:x hyp}
\end{defin}
As a consequence of Corollary~\ref{coro:growth} and of Proposition~\ref{prop:exist and uniq of eq1}, the stopped process $(y_{t\wedge \tau_M})_{t\in [0,T]}$  solution to the  equation without drift function    satisfies the assumptions of Definition~\ref{def:x hyp}: 
\begin{prop}
Let $g\in \mathbf G^1$. Assume that $f$ is of order $\alpha>\frac{3}{2}$.  Let $(y_t)_{t\in [0,T]}$ be the unique solution  to equation~\eqref{eq:real equation on y} given by Proposition~\ref{prop:exist and uniq of eq1}. 
Let $M > g(1)-g(0)$ and recall the definition of $\tau_M:=\inf \{t \geq 0: y_t(1)-y_t(0) \geq M\} \wedge T$.
Then $(y_{t\wedge \tau_M})_{t\in [0,T]}$ satisfies the $X_M$-hypotheses. 
\label{prop:y verifie hyp xm}
\end{prop}

Under those less restrictive assumptions on $\varphi$, $B$ and $x$, the following lemma shows the existence of an $L_2(\R,\C)$-valued process $(h_t)_{t\in[0,T]}=(h^\Re_t+i h^\Im_t)_{t\in[0,T]}$ such that 
\begin{align}
\label{eq:a_inverser avec masse}
B(x_t(u),x_t) = \frac{1}{\big(\int_0^1 \varphi(x_t(u)-x_t(v))\mathrm dv\big)^{1/2}}  \int_\R e^{-ik x_t(u)}  f(k)  h_t(k) \mathrm dk.
\end{align}

\begin{lemme}
\label{lemme:inversion Fourier}
Let $M > g(1)-g(0)$, $j \in \N$ and $\alpha>0$. Let us assume that $f$ is of order $\alpha$, that $B:\R \times L_2[0,1] \to \R$ satisfies the $B$-hypotheses of order $2j$ and thatthe process $(x_t)_{t\in [0,T]}$ satisfies the $X_M$-hypotheses.

If $2j \geq \alpha$, then there is a $(\mathcal G_t)_{t\in [0,T]}$-adapted process $(h_t)_{t\in [0,T]}$ which is solution, for every $t\in [0,T]$, to equation~\eqref{eq:a_inverser avec masse} and such that   there exists $C_M>0$ depending only on $B$, $f$, $\varphi$ and~$M$ for which  $\int_0^T\int_\R |h_t(k)|^2 \mathrm dk \mathrm dt \leq C_M$ holds almost surely.
\end{lemme}

\begin{proof}
Let $(x_t)_{t \in [0,T]}$ be a process satisfying the $X_M$-hypotheses. Therefore, for a fixed $t \in [0,T]$, the map $u \mapsto x_t(u)$ is a continuous strictly increasing function and can be seen as the quantile function of a measure $\mu_t \in \mathcal P_2(\R)$. Let us denote by $F_t$ and $p_t$ respectively the c.d.f. and the density associated to $\mu_t$. 
More precisely, $F_t(x_t(u))=u$ for all $u \in [0,1]$, $F_t(y)=0$ for all $y \leq x_t(0)$ and $F_t(y)=1$ for all $y \geq x_t(1)$. Since almost surely, for every $t\in [0,T]$, $x_t(1)-x_t(0) \leq M$, we have $|\operatorname{Supp} p_t|\leq M$, where $|\operatorname{Supp} p_t|$ denotes the Lebesgue measure of the support of $p_t$. 

By the substitution $y=x_t(u)$, equation~\eqref{eq:a_inverser avec masse} is equivalent to
\begin{align}
\label{eq:apres chgt de var}
B(y,x_t) \left( \int_0^1 \varphi(y-x_t(v)) \mathrm dv \right)^{1/2}=\int_\R e^{-ik y}  f(k)  h_t(k) \mathrm dk
\end{align}
for every $y\in [x_t(0),x_t(1)]$. 
Let us fix a $\mathcal C^\infty$-function $\Psi: \R \to \R$ that is equal to $0$ on $(-\infty,0]$ and equal to $1$ on $[1,+\infty)$. 
For every $a<b$, we  define the cut-off function $\eta_{a,b}: \R \to \R$ by
\begin{align*}
\eta_{a,b}(y)=
\begin{cases}
1 &\text{ on } [a,b], \\
\Psi (y-(a-1)) &\text{ on } (a-1,a),\\
\Psi (b+1-y) &\text{ on } (b,b+1).\\
0 &\text{ elsewhere. }\\
\end{cases}
\end{align*}
Let us denote by $\eta_t:=\eta_{x_t(0),x_t(1)}$. For every $y \in [x_t(0),x_t(1)]$, $\eta_t(y)=1$. Moreover, $\eta_t$ has a compact support included in $[x_t(0)-1,x_t(1)+1]$. 
Therefore, if a process $(h_t)_{t\in [0,T]}$ satisfies
\begin{align*}
B(y,x_t)  \eta_t(y) \left(\int_0^1 \varphi(y-x_t(v)) \mathrm dv\right)^{1/2} =\int_\R e^{-ik y}  f(k)  h_t(k) \mathrm dk
\end{align*}
for every $y \in \R$, then it satisfies~\eqref{eq:apres chgt de var} for every $y \in [x_t(0),x_t(1)]$ and thus it satisfies~\eqref{eq:a_inverser avec masse} for every $u \in [0,1]$. 
Therefore, the process defined by 
\begin{align*}
h_t(k):=\frac{1}{\sqrt{2\pi}f(k)} \mathcal F^{-1} \left( B(\cdot,x_t) \eta_t \left(\int_0^1 \varphi(\cdot-x_t(v)) \mathrm dv \right)^{1/2} \right) (k)
\end{align*}
is solution to equation~\eqref{eq:a_inverser avec masse}, provided that $h_t$ is square integrable for every $t\in [0,T]$. Define 
$\Phi_t:=B(\cdot,x_t) \eta_t \left(\int_0^1 \varphi(\cdot-x_t(v)) \mathrm dv\right)^{1/2}$. 
  Note that for every $t\in [0,T]$,   $y\mapsto \eta_t(y) (\int_0^1 \varphi(y-x_t(v)) \mathrm dv)^{1/2}$ is a bounded $\mathcal C^\infty$-function with compact support and $y\mapsto B(y,x_t)$ is a bounded continuous function. Therefore, $ \Phi_t$ belongs to $L^1(\R, \C)$ and $h_t$ is well-defined. Moreover, since $( \Phi_t)_{t\in [0,T]}$ is $(\mathcal G_t)_{t\in [0,T]}$-adapted, $(h_t)_{t\in [0,T]}$ is also $(\mathcal G_t)_{t\in [0,T]}$-adapted. 

Furthermore, we know by assumption that there is $c>0$ such that for every $k \in \R$,  $\frac{1}{f(k)}\leq c \crochetk^\alpha \leq c \crochetk^{2j}$. Thus  by Plancherel's Theorem (and denoting by $\Delta$ the Laplacian) we have
\begin{align*}
\|h_t\|_{L_2}^2=\int_\R |h_t(k)|^2 \mathrm dk 
=\int_\R \frac{1}{2\pi}  \frac{|\mathcal F^{-1} \Phi_t (k)|^2}{|f(k)|^2} \mathrm dk
&\leq C \int_\R \left| \crochetk^{2j} \mathcal F^{-1} \Phi_t (k)  \right|^2 \mathrm dk \\
&= C \int_\R \left|  \mathcal F^{-1}((1+\Delta)^j \Phi_t) (k)  \right|^2 \mathrm dk \\
&= C \left\| (1+\Delta)^j \Phi_t \right\|_{L_2}^2.
\end{align*}
On the one hand,  $B$ satisfies the $B$-hypotheses of order $2j$, then for every  $i \in \{0,1,\dots,2j\}$, for every $t \in [0,T]$ and for every $y \in \R$, 
$|\partial_1^{(i)} B(y,x_t)| \leq C_i(M)$.   
On the other hand,  $y\mapsto \eta_t(y) (\int_0^1 \varphi(y-x_t(v)) \mathrm dv)^{1/2}$ is a $\mathcal C^\infty$-function with compact support, thus this function and all its derivatives are bounded on $\R$. We deduce that for every $i \leq 2j$,
there is a constant $C_i$ depending on $B$ and $\varphi$ such that almost surely, for every $t\in [0,T]$,  $\|\partial^i \Phi_t\|_{L_\infty} \leq C_i$.

Recall that the support of $\eta_t$ is included in $[x_t(0)-1,x_t(1)+1]$. Henceforth, almost surely for every $t\in [0,T]$, the Lebesgue measure of the support of $\Phi_t$ is bounded by $M+2$. Therefore, for every $i \leq 2j$, there is a constant $C_{i,M}$ such that almost surely, for every $t\in [0,T]$, 
\begin{align*}
\|\partial^i \Phi_t\|_{L_2} \leq |\operatorname{Supp} \Phi_t|^{1/2} \|\partial^i \Phi_t\|_{L_\infty} \leq C_{i,M}. 
\end{align*}
We deduce that there is $C_M>0$ such that $\int_0^T \|h_t\|_{L_2}^2 \mathrm dt \leq C_M$ almost surely, which completes the proof. 
\end{proof}

Thus, we can state the following theorem, which is a rewriting of Theorem~\ref{theointro_partie1} under the precise assumptions on~$b$. 
\begin{theo}
\label{theo:wp_general_case}
Let $g \in \mathbf G^1$ and $j \in \Nun$. 
Let $f$ be a function of order $\alpha>\frac{3}{2}$.
Let $B: \R \times L_2[0,1] \to \R$ satisfy the $B$-hypotheses of order $2j$.
If $2j \geq \alpha$, there exists a weak solution to equation~\eqref{eq:definition_FP_perturbee} and uniqueness in law holds for this equation. 
\end{theo}

Let us briefly explain the different steps of the proof of Theorem~\ref{theo:wp_general_case}, the detailled proof being given in~\cite[Parag. II.4]{marx_thesis}. 

\textit{Step 1.}
Let us fix $M \geq g(1)-g(0)$ and let us define the stopped version of equation~\eqref{eq:definition_FP_perturbee}:
\begin{align}
\label{eq:stoppee}
\left\{ 
\begin{aligned}
dy_t(u)&=  \mathds 1_{\{t \leq \tau_M\}}\left(B(y_t(u),y_t) \mathrm dt +\frac{1}{m_t(u)^{1/2}} \int_\R f(k) \Re \left(  e^{-ik y_t(u)}  \mathrm dw(k,t)\right)\right),\\
y_0&=g,
\end{aligned}
\right.
\end{align}
where $m_t(u)= \int_0^1 \varphi(y_t(u)-y_t(v))\mathrm dv$ and $\tau_M:=\inf\{t \geq 0: y_t(1)-y_t(0) \geq M\} \wedge T$. 
We start by proving, for $f$ of order $\alpha > \frac{3}{2}$ and for $B$ satisfying the $B$-hypotheses of order $2j \geq \alpha$, the existence of a weak solution to~\eqref{eq:stoppee}. The scheme of proof is the following: since $\alpha> \frac{3}{2}$, we know by Proposition~\ref{prop:y verifie hyp xm} that $(y_{t\wedge \tau_M})_{t\in [0,T]}$ satisfies the $X_M$-hypotheses. Then, by  Lemma~\ref{lemme:inversion Fourier}, there exists an appropriate process $(h_t)_{t\in[0,T]}$ such that $\int_0^T \int_\R |h_t(k) |^2 \mathrm dk \mathrm dt \leq C_M$. Then the proof is the same as for Theorem~\ref{theo:weak_wp_hypotheses_simplifiees}. 

\textit{Step 2.}
By analogy with Theorem~\ref{theo:weak_wp_hypotheses_simplifiees}, we prove that 
if $(\Omega^i, \mathcal G^i, (\mathcal G^i_t)_{t\in [0,T]}, \mathbb P^i, z^i,w^i)$, $i=1,2$, are two weak solutions to equation~\eqref{eq:stoppee}, then the laws of $(z^1,w^1)$ and $(z^2,w^2)$ are equal in $\mathcal C([0,1] \times [0,T])\times \mathcal C(\R \times [0,T],\R^2)$.
First, we observe that this statement is true for $B\equiv 0$, by an infinite-dimensional version of Yamada-Watanabe result  (see~\cite[Prop 5.3.20]{karatzasshreve91}). Then we show that a solution to equation~\eqref{eq:stoppee} satisfies the $X_M$-hypotheses. The statement of step 2 follows by the same arguments as for Theorem~\ref{theo:weak_wp_hypotheses_simplifiees}.

\textit{Step 3.}
For each integer $M$ greater than $M_0:=g(1)-g(0)$, 
we consider the solution to equation~\eqref{eq:stoppee} on the canonical probability space. Let $\Omega=\Omega_1 \times \Omega_2$, where $\Omega_1:= \mathcal C([0,1]\times [0,T])$ and $\Omega_2:= \mathcal C(\R \times [0,T],\R^2)$, equipped with the class $\mathcal B(\Omega)$ of Borel subsets of $\Omega$. 
To every $z\in \Omega_1$, we associate $\zeta^M_z=\inf\{t \geq 0: z_t(1)-z_t(0) \geq M\}\wedge T$. 
Let $\mathcal G^M$ be the $\sigma$-algebra generated by the map $\pi^M:z \in \Omega_1 \mapsto z_{ \cdot \wedge \zeta^M_z} \in \Omega_1$. 
By step 1, there is a weak solution $(\Omega, \mathcal G^M \otimes \mathcal B(\Omega_2), (\mathcal G_t)_{t\in [0,T]}, \mathbb Q^M, z_{\cdot \wedge \zeta_z^M},w)$ to equation~\eqref{eq:stoppee}.

We prove that the family $( \mathbb Q^M)_{M\geq M_0}$ is consistent, which follows from uniqueness in law proved in step 2. We use Theorem V.4.2 of Parthasarathy's book~\cite[p.143]{parthasarathy} to construct a probability  measure $\mathbb Q$ on $(\Omega, \mathcal B(\Omega))$ such that for each $M \geq M_0$, for each $A \in \mathcal G^M$ and for each $B \in \mathcal B(\Omega_2)$, $\mathbb Q  \left[A \times B  \right]= \mathbb Q^M  \left[A \times B  \right]$. 

\textit{Step 4.}
In order to prove the theorem, it remains to prove that $\mathbb Q [ \zeta^M_z<T] \to 0$ when $M \to + \infty$.  
We control the martingale part of a solution $(z_t)_{t \in [0,T]}$ to equation~\eqref{eq:definition_FP_perturbee} by the same arguments as in Proposition~\ref{prop:growth}. For the control of the drift part, we use assumption $(B3)$ on~$B$ (it is the only point where this assumption is needed) to obtain:
\begin{align*}
\mathbb Q^M \left[ \sup_{t \leq T} \left|\int_0^{t \wedge \zeta^M} ( B(z_s(1),z_s)- B(z_s(0),z_s) ) \mathrm ds  \right| \geq \frac{M}{2}  \right]
\leq \frac{2}{M} 2T C_0(M) \underset{M \to + \infty}{ \longrightarrow} 0.
\end{align*}
This concludes the proof of weak well-posedness for equation~\eqref{eq:definition_FP_perturbee}.

\section{A continuum of admissible drift functions}
\label{part:continuum_admissible_drift}

In this part, we make the connection between the result of restoration of uniqueness obtained in Theorem~\ref{theo:weak_wp_hypotheses_simplifiees} and  results of existence and uniqueness for standard McKean-Vlasov equations driven by a velocity field that is merely measurable in the space variable (see~\cite{jourdain97,mishuraverten,lacker18,roecknerzhang18}). The connection reads in the form of a new existence and uniqueness result but for a suitable notion of weak solution and for a class of admissible drifts. We address both in the next two subsections.

\subsection{Description of the class of admissible drift functions}
\label{section:description_admissible_drift}

Recall the definition of the distance in total variation between two probability measures. For any $\mu, \nu \in \mathcal P(\R)$, 
\begin{align}
\label{def:dtv}
\dtv(\mu,\nu) = 2 \inf_{\substack{\mathcal L(X)=\mu \\ \mathcal L(Y)=\nu}} \P{X\neq Y}, 
\end{align}
where the infimum is taken here over every coupling $(X,Y)$ of random variables $X$ and $Y$ in $L_2(\Omega, \mathcal F, \mathbb P)$ with respective distributions $\mu$ and $\nu$, where $(\Omega,\mathcal F, \mathbb P)$ is any fixed Polish and atomless probability space.

Let us define the following space on which we will consider the drift function:
\begin{defin}
\label{def:class_HC}
Let $\eta > 0$ and  $\delta \in [0,1]$. 
We say that $b: \R \times \mathcal P_2(\R) \to \R$ is of class $(H^\eta, \mathcal C^\delta)$ if there are measurable functions  $\lambda^\Re, \lambda^\Im: \R \times \mathcal P_2(\R) \to \R$ and $\Lambda: \R \to \R_+$ such that for every $x\in \R$ and $\mu \in \mathcal P_2(\R)$, 
\begin{align*}
b(x,\mu)=\int_\R \crochetk^{-\eta} \Big( \cos (kx) \lambda^\Re(k,\mu) + \sin(kx) \lambda^\Im(k, \mu)  \Big) \mathrm dk ,
\end{align*}
where
\begin{itemize}
\item $\lambda:=\lambda^\Re + i\lambda^\Im$ is bounded in the measure variable: for every $k \in \R$ and $\mu \in \mathcal P_2(\R)$,  $|\lambda(k,\mu)| \leq \Lambda(k)$;
\item $\lambda$ is $\delta$-Hölder continuous in the measure variable: for every $k$ and for every $\mu, \nu \in \mathcal P_2(\R)$,  $|\lambda(k,\mu)-\lambda(k,\nu)| \leq \Lambda(k) \dtv(\mu,\nu)^\delta$;
\item  $\Lambda \in L_1(\R) \cap L_2(\R)$.
\end{itemize}
\end{defin}

In particular, if $b$ is of class $(H^\eta, \mathcal C^\delta)$, then for every $\mu \in \mathcal P_2(\R)$, the map $x \mapsto b(x,\mu)$ belongs to the Sobolev space $H^\eta(\R)$. Indeed, denoting by $\mathcal F(b(\cdot,\mu))$ the Fourier transform of $b(\cdot, \mu)$, we have
\begin{align*}
\int_\R |\crochetk^{\eta} \mathcal F(b(\cdot,\mu))(k)|^2 \mathrm dk \leq C 
\int_\R |\lambda(-k,\mu)|^2 \mathrm dk \leq C \int_\R \Lambda(-k)^2 \mathrm dk <+\infty. 
\end{align*}

Moreover, if $b$ is of class $(H^\eta, \mathcal C^\delta)$, then for every $x \in \R$, $\mu \mapsto b(x,\mu)$ is $\delta$-Hölder continuous in total variation distance:
\begin{align*}
|b(x,\mu)-b(x,\nu)| 
\leq \int_\R \crochetk^{-\eta} | \lambda(k,\mu)-\lambda(k,\nu)| \mathrm dk 
&\leq \int_\R \crochetk^{-\eta} \Lambda(k) \mathrm dk \;
\dtv(\mu,\nu)^\delta. 
\end{align*}
Since $\eta \geq 0$ and $\Lambda \in L_1(\R)$, $\int_\R \crochetk^{-\eta} \Lambda(k) \mathrm dk$ is finite. 

In order to apply our strategy, we need to assume the following minimal regularity assumption on the drift $b$: 
\begin{align*}
\eta > \frac{3}{2}(1-\delta).
\end{align*}
It describes a continuum of admissible drift functions~$b$ between the following two extremal classes:
\begin{itemize}
\item if $\delta=0$: the drift is only bounded in the measure variable. 
In that case, $\eta$ has to satisfy $\eta>\frac{3}{2}$: this coincide exactly with the assumptions of  Theorem~\ref{theo:weak_wp_hypotheses_simplifiees}, where we assumed that for each $\mu \in \mathcal P_2(\R)$, $b(\cdot,\mu)$ belongs  to $H^\eta(\R)$ with a uniform $H^\eta$-norm for some $\eta \geq \alpha>\frac{3}{2}$. 
\item if $\delta=1$: the drift is Lipschitz-continuous in total variation distance with respect to the measure argument. Jourdain~\cite{jourdain97},  Mishura-Veretennikov~\cite{mishuraverten}, Lacker~\cite{lacker18}, Chaudru de Raynal-Frikha~\cite{chaudrufrikha18}, Röckner-Zhang~\cite{roecknerzhang18} among others have proved results under this assumption if $b$ is only measurable and bounded in the space variable. Our result applies  if $b$ belongs to $H^\eta(\R)$ for some $\eta>0$ and if the Fourier transform of $b$ belongs to $L_1(\R)$; it is a subset of the space $\mathcal C_0(\R)$ of continuous functions vanishing at infinity. 
\end{itemize}

\subsection{Definition of the notion of solution}
\label{section:definition_solution}

Let us consider a new model, with the purpose to make a link between the results obtained in this paper and recent regularization by noise results for McKean-Vlasov equations obtained among others by~\cite{jourdain97,mishuraverten,lacker18, chaudrufrikha18, roecknerzhang18}. There are some important changes with respect to the model~\eqref{eq:definition_FP_perturbee} previously studied in this work. The main modification consists in adding a Brownian motion~$\beta$, independent of $w$, in order to take benefit from some additional regularizing effect. 
In short, the role of $\beta$ in the model below is to smooth out the (finite dimensional) space variable in the drift coefficient. Obviously, this comes in contrast with 
the role of the Brownian sheet $w$, the action of which is 
to mollify 
the velocity field in the measure argument, as made clear by 
Theorem~\ref{theo:weak_wp_hypotheses_simplifiees}. Of course, we know 
from the standard diffusive case (i.e. $w \equiv 0$ and $b(x,\mu) \equiv b(x)$) that, in order to fully benefit from the action of $\beta$ onto the space variable, we should average out over all the possible realizations of $\beta$ (for instance, we may consider the 
semi-group generated by the diffusion process). In the present context, this prompts us to disentangle the roles of the two noises 
$\beta$ and $w$ in the mean-field interaction. Similarly to the standard McKean-Vlasov model, we shall compute the law of the particle (i.e. the mean-field component) with respect to the noise carrying $\beta$ and the initial condition, but, similarly to the model addressed in the previous section, we shall freeze the realization of $w$. According to the terminology that has been used in the literature (see 
in particular the mean-field game literature \cite{gueant:lasry:lions,carmonadelarue18:II}, see also 
the earlier references
\cite{vaillancourt,dawson:vaillancourt,kurtz:xiong:1,kurtz:xiong:2,coghi}), 
$\beta$ will be regarded as an idiosyncratic noise acting independently on each particle and $w$ as a common (or systemic) noise.
To sum-up, in the previous sections, we defined $\mu_t$  as $\mu_t=\Leb_{[0,1]} \circ (z_t)^{-1}$, the space $[0,1]$ therein carrying the initial condition in the form $z_{0}(u)=g(u)$ for $u \in [0,1]$. Implicitly, this allowed us to identify~$\mu_{t}$ with the conditional law of $z_t$ given $(w(k,s))_{k \in \R, s\leq t}$. Now, $\mu_t$ will be understood as the law of the particle over the randomness carrying both $\beta$ and the initial condition. This idea is made more precise in Remark~\ref{rem:definition_mut}.

There are two other modifications of the model introduced in this section. In the Girsanov's arguments that we will use in the following proofs, we will not be able  to preserve the monotonicity of the solution with respect to the variable~$u$ 
as in the first part. 
So we decide to use the same framework as usual in the literature on McKean-Vlasov SDEs, namely we take as initial condition a random variable $\xi$ of prescribed law, independent from $\beta$ and $w$. 
Furthermore, we decide to consider the easiest possible assumption on the mass, namely that it is constant equal to one.

Let  $\eta > 0$ and  $\delta \in [0,1]$ be such that $\eta > \frac{3}{2}(1-\delta)$. Let $b:\R \times \mathcal P_2(\R) \to \R$ be of class $(H^\eta, \mathcal C^\delta)$. Let $f: \R \to \R$ be a function of order $\alpha$, such that
\begin{align}
\label{assumption_alpha}
\frac{3}{2}<  \alpha  \leq \frac{\eta}{1-\delta} ,
\end{align}
(if $\delta=1$, we just require that $\alpha> \frac{3}{2}$). 
The condition $\eta> \frac{3}{2}(1-\delta)$ insures that this choice of $\alpha$ is possible. 
Let $\mu_0$ be any given initial condition in $\mathcal P_2(\R)$. 
Let us consider the following SDE:
\begin{align}
\label{sde_beta_drift_general}
\left\{
\begin{aligned}
\mathrm dz_t&=  \int_\R f(k) \Re (  e^{-ik z_t} \mathrm d w (k,t)) + \mathrm d\beta_t + b( z_t, \mu_t) \mathrm dt, \\
\mu_t &= \mathcal L^{\mathbb P}(z_t |\mathcal G_t^{\mu, W})  \text{ a.s.}  \\
z_0 &=\xi, \quad    \mathcal L^{\mathbb P} (\xi)=\mu_0,
\end{aligned}
\right.
\end{align}
where the filtration $(\mathcal G_t^{\mu, W})_{t\in [0,T]}$ is defined by $\mathcal G_t^{\mu,W}:= \sigma \{ w(\cdot,s), \mu_s \; ; s\leq t\}$ and where $(\mu,w)$ is independent of $(\beta,\xi)$. 
\textit{Note:} In that equation and in all this section, $\mathcal L^{\mathbb P} (X)$ denotes the law of the random variable $X$ under the probability measure $\mathbb P$, that is the distribution $\mathbb P \circ X^{-1}$.

Let us define the notion of weak solution to~\eqref{sde_beta_drift_general}:
\begin{defin}
\label{def:weak_solution_drift_general}
An element $\mathbf \Omega= (\Omega, \mathcal G, (\mathcal G_t)_{t\in [0,T]}, \mathbb P, z,w, \beta, \xi)$ is said to be a \emph{weak solution} to equation~\eqref{sde_beta_drift_general} if
\begin{itemize}
\item[-] $(\Omega, \mathcal G,(\mathcal G_t)_{t\in [0,T]}, \mathbb P )$ is a filtered probability space satisfying usual conditions,
\item[-] $(w,\beta,\xi)$ are independent random variables on $(\Omega, \mathcal G)$, where
\subitem $\triangleright$ $w:=(\wR,\wI)$, with $(\wR(k,t))_{k\in \R,t\in [0,T]}$ and $(\wI(k,t))_{k\in \R,t\in [0,T]}$  two independent $(\mathcal G_t)_{t\in [0,T]}$-Brownian sheets under $\mathbb P$, 
\subitem $\triangleright$ $(\beta_t)_{t\in [0,T]}$ is a standard $(\mathcal G_t)_{t\in [0,T]}$-Brownian motion under $\mathbb P$,
\subitem $\triangleright$ for any $t \in [0,T]$, the $\sigma$-field  $\sigma \{\wR(k,t')-\wR(k,t),\; \wI(k,t')-\wI(k,t),\; \beta_{t'}-\beta_{t} ; $  $k \in \R,\ t' \in [t,T] \}$ is independent of ${\mathcal G}_{t}$ under ${\mathbb P}$,
\subitem $\triangleright$ $\xi$ has distribution $\mu_0$ under $\mathbb P$;
\item[-] $(z_t)_{t\in [0,T]}$ is a continuous $(\mathcal G_t)_{t\in [0,T]}$-adapted process satisfying $\mathbb P$-almost surely, for every $t\in [0,T]$,
\begin{align*}
z_t= \xi + \int_0^t \!\! \int_\R f(k) \Re( e^{-ikz_s} \mathrm dw(k,s)) +\beta_t+\int_0^t b(z_s,\mu_s) \mathrm ds.
\end{align*}
\item[-] $(\mu_{t})_{t \in [0,T]}$ is a ${\mathcal P}_{2}({\mathbb R})$-valued continuous 
$({\mathcal G}_{t})_{t \in [0,T]}$-adapted process such that, for every $t\in [0,T]$, ${\mathbb P}$-almost surely, $\mu_t= \mathcal L^{\mathbb P} (z_t | \mathcal G_t^{\mu,W})$, where $\mathcal G_t^{\mu,W}:= \sigma \{ w(k,s), \mu_s \; ; k\in \R, s\leq t\}$,
\item[-] \textit{compatibility condition:} $(\mu,w)$ is independent of $(\beta,\xi)$ under ${\mathbb P}$ (and thus $(\mu,w)$, $\beta$ and $\xi$ are independent) and, more generally, 
for every $t \in [0,T]$, the processes $(\xi,w,\mu)$ and $\beta$ are conditionally independent 
given $\mathcal G_t$.
\end{itemize}
\end{defin}

\begin{rem}
\label{rem:definition_mut}
The last two conditions are certainly the most difficult ones to understand. In fact, both are dictated by the fact that we are looking for weak solutions only: \textit{a priori}, nothing is said on the measurability of $z$ and $\mu$ with respect to the inputs 
$\xi$, $w$ and $\beta$. In particular, at this stage, $\mu$ may not be measurable with respect to $w$ (which comes in contrast with the intuitive explanations we gave in introduction of the section). This is the rationale for defining the McKean-Vlasov constraint 
in terms of the conditional law of $z_{t}$ given the $\sigma$-field generated (up to time $t$) not only by $w$ but also 
by $\mu$ itself. Similarly, the compatibility condition  has been widely used (in a slightly stronger manner) in the analysis of weak solutions to stochastic equations, see for instance \cite{kurtz07, kurtz14}. In short, it says that the observation of $z$ does not 
corrupt the independence property of  
$(\xi,\mu,w)$ and $\beta$. Quite obviously, see for instance \cite[Remark I.11]{carmonadelarue18:II}, compatibility is automatically satisfied if
$\mu$ is adapted with respect to the completion of ${\mathcal G}^W$, in which case the solution should be called semi-strong.
\end{rem}

We will prove weak well-posedness for the SDE~\eqref{sde_beta_drift_general} in three steps: $i)$ when the drift $b$ is equal to zero; $ii)$ when the drift $b$ is bounded and Lipschitz-continuous in total variation distance with respect to the measure variable; essentially, we will adapt to our case the proof given by Lacker~\cite{lacker18}, where we will make use of the averaging over the noise~$\beta$; $iii)$ in the general case, when the drift $b$ belongs to the class $(H^\eta, \mathcal C^\delta)$: we will use here the same arguments as in the first part, using the infinite-dimensional Brownian sheet~$w$ to mollify~$b$ in the measure argument.

Let us first consider the case where the drift is zero: 
\begin{align}
\label{sde_beta_drift_nul}
\left\{
\begin{aligned}
\mathrm dz_t&=  \int_\R f(k) \Re (  e^{-ik z_t} \mathrm d w (k,t)) + \mathrm d\beta_t , \\
z_0 &=\xi, \quad    \mathcal L^{\mathbb P} (\xi)=\mu_0. 
\end{aligned}
\right.
\end{align}
In this case, well-posedness holds even in a strong sense. 
\begin{prop}
\label{prop:sansdrift}
Let $f:\R \to \R$ be a function of order $\alpha> \frac{3}{2}$. Then there is a unique strong solution to equation~\eqref{sde_beta_drift_nul}. Moreover, if $\mathbf{\Omega^1}$ and $\mathbf{\Omega^2}$ are two solutions to~\eqref{sde_beta_drift_nul}, then $\mathcal L^{\mathbb P^1}(z^1,w^1, \beta^1)=\mathcal L^{\mathbb P^2}(z^2,w^2, \beta^2)$. 
\end{prop}

\begin{proof}
Strong well-posedness can be proved by a classical fixed-point argument, as in the proof of Proposition~\ref{prop:sde M ex&uniq} for example (but the proof is now easier since the mass is equal to $1$ everywhere). The additional noise $\beta$ does not change anything to this proof. Moreover, the assumption $\alpha> \frac{3}{2}$ insures that the assumption of square integrability of $k \mapsto \crochetk f(k)$ is satisfied (see Proposition~\ref{prop:sde M ex&uniq}); in other words, it insures that the diffusive coefficient in front of the noise $w$ is Lipschitz-continuous. 

Furthermore, by Yamada-Watanabe Theorem, the law of $(z,w,\beta)$ solution to~\eqref{sde_beta_drift_nul} is uniquely determined. That result is stated and proved in~\cite[Prop 5.3.20, p.309]{karatzasshreve91} in a finite-dimensional case, but the proof is the same for an   infinite dimensional noise.
Moreover,  a corollary to Yamada-Watanabe Theorem~\cite[Cor 5.3.23, p.310] {karatzasshreve91} states the following result: if $\mathbf \Omega=(\Omega, \mathcal G, (\mathcal G_t)_{t\in [0,T]}, \mathbb P, z,w, \beta, \xi)$ is a  solution to~\eqref{sde_beta_drift_nul}, then $\mathbb P$-almost surely, for every $t\in [0,T]$,
\begin{align*}
z_t= \mathcal Z_t(\xi, w,\beta),
\end{align*}
where $\mathcal Z$ is a function defined on the canonical space
\begin{align}
\label{Z_de_yamadawatanabe}
\mathcal Z: \R \times \mathcal C(\R \times [0,T] , \R^2) \times \mathcal C([0,T], \R) & \to \mathcal C([0,T], \R)  \\
(x,\omega^W,\omega^\beta) & \mapsto \mathcal Z(x,\omega^W,\omega^\beta)\notag 
\end{align}
 which is progressively measurable with respect to the canonical filtration on $ \mathcal C(\R \times [0,T] , \R^2) \times \mathcal C([0,T], \R)$. 
 Remark that $\mathcal C([0,T], \R)$ represents here the canonical space on which we define the Wiener measure of a standard Wiener process on $[0,T]$, and $\mathcal C(\R \times [0,T] , \R^2)$ represents the Wiener space associated to the measure of a $\R^2$-valued Brownian sheet $(\wR, \wI)$ on $\R \times [0,T]$. 
\end{proof}

\subsection{Resolution of the SDE when the drift is Lipschitz continuous}
\label{section:resolution_lipschitz}

Let us assume that $\widetilde{b}: \R \times \mathcal P_2(\R) \to \R$ is uniformly bounded and uniformly Lipschitz-continuous in total variation distance in the measure variable. 
We consider the following SDE with the drift $\widetilde{b}$:
\begin{align}
\label{sde_beta_drift_lipschitz}
\left\{
\begin{aligned}
\mathrm dz_t&=  \int_\R f(k) \Re (  e^{-ik z_t} \mathrm d w (k,t)) + \mathrm d\beta_t + \widetilde{b}( z_t, \mu_t) \mathrm dt, \\
\mu_t &= \mathcal L^{\mathbb P}(z_t | \mathcal G_t^{\mu,W}),  \\
z_0 &=\xi, \quad    \mathcal L^{\mathbb P} (\xi)=\mu_0,
\end{aligned}
\right.
\end{align}
with the same assumptions and the same interpretation as in Definition 
\ref{def:weak_solution_drift_general}. 
Let us prove existence and uniqueness of a weak solution. 
\begin{prop}
\label{prop:drift_lipschitz_exist}
Let $f:\R \to \R$ be a function of order $\alpha> \frac{3}{2}$. Let $\widetilde{b}: \R \times \mathcal P_2(\R) \to \R$ be a function such that there exists $C>0$ 
satisfying for every $x \in \R$ and for every $\mu,\nu \in \mathcal P_2(\R)$
\begin{itemize}
\item[-] $ |\widetilde{b}(x,\mu)| \leq C$;
\item[-] $|\widetilde{b}(x,\mu)-\widetilde{b}(x,\nu) | \leq C \dtv (\mu,\nu)$.
\end{itemize}
Then there exists a  weak solution to~\eqref{sde_beta_drift_lipschitz}.
\end{prop}

\begin{prop}
\label{prop:drift_lipschitz_uniqu}
Under the same assumptions as Proposition~\ref{prop:drift_lipschitz_exist}, 
if  $\mathbf{\Omega^1}$ and $\mathbf{\Omega^2}$ are two weak solutions to~\eqref{sde_beta_drift_lipschitz}, then $\mathcal L^{\mathbb P^1}(z^1,w^1)=\mathcal L^{\mathbb P^2}(z^2,w^2)$. In particular, uniqueness in law holds for the SDE~\eqref{sde_beta_drift_lipschitz}. 
Moreover, for any weak solution $\mathbf \Omega$, $(\mu_t)_{t\in [0,T]}$ is adapted to the completion of  $(\mathcal G_t^W= \sigma \{ w(\cdot,s)  \; ; s\leq t\})_{t\in [0,T]}$. 
\end{prop}

Note that the statement of Proposition~\ref{prop:drift_lipschitz_uniqu} shows that the weak solution of~\eqref{sde_beta_drift_lipschitz} is adapted to the filtration generated by the noise $w$.

\begin{rem}
\label{rem:adaptness}
The question of the filtration under which the measure-valued process $(\mu_t)_{t\in [0,T]}$ is adapted is important here. Actually, we will see in the proof of existence that the weak solution that we will construct is automatically adapted with respect to the filtration generated by $w$. Nevertheless, we want to give a more general statement for uniqueness, \textit{i.e.} we want to be able to compare two weak solutions where $(\mu_t)_{t\in [0,T]}$ is adapted with respect to a filtration generated by $w$ and possibly another source of randomness, provided $(\mu,W)$ remains independent of $(\beta,\xi)$. This will be useful in the proof of Theorem~\ref{theo:interpolation_lacker}, which states well-posedness for the SDE with $(H^\eta,\mathcal C^\delta)$-drift $b$, since for this general case, our proof based on Girsanov's Theorem does not imply that $(\mu_t)_{t\in [0,T]}$ is adapted with respect to  the filtration generated by $w$ (see Remark~\ref{rem:adaptness_bis}). 
\end{rem}

The assumptions on $\widetilde{b}$ are the same as in~\cite{lacker18}. We will essentially apply the same proof, which we will recall hereafter. 

\subsubsection{Existence of a weak solution to the intermediate SDE}

Let us prove in this paragraph Proposition~\ref{prop:drift_lipschitz_exist}. We begin by constructing a weak solution on the canonical space. 

\begin{proof}[Proof (Proposition~\ref{prop:drift_lipschitz_exist})]
Let us consider the filtered canonical probability space, denoted by $(\Omega^W, \mathcal G^W, (\mathcal G^W_t)_{t\in [0,T]}, \mathbb P^W)$, where $\Omega^W:= \mathcal C(\R \times [0,T] , \R^2)$, $\mathcal G^W$ is the Borel $\sigma$-algebra on $\Omega^W$,  $(\mathcal G^W_t)_{t\in [0,T]}$ is the canonical filtration on $(\Omega^W, \mathcal G^W)$ and  $\mathbb P^W$ is the probability measure on $(\Omega^W, \mathcal G^W)$ such that the  distribution of the random variable $w^W
\mapsto w^W$ is the law of two independent (real-valued) Brownian sheets on $\R \times [0,T]$.

Let $(\Omega^{\beta,\xi}, \mathcal G^{\beta,\xi}, (\mathcal G^{\beta,\xi}_t)_{t\in [0,T]}, \mathbb P^{\beta,\xi})$ be another filtered probability space on which we define   two independent random variables $\xi$ and $(\beta_t)_{t\in [0,T]}$  such that $(\beta_t)_{t\in [0,T]}$ is a $(\mathcal G^{\beta,\xi}_t)_{t\in [0,T]}$-adapted Brownian motion and such that the law of $\xi$ is $\mu_0$.

Let  $(\Omega, \mathcal G, (\mathcal G_t)_{t\in [0,T]}, \mathbb P)$ be the product space: $\Omega=\Omega^W \times \Omega^{\beta,\xi}$, $\mathcal G=\mathcal G^W \otimes \mathcal G^{W,\beta}$, $\mathcal G_t=\sigma(\mathcal G_t^W, \mathcal G_t^{\beta,\xi})$ and $\mathbb P = \mathbb P^W \otimes \mathbb P^{\beta,\xi}$. In particular, $w$ is independent of $(\beta,\xi)$ under $\mathbb P$. 
Up to adding negligible subsets, we assume that the filtration $(\mathcal G_t)_{t\in [0,T]}$ is complete. 
Let $(z_t)_{t\in [0,T]}$ be the unique solution on $(\Omega, \mathcal G, (\mathcal G_t)_{t\in [0,T]}, \mathbb P)$ of the SDE:
\begin{align}
\label{sde_beta_drift_nul_espace_canonique}
\left\{
\begin{aligned}
\mathrm dz_t&=  \int_\R f(k) \Re (  e^{-ik z_t} \mathrm d w (k,t)) + \mathrm d\beta_t , \\
z_0 &=\xi, \quad    \mathcal L^{\mathbb P} (\xi)=\mu_0. 
\end{aligned}
\right.
\end{align}
Existence and uniqueness of a strong solution to~\eqref{sde_beta_drift_nul_espace_canonique} is given by Proposition~\ref{prop:sansdrift}. Furthermore, by Yamada-Watanabe Theorem, there is a $(\mathcal G_t)_{t\in [0,T]}$-progressively measurable map $\mathcal Z_t$ as defined in~\eqref{Z_de_yamadawatanabe} such that $\mathbb P$-almost surely, $z_t=\mathcal Z_t(\xi,w,\beta)$.

Let us denote by $\mathcal C$  the space $\mathcal C([0,T],\R)$ and by $\mathcal P(\mathcal C)$ the space of probability measures on $\mathcal C$. 
For each time $t\in [0,T]$, let us denote by $\pi_t: \mu^T \in \mathcal P(\mathcal C) \mapsto \mu^t \in \mathcal P({\mathcal C})$ the map associating to $\mu^T$ the push-forward measure of $\mu^T$ by the map $x \in \mathcal C \mapsto x_{\cdot \wedge t} \in {\mathcal C}$. 
Let $(\mathcal X,d)$ be the complete metric space of  functions
\begin{align*}
\mu: \Omega^W= \mathcal C(\R \times [0,T], \R^2)  &\to \mathcal P(\mathcal C) \\
w &\mapsto \mu^T(w),
\end{align*}
such that, for each $t\in [0,T]$, $(\mu^t = \pi^t(\mu^T))_{t\in [0,T]}$ is $(\mathcal G^W_t)_{t\in [0,T]}$-progressively measurable. The distance $d$ is defined by $d(\mu,\nu):= \mathbb E^W \left[  \dtv (\mu^T,\nu^T)^2 \right]^{1/2}$, where 
$\dtv$ is here understood as the total variation distance on ${\mathcal P}(\mathcal C)$ (while we defined it before on ${\mathcal P}(\R)$).
Furthermore, for $\mu \in \mathcal X$ and for $t \in [0,T]$, we call $\mu_{t}$ the image of $\mu$ by the mapping 
$x \in {\mathcal C} \mapsto x_{t} \in \R$.

Let $\nu \in \mathcal X$.
Recall that $\widetilde{b}: \R \times \mathcal P_2(\R) \to \R$ is uniformly bounded. Therefore
\begin{align*}
\mathcal E^\nu_t:= \exp \left( \int_0^t \widetilde{b}(z_s, \nu_s) \mathrm d\beta_s - \frac{1}{2}\int_0^t | \widetilde{b}(z_s, \nu_s) |^2 \mathrm ds \right)
\end{align*}
is a $(\mathcal G_t)_{t \in [0,T]}$-martingale.
Let $\mathbb P^\nu$ be the probability measure on $(\Omega, \mathcal G)$ absolutely continuous with respect to $\mathbb P=\mathbb P^W \otimes \mathbb P^{\beta,\xi}$, with density:
\begin{align*}
\frac{\mathrm d \mathbb P^\nu}{\mathrm d\mathbb P}=\mathcal E^\nu_T.
\end{align*}
For every $w \in \Omega^W$, let us denote by $\mathbb P^{\nu, \beta,\xi}(w)$ the probability measure on $(\Omega^{\beta,\xi}, \mathcal G^{\beta,\xi})$ with the following density with respect to $\mathbb P^{\beta,\xi}$:
\begin{align*}
\frac{\mathrm d \mathbb P^{\nu,\beta,\xi}(w)}{\mathrm d \mathbb P^{\beta,\xi}}=\mathcal E^\nu_T (w).
\end{align*}
Equivalently, $\mathbb P^{\nu, \beta,\xi}:\Omega^W \times \mathcal G^{\beta,\xi} \to \R_+$ is also defined as the conditional probability satisfying
 for every $(A^W,A^{\beta,\xi}) \in  \mathcal G^W \times \mathcal G^{\beta,\xi}$
\begin{align*}
\mathbb P^\nu (A^W \times A^{\beta,\xi})= \int_{A^W}   \mathbb P^{\nu, \beta,\xi } (w, A^{\beta,\xi}) \mathrm d \mathbb P^W (w).
\end{align*}
Let us define $\mathrm d \widetilde{\beta}_t^\nu:= \mathrm d\beta_t- \widetilde{b} (z_t,\nu_t) \mathrm dt$. By Girsanov's Theorem, $(\widetilde{\beta}_t^\nu)_{t\in [0,T]}$ is a Brownian motion under the measure $\mathbb P^\nu$, 
$(\widetilde{\beta}^\nu, \xi, w)$ are independent under $\mathbb P^\nu$ and, for any $t \in [0,T]$, the $\sigma$-field $\sigma \{w(k,t')-w(k,t),\widetilde \beta_{t'}^\nu -\widetilde \beta_{t}^\nu; \ k \in \R, \ t' \in [t,T] \}$ is independent of ${\mathcal G}_{t}$ under ${\mathbb P}^{\nu}$. 
Moreover the process $(z_t)_{t\in [0,T]}$ satisfies:
\begin{align*}
\mathrm dz_t =  \int_\R f(k) \Re (  e^{-ik z_t} \mathrm d w (k,t)) + \mathrm d\widetilde{\beta}^\nu_t + \widetilde{b} (z_t,\nu_t) \mathrm dt.
\end{align*}
If $\nu$ satisfies for every $t\in [0,T]$, $\mathbb P^W$-almost surely, 
\begin{align}
\label{condition_point_fixe_nu}
\nu_t= \mathcal L^{\mathbb P^\nu} (z_t |\mathcal G_t^W),
\end{align}
 then it also satisfies for every $t\in [0,T]$, 
$\mathbb P^W$-almost surely, $\nu_t=\mathcal L^{\mathbb P^\nu} (z_t |\mathcal G_t^{\nu,W})$, where $\mathcal G_t^{\nu,W}=\sigma \{ w(\cdot,s), \nu_s \; ; s\leq t\}$. Furthermore, $(\nu,w)$ is adapted to the completion of $\mathcal G^W$; hence under~$\mathbb P^\nu$, $(\nu,w)$ is independent of $(\tilde \beta^{\nu},\xi)$, and by Remark~\ref{rem:definition_mut} the compatibility condition is automatically satisfied. Thus if~\eqref{condition_point_fixe_nu} is satisfied for any $t\in [0,T]$ $\mathbb P^W$-almost surely, then  $(\Omega , \mathcal G, (\mathcal G_t)_{t\in [0,T]}, \mathbb P^\nu,z, w, \widetilde{\beta}^\nu, \xi)$  is a weak solution to~\eqref{sde_beta_drift_lipschitz}.
 Equivalently, it is solution if for $\mathbb P^W$-almost every $w \in \Omega^W$,  for every $t\in [0,T]$, $\nu_t(w)=\mathcal L^{\mathbb P^{\nu,\beta,\xi}(w)}(\mathcal Z_t(\cdot,w,\cdot))$ (the latter obviously implying 
 \eqref{condition_point_fixe_nu}
 and the converse following from the fact that, in 
 \eqref{condition_point_fixe_nu},  
 $\mathcal G_t^W$ can be replaced by 
 $\mathcal G_T^W$, which implies not only that, for any $t \in [0,T]$, 
  for $\mathbb P^W$-almost every $w \in \Omega^W$, 
 $\nu_t(w)=\mathcal L^{\mathbb P^{\nu,\beta,\xi}(w)}(\mathcal Z_t(\cdot,w,\cdot))$
but also that the quantifiers \textit{for all $t$} and \textit{for $\mathbb P^W$-almost every} 
can be exchanged by a standard continuity argument). Notice in particular, that by Fubini's Theorem, $w \mapsto \mathcal L^{\mathbb P^{\nu,\beta,\xi}(w)}(\mathcal Z_t(\cdot,w,\cdot))$ is $\mathcal G_t^W$-measurable. 

Let us prove that there is a process $\nu \in \mathcal X$ satisfying~\eqref{condition_point_fixe_nu}. 
For every $\nu \in \mathcal X$, let us define $\phi(\nu)_t :=  w \mapsto \mathcal L^{\mathbb P^{\nu,\beta,\xi}(w)}(\mathcal Z_t(\cdot,w,\cdot))$. By construction, $\phi(\nu)$  also belongs to $\mathcal X$.
For every $\mu^T \in \mathcal X$ and for every $t \in [0,T]$, let us denote by $\mu^t \in \mathcal P(\mathcal C)$ the push-forward measure of $\mu^T$ through the map $x \in \mathcal C \mapsto x_{\cdot \wedge t} \in \mathcal C$.  In particular, for every $t\in [0,T]$, $\phi(\nu)^t = \mathcal L^{\mathbb P^{\nu,\beta,\xi}(w)}(\mathcal Z_{\cdot \wedge t}(\cdot,w,\cdot))$. 
For $\mu,\nu \in \mathcal P_2(\mathcal C)$, let us denote by $H(\mu| \nu)$ the relative entropy
\begin{align*}
H(\mu|\nu) = \int_{\mathcal C}  \ln \frac{\mathrm d \mu}{\mathrm d\nu} \mathrm d\mu \quad \text{if } \mu \ll \nu, \quad \quad
H(\mu|\nu)=+\infty \quad \text{otherwise}.
\end{align*}

Here, we apply the same strategy of proof as in~\cite[Thm 2.4]{lacker18}.
Let us state the following lemma, which is shown at the end of the current proof. 
\begin{lemme}
\label{lemme:inegalite_lacker}
For every $\mu, \nu \in \mathcal X$ and for every $t\in [0,T]$, 
\begin{align*}
H(\phi(\mu)^t | \phi(\nu)^t)
=\frac{1}{2}
 \mathbb E ^{\mathbb P^\mu} \left[
\int_0^t | \widetilde{b}(z_s, \nu_s) - \widetilde{b}(z_s, \mu_s) |^2 \mathrm ds     \Big | \mathcal G_t^W 
\right]. 
\end{align*}
\end{lemme}

By Lipschitz-continuity of $\widetilde{b}$, there is $C>0$ such that
\begin{align*}
H(\phi(\mu)^t | \phi(\nu)^t) 
\leq C \int_0^t \mathbb E ^{\mathbb P^\mu} \left[\dtv(\mu_s,\nu_s)^2
\Big | \mathcal G_t^W \right] \mathrm ds
&= C \int_0^t \dtv(\mu_s,\nu_s)^2 \mathrm ds \\
&\leq  C \int_0^t \dtv(\mu^s,\nu^s)^2 \mathrm ds.
\end{align*}
By Pinsker's inequality, $\dtv(\phi(\mu)^t, \phi(\nu)^t)^2 \leq 2 H(\phi(\mu)^t | \phi(\nu)^t)$. 
Therefore, there is $C$ such that for every $t\in [0,T]$,
\begin{align}
\label{phi une fois}
\mathbb E^W \left[ \dtv ( \phi(\mu)^t, \phi(\nu)^t)^2   \right]
\leq C \int_0^t  \mathbb E^W \left[  \dtv ( \mu^s, \nu^s)^2   \right] \mathrm ds.
\end{align}
For every $n \in \Nun$, let us write $\phi^{\circ n}$ for $\underbrace{\phi \circ \dots \circ \phi}_{n \text{ times}}$. It follows from a simple recursion and from~\eqref{phi une fois} that for every $t\in [0,T]$ and  $n \in \N$
\begin{align*}
\mathbb E^W \left[ \dtv ( \phi^{\circ n}(\mu)^t, \phi^{\circ n}(\nu)^t)^2   \right]
\leq \frac{C^n t^n}{n!}  \mathbb E^W \left[  \dtv ( \mu^t, \nu^t)^2   \right] .
\end{align*}
Recall that the distance $d$ on $\mathcal X$ is defined by $d(\mu,\nu)=\mathbb E^W \left[  \dtv (\mu^T,\nu^T)^2 \right]^{1/2}$. Thus for every $n \geq 1$ and for every $\mu, \nu \in \mathcal X$, 
\begin{align*}
d( \phi^{\circ n}(\mu), \phi^{\circ n}(\nu))^2 \leq \frac{C^nT^n}{n!} d(\mu,\nu)^2.  
\end{align*}
Therefore, for $n$ large enough so that $\frac{C^nT^n}{n!} <1$, $\phi^{\circ n}$ is a contraction. Therefore, by Picard's fixed-point Theorem, there is a unique solution, called $\mu^\phi \in \mathcal X$, of $\mu^\phi= \phi(\mu^\phi)$. 
In particular, 
there exists a weak solution to equation~\eqref{sde_beta_drift_lipschitz}. This completes the proof of Proposition~\ref{prop:drift_lipschitz_exist}. 
\end{proof}

\begin{proof}[Proof (Lemma~\ref{lemme:inegalite_lacker})]
Let us first compute for every $t\in [0,T]$, 
\begin{align}
\label{lem38d}
H(\phi(\mu)^t | \phi(\nu)^t)
= \int_{\mathcal C} \ln \frac{\mathrm d \phi(\mu)^t}{\mathrm d \phi(\nu)^t} \mathrm d \phi(\mu)^t
=\mathbb E ^{\mathbb P^\mu} \left[ \ln  \frac{\mathrm d \phi(\mu)^t}{\mathrm d \phi(\nu)^t} (z_{\cdot \wedge t}) \Big | \mathcal G_t^W\right].
\end{align}
Let us prove that 
\begin{align}
\label{lem38a}
\frac{\mathrm d \phi(\mu)^t}{\mathrm d \phi(\nu)^t} (z_{\cdot \wedge t})
= \mathbb E^{\mathbb P^\nu} \left[ \frac{\mathrm d \mathbb P^\mu}{\mathrm d \mathbb P^\nu}\Big| \mathcal G_t^{z,W} \right],
\end{align}
where $\mathcal G_t^{z,W}=\sigma \{ z_s, w(\cdot,s) \; ; s\leq t\}$. Indeed,  for every measurable and bounded functions $f:\mathcal C([0,T], \R) \to \R$ and $g:\mathcal C(\R \times [0,T], \R^2) \to \R$ (recall that we denote by $\mathcal C$ the space $\mathcal C([0,T], \R)$):
\begin{multline}
\label{lem38b}
\mathbb E^{\mathbb P^\nu} \left[   \frac{\mathrm d \mathbb P^\mu}{\mathrm d \mathbb P^\nu} f(z_{\cdot \wedge t}) g(w_{\cdot \wedge t})\right]
=\mathbb E^{\mathbb P^\mu} \left[  f(z_{\cdot \wedge t}) g(w_{\cdot \wedge t})\right]
=\mathbb E^{\mathbb P^\mu} \left[  \int_{\mathcal C} f(x) \mathrm d \phi(\mu)^t(x) \; g(w_{\cdot \wedge t})\right] \\
= \mathbb E^{\mathbb P^\mu} \left[  \int_{\mathcal C} f(x) \frac{\mathrm d \phi(\mu)^t}{\mathrm d \phi(\nu)^t}(x) \mathrm d \phi(\nu)^t(x) \;g(w_{\cdot \wedge t})\right]
= \mathbb E^{\mathbb P^\nu} \left[  \frac{\mathrm d \mathbb P^\mu}{\mathrm d \mathbb P^\nu} \int_{\mathcal C} f(x) \frac{\mathrm d \phi(\mu)^t}{\mathrm d \phi(\nu)^t}(x) \mathrm d \phi(\nu)^t(x) \;g(w_{\cdot \wedge t})\right] \\
\shoveright{
= \mathbb E^{\mathbb P^\nu} \left[ \mathbb E^{\mathbb P^\nu} \left[ \frac{\mathrm d \mathbb P^\mu}{\mathrm d \mathbb P^\nu} \Big| \mathcal G_t^W \right] \int_{\mathcal C} f(x) \frac{\mathrm d \phi(\mu)^t}{\mathrm d \phi(\nu)^t}(x) \mathrm d \phi(\nu)^t(x) \;g(w_{\cdot \wedge t})\right] }\\
= \mathbb E^{\mathbb P^\nu} \left[ \mathbb E^{\mathbb P^\nu} \left[ \frac{\mathrm d \mathbb P^\mu}{\mathrm d \mathbb P^\nu} \Big| \mathcal G_t^W \right] f(z_{\cdot \wedge t}) \frac{\mathrm d \phi(\mu)^t}{\mathrm d \phi(\nu)^t}(z_{\cdot \wedge t})  \;g(w_{\cdot \wedge t})\right].
\end{multline}
Moreover, recalling the relation $\mathrm d \beta_s= \mathrm d\widetilde{\beta}^\nu_s+ \widetilde{b}(z_s,\nu_s) \mathrm ds$, 
\begin{multline*} 
\mathbb E^{\mathbb P^\nu} \left[ \frac{\mathrm d \mathbb P^\mu}{\mathrm d \mathbb P^\nu} \Big| \mathcal G_t^W \right]
= \mathbb E^{\mathbb P^\nu} \left[ \mathcal E_T^\mu (\mathcal E_T^\nu)^{-1} \Big| \mathcal G_t^W \right] \\
\begin{aligned}
&=\mathbb E^{\mathbb P^\nu} \left[\exp \left( \int_0^T (\widetilde{b}(z_s, \mu_s)-\widetilde{b}(z_s, \nu_s)) \mathrm d\beta_s - \frac{1}{2}\int_0^T | \widetilde{b}(z_s, \mu_s) |^2 \mathrm ds +  \frac{1}{2}\int_0^T | \widetilde{b}(z_s, \nu_s) |^2 \mathrm ds \right) \Big| \mathcal G_t^W \right] \\
&=\mathbb E^{\mathbb P^\nu} \left[\exp \left( \int_0^T (\widetilde{b}(z_s, \mu_s)-\widetilde{b}(z_s, \nu_s)) \mathrm d\widetilde{\beta}^\nu_s - \frac{1}{2}\int_0^T | \widetilde{b}(z_s, \mu_s) -\widetilde{b}(z_s, \nu_s)|^2 \mathrm ds  \right) \Big| \mathcal G_t^W \right] .
\end{aligned}
\end{multline*}
For every bounded and measurable $g:\mathcal C(\R \times [0,T], \R^2) \to \R$
\begin{multline*}
\mathbb E^{\mathbb P^\nu} \Big[ \exp \Big( \int_0^T (\widetilde{b}(z_s, \mu_s)-\widetilde{b}(z_s, \nu_s)) \mathrm d\widetilde{\beta}^\nu_s - \frac{1}{2}\int_0^T | \widetilde{b}(z_s, \mu_s) -\widetilde{b}(z_s, \nu_s)|^2 \mathrm ds  \Big)g(w_{\cdot\wedge t}) \Big] \\
\begin{aligned}
&=\mathbb E^{\mathbb P^\nu} \Big[ \mathbb E^{\mathbb P^{\nu,\beta,\xi}(w)} \Big[  \exp \Big( \int_0^T \!(\widetilde{b}(z_s, \mu_s)-\widetilde{b}(z_s, \nu_s)) \mathrm d\widetilde{\beta}^\nu_s - \frac{1}{2}\int_0^T | \widetilde{b}(z_s, \mu_s) -\widetilde{b}(z_s, \nu_s)|^2 \mathrm ds  \Big) \Big]g(w_{\cdot\wedge t}) \Big]  \\
&=\mathbb E^{\mathbb P^\nu} \left[ g(w_{\cdot\wedge t}) \right] ,
\end{aligned}
\end{multline*}
since under $\mathbb P^\nu$, $\widetilde{\beta}^\nu$ and $w$ are independent and since the  exponential is a $\mathbb P^{\nu,\beta,\xi}(w)$-martingale by Novikov's condition (recalling that $\widetilde{b}$ is uniformly bounded).
Thus
\begin{align}
\label{lem38c}
\mathbb E^{\mathbb P^\nu} \left[ \frac{\mathrm d \mathbb P^\mu}{\mathrm d \mathbb P^\nu} \Big| \mathcal G_t^W \right]=1. 
\end{align}
Using equalities~\eqref{lem38b} and~\eqref{lem38c}, we get equality~\eqref{lem38a}. 

Therefore, back to equality~\eqref{lem38d}, we obtain
\begin{align*}
H(\phi(\mu)^t | \phi(\nu)^t)
=\mathbb E ^{\mathbb P^\mu} \left[ \ln  \mathbb E^{\mathbb P^\nu} \left[ \frac{\mathrm d \mathbb P^\mu}{\mathrm d \mathbb P^\nu}\Big| \mathcal G_t^{z,W} \right] \Big | \mathcal G_t^W\right].
\end{align*}
Recall that $(\Omega , \mathcal G, (\mathcal G_t)_{t\in [0,T]}, \mathbb P^\nu,z, w, \widetilde{\beta}^\nu, \xi)$  is a weak solution to~\eqref{sde_beta_drift_lipschitz}. Thus $\mathbb P^\nu$-almost surely, $ \widetilde{\beta}^\nu$ satisfies for every $t\in [0,T]$:
\begin{align*}
 \widetilde{\beta}^\nu_t = z_t -z_0 - \int_0^t \!\! \int_\R f(k) \Re( e^{-ik z_s} \mathrm dw(k,s) ) - \int_0^t \widetilde{b}(z_s,\mu_s) \mathrm ds.
\end{align*}
Thus $(\widetilde{\beta}^\nu_t)_{t\in [0,T]}$ is $(\mathcal G_t^{z,W})_{t\in [0,T]}$-adapted and we deduce that
\begin{multline*}
H(\phi(\mu)^t | \phi(\nu)^t) \\
\begin{aligned}
&=\mathbb E ^{\mathbb P^\mu} \Big[ \ln  \exp \Big( \int_0^t (\widetilde{b}(z_s, \mu_s)-\widetilde{b}(z_s, \nu_s)) \mathrm d\widetilde{\beta}^\nu_s - \frac{1}{2}\int_0^t | \widetilde{b}(z_s, \mu_s) -\widetilde{b}(z_s, \nu_s)|^2 \mathrm ds \Big)  \Big | \mathcal G_t^W\Big] \\
&=\mathbb E ^{\mathbb P^\mu} \Big[ \int_0^t (\widetilde{b}(z_s, \mu_s)-\widetilde{b}(z_s, \nu_s)) \mathrm d\widetilde{\beta}^\nu_s - \frac{1}{2}\int_0^t | \widetilde{b}(z_s, \mu_s) -\widetilde{b}(z_s, \nu_s)|^2 \mathrm ds \Big | \mathcal G_t^W\Big] \\
&=\mathbb E ^{\mathbb P^\mu} \Big[ \int_0^t (\widetilde{b}(z_s, \mu_s)-\widetilde{b}(z_s, \nu_s)) \mathrm d\widetilde{\beta}^\mu_s + \frac{1}{2}\int_0^t | \widetilde{b}(z_s, \mu_s)-\widetilde{b}(z_s, \nu_s) |^2 \mathrm ds   \Big | \mathcal G_t^W\Big] \\
&=  \mathbb E ^{\mathbb P^\mu} \Big[   \frac{1}{2}\int_0^t | \widetilde{b}(z_s, \mu_s) -\widetilde{b}(z_s, \nu_s)|^2 \mathrm ds    \Big | \mathcal G_t^W\Big],
\end{aligned}
\end{multline*}
because $\mathrm d\widetilde{\beta}^\nu_s- \mathrm d\widetilde{\beta}^\mu_s=(\widetilde{b}(z_s, \mu_s)-\widetilde{b}(z_s, \nu_s)) \mathrm ds$. This completes the proof of the Lemma. 
\end{proof}

\subsubsection{Uniqueness in law for the  intermediate SDE}

Let us prove in this paragraph Proposition~\ref{prop:drift_lipschitz_uniqu}. 

\begin{proof}[Proof (Proposition~\ref{prop:drift_lipschitz_uniqu})]
Let $\mathbf{\Omega^1}$ and $\mathbf{\Omega^2}$ be two weak solutions to~\eqref{sde_beta_drift_lipschitz}, often denoted by~$\mathbf{\Omega^n}$, $n=1,2$. In particular, the  process $(z^n_t)_{t\in [0,T]}$ satisfies $\mathbb P^n$-almost surely, 
\begin{align*}
z^n_t= \xi^n + \int_0^t \!\! \int_\R f(k) \Re( e^{-ikz^n_s} \mathrm dw^n(k,s)) +\beta^n_t+\int_0^t \widetilde{b}(z^n_s,\mu^n_s) \mathrm ds, 
\end{align*}
where for every $t\in [0,T]$, $\mathbb P^n$-almost surely $\mu^n_t= \mathcal L^{\mathbb P^n} (z_t^n | \mathcal G_t^{\mu^n,W^n})$
and where $(\mu^n,w^n)$ is independent of $(\beta^n,\xi^n)$. 

Let $\mathbb Q^n$ be the probability measure on $(\Omega^n, \mathcal G^n)$ with the following density with respect to~$\mathbb P^n$, 
\begin{align}
\label{qi_sur_pi}
\frac{\mathrm d \mathbb Q^n}{\mathrm d \mathbb P^n}
= \exp \left( -\int_0^T \widetilde{b}(z^n_s, \mu_s^n) \mathrm d\beta^n_s -\frac{1}{2} \int_0^T |\widetilde{b}(z^n_s, \mu_s^n)|^2 \mathrm ds \right). 
\end{align}
Let $\widetilde{\beta}^n_t =\beta^n_t+\int_0^t \widetilde{b}(z^n_s,\mu^n_s) \mathrm ds$. By Girsanov's Theorem, $\mathcal L^{\mathbb Q^n}(w^n,\widetilde{\beta}^n, \xi^n)=\mathcal L^{\mathbb P^n}(w^n,\beta^n, \xi^n)$
and for any $t \in [0,T]$, the $\sigma$-field $\sigma \{w^n(k,t')-w^n(k,t),\widetilde \beta^n_{t'}-\widetilde \beta^n_{t}, \ k \in \R, \ t' \in [t,T] \}$ is independent of ${\mathcal G}_{t}^n$ under ${\mathbb Q}^n$.  It follows that $\widetilde{\mathbf {\Omega^n}}=(\Omega^n,  \mathcal G^n, (\mathcal G^n_t)_{t\in [0,T]}, \mathbb Q^n, z^n,w^n, \widetilde{\beta}^n, \xi^n)$ is a weak solution to the SDE~\eqref{sde_beta_drift_nul} with zero drift. By Proposition~\ref{prop:sansdrift}, $\mathcal L^{\mathbb Q^1}(z^1,w^1,\widetilde{\beta}^1)=\mathcal L^{\mathbb Q^2}(z^2,w^2,\widetilde{\beta}^2)$ and $\mathbb Q^n$-almost surely, $z^n=\mathcal Z(\xi^n,w^n, \widetilde{\beta}^n)$, where $\mathcal Z$ is of the form~\eqref{Z_de_yamadawatanabe}.

Moreover, recall that for $n=1,2$, $\mu^n_t= \mathcal L^{\mathbb P^n} (z_t^n | \mathcal G_t^{\mu^n,W^n})$. Recall also that $\mu^\phi$ is defined as being the unique fixed-point of $\phi$ in $\mathcal X$ (see proof of Proposition~\ref{prop:drift_lipschitz_exist}). 
Let us state the following lemma, which will be shown at the end of the current proof. 
\begin{lemme}
\label{lemme:mui_point_fixe}
Let $n=1,2$. Then $\mathbb Q^n$-almost surely, for every $t\in [0,T]$, $\mu^n_t= \mu^\phi (w^n)_t$. In particular, $(\mu^n_t)_{t\in [0,T]}$ is adapted to the completion of $(\mathcal G_t^{W^n})_{t\in [0,T]}$, where $\mathcal G_t^{ W^n}:= \sigma\{ w^n(k,s)\;; k\in \R, s\leq t \}$.
\end{lemme}

Let us consider a measurable function $\psi : \mathcal C([0,T],\R)\times \mathcal C(\R \times [0,T],\R^2) \to \R$ such that $\mathbb E^{\mathbb P^n} \left[ |\psi(z^n,w^n)| \right]<+\infty$ for $n=1,2$. It follows from~\eqref{qi_sur_pi} and from Lemma~\ref{lemme:mui_point_fixe} that
\begin{align*}
\mathbb E^{\mathbb P^n} \left[ \psi(z^n,w^n) \right]
&=\mathbb E^{\mathbb Q^n} \left[ \psi(z^n,w^n)
\exp \left( \int_0^T \widetilde{b}(z^n_s, \mu_s^n) \mathrm d\beta^n_s +\frac{1}{2} \int_0^T |\widetilde{b}(z^n_s, \mu_s^n)|^2 \mathrm ds \right)
 \right] \notag\\
 &=\mathbb E^{\mathbb Q^n} \left[ \psi(z^n,w^n)
\exp \left( \int_0^T \widetilde{b}(z^n_s, \mu_s^n) \mathrm d\widetilde{\beta}^n_s -\frac{1}{2} \int_0^T |\widetilde{b}(z^n_s, \mu_s^n)|^2 \mathrm ds \right)
 \right] \notag\\
&= \mathbb E^{\mathbb Q^n} \left[ \psi(z^n,w^n)
\exp \left( \int_0^T \widetilde{b}(z^n_s, \mu^\phi(w^n)_s) \mathrm d\widetilde{\beta}^n_s -\frac{1}{2} \int_0^T |\widetilde{b}(z^n_s, \mu^\phi(w^n)_s)|^2 \mathrm ds \right)
 \right] \\
 &= \mathbb E^{\mathbb Q^n} \left[ \widetilde{\psi}(z^n,w^n, \widetilde{\beta}^n)
 \right]
\end{align*}
where  $\widetilde{\psi}$ is a measurable map such that $\mathbb E^{\mathbb Q^n} \left[ |\widetilde{\psi}(z^n,w^n, \widetilde{\beta}^n)|
 \right]<+\infty$; the measurability of $\widetilde{\psi}$ follows from the fact that $\mu^\phi$ belongs to $\mathcal X$. Furthermore, $\mu^\phi$ does not depend on $n=1,2$, since it is the unique fixed-point of $\phi$. 
Recalling the equality $\mathcal L^{\mathbb Q^1}(z^1,w^1,\widetilde{\beta}^1)=\mathcal L^{\mathbb Q^2}(z^2,w^2,\widetilde{\beta}^2)$, we conclude that $\mathbb E^{\mathbb P^1} \left[ \psi(z^1,w^1) \right]=\mathbb E^{\mathbb P^2} \left[ \psi(z^2,w^2) \right]$. Moreover, by Lemma~\ref{lemme:mui_point_fixe}, $(\mu^n_t)_{t\in [0,T]}$ is $(\mathcal G_t^{W^n})_{t\in [0,T]}$-measurable. 
This completes the proof of Proposition~\ref{prop:drift_lipschitz_uniqu}. 
\end{proof}

\begin{proof}[Proof (Lemma~\ref{lemme:mui_point_fixe})]
Let us forget about the exponent $n$ in this proof. On the one hand, the process $(\mu_t)_{t\in [0,T]}$ satisfies for every $t\in [0,T]$, $\mu_t= \mathcal L^{\mathbb P} (z_t | \mathcal G_t^{\mu,W})$ and $(\mu,w)$ is independent of $(\beta,\xi)$. 
Moreover, it follows from equality~\eqref{qi_sur_pi} that $\mathbb P$ is absolutely continuous with respect to~$\mathbb Q$ with a density given by
\begin{align}
\label{dp_sur_dq}
\frac{\mathrm d\mathbb P}{\mathrm d\mathbb Q}
&= \exp \left( \int_0^T \widetilde{b}(z_s, \mu_s) \mathrm d\beta_s +\frac{1}{2} \int_0^T |\widetilde{b}(z_s, \mu_s)|^2 \mathrm ds \right) \notag \\
&= \exp \left( \int_0^T \widetilde{b}(z_s, \mu_s) \mathrm d\widetilde{\beta}_s -\frac{1}{2} \int_0^T |\widetilde{b}(z_s, \mu_s)|^2 \mathrm ds \right).
\end{align} 

On the other hand, since $\mu^\phi$ is the fixed point of $\phi$, the process $(\mu^\phi(w)_t)_{t\in [0,T]}$ satisfies $\mu^\phi(w)=\phi(\mu^\phi) (w)= \mathcal L^{\mathbb P^{\mu^\phi, \beta,\xi}(w)}( \mathcal Z(\cdot,w,\cdot))$. Since under $\mathbb Q$, $z=\mathcal Z (\xi,w,\widetilde{\beta})$, we deduce that for every $t\in [0,T]$, $\mu^\phi(w)_t=\mathcal L^{\mathbb R} (z_t |\mathcal G_t^W)$, where $\mathcal G_t^{W}:= \sigma\{ w(k,s) \;; k\in \R, s\leq t \}$
and $\mathbb R$ is defined by
\begin{align}
\label{dr_sur_dq}
\frac{\mathrm d\mathbb R}{\mathrm d\mathbb Q}
&= \exp \left( \int_0^T \widetilde{b}(z_s, \mu^\phi(w)_s) \mathrm d\widetilde{\beta}_s -\frac{1}{2} \int_0^T |\widetilde{b}(z_s, \mu^\phi(w)_s)|^2 \mathrm ds \right).
\end{align}

Let us prove that
\begin{itemize}
\item[$(1)$] under the probability measure $\mathbb Q$, $(\mu,w)$ is independent of $(\widetilde{\beta},\xi)$;
\item[$(2)$] for every $t\in [0,T]$, $\mathbb \R$-almost surely,  $\mu^\phi(w)_t=\mathcal L^{\mathbb R} (z_t |\mathcal G_t^{\mu,W})$;
\item[$(3)$] conclude the proof of the lemma by comparing, for every $t\in [0,T]$, $\mu_t= \mathcal L^{\mathbb P} (z_t | \mathcal G_t^{\mu,W})$ with $\mu^\phi(w)_t=\mathcal L^{\mathbb R} (z_t |\mathcal G_t^{\mu,W})$.
\end{itemize}

\textbf{Proof of $(1)$.}
By definition of a weak solution, under probability measure $\mathbb P$, $w$, $\beta$ and $\xi$ are independent random variables and $(\mu,w)$ is independent of $(\beta,\xi)$. 
Let us consider bounded and measurable functions $f: \R \to \R$ and $\psi: \mathcal C([0,T], \mathcal P_2(\R)) \times \mathcal C(\R \times [0,T] ,\R^2) \to \R$ and let $g:[0,T] \to \R$ be a deterministic square integrable function. 
Recalling that $\mathrm d\widetilde{\beta}_t =\mathrm d\beta_t+ \widetilde{b}(z_t,\mu_t) \mathrm dt$, let us compute
\begin{multline*}
\mathbb E^{\mathbb Q} \left[ \psi(\mu,w) f(\xi) \exp\Big(\textstyle\int_0^T g_s \mathrm d\widetilde{\beta}_s - \frac{1}{2}\int_0^T g_s^2 \mathrm ds \Big) \right] \\
\begin{aligned}
&=\mathbb E^{\mathbb P} \left[ \psi(\mu,w) f(\xi) \exp\Big(\textstyle\int_0^T g_s \mathrm d\widetilde{\beta}_s - \frac{1}{2}\int_0^T g_s^2 \mathrm ds \Big)  \exp \Big(\textstyle -\int_0^T \widetilde{b}(z_s, \mu_s) \mathrm d\beta_s -\frac{1}{2} \int_0^T |\widetilde{b}(z_s, \mu_s)|^2 \mathrm ds \Big)\right]\\
&=\mathbb E^{\mathbb P} \left[ \psi(\mu,w) f(\xi) \exp\Big(\textstyle\int_0^T (g_s- \widetilde{b}(z_s,\mu_s)) \mathrm d\beta_s - \frac{1}{2}\int_0^T |g_s- \widetilde{b}(z_s,\mu_s)|^2 \mathrm ds \Big)  \right].
\end{aligned}
\end{multline*}
We now show that the last line is in fact equal to $\mathbb E^{\mathbb P} [ \psi(\mu,w) f(\xi)]$.
By expanding the exponential martingale by It\^o's formula, it is in fact sufficient to prove that, for any 
$({\mathcal G}_{t})_{t\in [0,T]}$ progressively-measurable and square integrable process $(H_{t})_{t \in [0,T]}$,
the stochastic integral $\int_{0}^T H_{s} d \beta_{s}$ is orthogonal to 
$\psi(\mu,w) f(\xi)$ under ${\mathbb P}$. By a standard approximation, it is even sufficient to do so for simple processes
$(H_{t})_{t \in [0,T]}$. In other words, it suffices to prove that, for any $0 \leq t \leq t' \leq T$, for any $\mathcal G_t$-measurable square-integrable random variable $H_t$, 
\begin{equation*}
{\mathbb E}^{\mathbb P} \left[ \psi(\mu,w) f(\xi) H_{t} (\beta_{t'}- \beta_{t}) \right] =0.
\end{equation*}
By taking the conditional expectation given ${\mathcal G}_{t}$ in the expectation appearing in the left-hand side, it is sufficient to prove  that, for any $0 \leq t \leq t' \leq T$, 
\begin{align*}
{\mathbb E}^{\mathbb P} \left[ \psi(\mu,w) f(\xi) (\beta_{t'}- \beta_{t}) \Big| \mathcal G_t \right]=0.
\end{align*}
Thanks to the compatibility condition in Definition 
\ref{def:weak_solution_drift_general},
\begin{align*}
{\mathbb E}^{\mathbb P} \left[ \psi(\mu,w) f(\xi) (\beta_{t'}- \beta_{t}) \Big| \mathcal G_t \right]={\mathbb E}^{\mathbb P} \left[ \psi(\mu,w) f(\xi) \Big| \mathcal G_t \right]
{\mathbb E}^{\mathbb P} \left[  (\beta_{t'}- \beta_{t}) \Big| \mathcal G_t \right]=0
\end{align*}
because $\beta_{t'}-\beta_t$ is independent of $\mathcal G_t$. 
Therefore, 
\begin{align*}
\mathbb E^{\mathbb Q} \left[ \psi(\mu,w) f(\xi) \exp\Big(\textstyle\int_0^T g_s \mathrm d\widetilde{\beta}_s - \frac{1}{2}\int_0^T g_s^2 \mathrm ds \Big) \right]
=\mathbb E^{\mathbb P} \left[ \psi(\mu,w) f(\xi)\right].
\end{align*}

It follows that 
\begin{multline*}
\mathbb E^{\mathbb Q} \left[ \psi(\mu,w) f(\xi) \exp\Big(\textstyle\int_0^T g_s \mathrm d\widetilde{\beta}_s - \frac{1}{2}\int_0^T g_s^2 \mathrm ds \Big) \right]
=\mathbb E^{\mathbb P} \left[ \psi(\mu,w) \right]
\cdot \mathbb E^{\mathbb P} \left[  f(\xi)\right]
 \\\begin{aligned}
&=\mathbb E^{\mathbb Q} \left[ \psi(\mu,w) \right]
\cdot \mathbb E^{\mathbb Q} \left[  f(\xi)\right] \\
&=\mathbb E^{\mathbb Q} \left[ \psi(\mu,w) \right]
\cdot \mathbb E^{\mathbb Q} \left[  f(\xi)\right] \cdot 
\mathbb E^{\mathbb Q} \left[ \exp\Big(\textstyle\int_0^T g_s \mathrm d\widetilde{\beta}_s - \frac{1}{2}\int_0^T g_s^2 \mathrm ds \Big) \right],
\end{aligned}
\end{multline*}
since $\Big(\exp\Big(\textstyle\int_0^T g_s \mathrm d\widetilde{\beta}_s - \frac{1}{2}\int_0^T g_s^2 \mathrm ds\Big)\Big)_{t\in [0,T]}$ is a martingale under the measure $\mathbb Q$. 
Moreover, the linear span of $\{ \exp (\int_0^T g_s \mathrm d\widetilde{\beta}_s ), g \in L_2([0,T], \R)\}$ is dense in $L_2(\Omega,\mathcal G^{\widetilde{\beta}} , \mathbb Q)$, where $\mathcal G^{\widetilde{\beta}}$ is the $\sigma$-algebra generated by $(\widetilde{\beta}_t)_{t\in [0,T]}$. 
Therefore,  $(f(\xi),  \exp(\int_0^T g_s \mathrm d\widetilde{\beta}_s ))$ generates the $\sigma$-algebra $\mathcal G^{\xi,\widetilde{\beta}}$, and thus  $(\mu,w)$ and $(\xi,\widetilde{\beta})$ are independent under the probability measure~$\mathbb Q$.

\textbf{Proof of $(2)$.}
Recall that for every $t\in [0,T]$, $\mu^\phi(w)_t=\mathcal L^{\mathbb R} (z_t |\mathcal G_t^W)$, and let us prove that  for every $t\in [0,T]$, $\mathbb R$-almost surely,  $\mu^\phi(w)_t=\mathcal L^{\mathbb R} (z_t |\mathcal G_t^{\mu,W})$.
Let $f:\mathcal  \R \to \R$, $g:\mathcal C([0,T], \mathcal P_2(\R))\to \R$ and $h:\mathcal C(\R \times [0,T], \R^2) \to \R$ be bounded and measurable functions. Fix $t\in [0,T]$. By~\eqref{dr_sur_dq}, we have 
\begin{multline}
\mathbb E^{\mathbb R} \left[ f(z_t) g(\mu_{\cdot \wedge t}) h(w_{\cdot \wedge t})  \right] \\
\begin{aligned}
&=\mathbb E^{\mathbb Q} \left[ f(z_t) \exp \Big( \textstyle \int_0^t \widetilde{b}(z_s, \mu^\phi(w)_s) \mathrm d\widetilde{\beta}_s -\frac{1}{2} \int_0^t |\widetilde{b}(z_s, \mu^\phi(w)_s)|^2 \mathrm ds \Big) g(\mu_{\cdot \wedge t}) h(w_{\cdot \wedge t})  \right] \\
&=\mathbb E^{\mathbb Q} \left[ {\mathbb E}^{\mathbb Q} \left[ F_t \, \vert \, {\mathcal G}_{t}^{\mu,W} \right] \;g(\mu_{\cdot \wedge t}) \; h(w_{\cdot \wedge t})  \right],
\end{aligned}
\label{lem40a}
\end{multline}
where (recall that  $z$ has the form
$z=\mathcal Z (\xi,w,\widetilde{\beta})$)
\begin{align*}
F_t&:= f(z_t) \exp \Big( \textstyle \int_0^t \widetilde{b}(z_s, \mu^\phi(w)_s) \mathrm d\widetilde{\beta}_s -\frac{1}{2} \int_0^t |\widetilde{b}(z_s, \mu^\phi(w)_s)|^2 \mathrm ds \Big) \\
&=  f(\mathcal Z_t (\xi,w,\widetilde{\beta})) \exp \Big( \textstyle \int_0^t \widetilde{b}(\mathcal Z_s (\xi,w,\widetilde{\beta}), \mu^\phi(w)_s) \mathrm d\widetilde{\beta}_s -\frac{1}{2} \int_0^t |\widetilde{b}(\mathcal Z_s (\xi,w,\widetilde{\beta}), \mu^\phi(w)_s)|^2 \mathrm ds \Big) .
\end{align*}
Note that $F_t$ is $\mathcal G_t^{W,\beta,\xi}$-measurable. 
By statement $(1)$, under probability measure $\mathbb Q$, $(\mu,w)$ is independent of $(\widetilde{\beta},\xi)$.
Hence  ${\mathbb E}^{\mathbb Q} \left[ F_t \, \vert \, {\mathcal G}_{t}^{\mu,W} \right]$ is $\mathcal G_t^{W}$-measurable.
Thus it follows from~\eqref{lem40a} that
\begin{align*}
\mathbb E^{\mathbb R} \left[ f(z_t) g(\mu_{\cdot \wedge t}) h(w_{\cdot \wedge t})  \right]
= \mathbb E^{\mathbb Q} \left[ {\mathbb E}^{\mathbb Q} \left[ F_t \, \vert \, {\mathcal G}_{t}^{\mu,W} \right] \;\mathbb E^{\mathbb Q} \left[g(\mu_{\cdot \wedge t}) | \mathcal G_t^W \right]  h(w_{\cdot \wedge t})  \right].
\end{align*}
Since $\mathbb E^{\mathbb Q} \left[g(\mu_{\cdot \wedge t}) | \mathcal G_t^W \right] h(w_{\cdot \wedge t})$ is $\mathcal G_t^W$-measurable and bounded, there is a bounded and measurable function  $k:\mathcal C(\R \times [0,T], \R^2) \to \R$ such that $\mathbb E^{\mathbb Q} \left[g(\mu_{\cdot \wedge t}) | \mathcal G_t^W \right] h(w_{\cdot \wedge t})=k(w_{\cdot \wedge t})$.
Thus,  redoing the same computations in reverse, we obtain:
\begin{multline}
\mathbb E^{\mathbb R} \left[ f(z_t) g(\mu_{\cdot \wedge t}) h(w_{\cdot \wedge t})  \right] 
=\mathbb E^{\mathbb Q} \left[ {\mathbb E}^{\mathbb Q} \left[ F_t \, \vert \, {\mathcal G}_{t}^{\mu,W} \right] \;k(w_{\cdot \wedge t})  \right] 
=\mathbb E^{\mathbb Q} \left[  F_t  \;k(w_{\cdot \wedge t})  \right] 
=\mathbb E^{\mathbb R} \left[ f(z_t)  \;k(w_{\cdot \wedge t})  \right]\\
\begin{aligned}
&= \mathbb E^{\mathbb R} \left[\int_\R f(x) \mathrm d \mu^\phi(w)_t(x) \;k(w_{\cdot \wedge t})  \right] 
= \mathbb E^{\mathbb R} \left[\int_\R f(x) \mathrm d \mu^\phi(w)_t(x) \;\mathbb E^{\mathbb Q} \left[g(\mu_{\cdot \wedge t}) | \mathcal G_t^W \right] h(w_{\cdot \wedge t})  \right] \\
&= \mathbb E^{\mathbb R} \left[\int_\R f(x) \mathrm d \mu^\phi(w)_t(x) \; g(\mu_{\cdot \wedge t}) \; h(w_{\cdot \wedge t})  \right].
\end{aligned}
\label{lem40b}
\end{multline}
Since the process $(\mu^\phi(w)_t)_{t\in [0,T]}$ is $(\mathcal G_t^W)_{t\in [0,T]}$-adapted, it is in particular $(\mathcal G_t^{\mu,W})_{t\in [0,T]}$-adapted, thus equality~\eqref{lem40b} implies that $\mu^\phi(w)_t=\mathcal L^{\mathbb R} (z_t |\mathcal G_t^{\mu,W})$. 
It completes the proof of $(2)$. 

\textbf{Proof of $(3)$.}
Let us denote for every $t\in [0,T]$, $\nu_t=\mu^\phi(w)_t$. We want to prove that $(\mu_t)_{t\in [0,T]}=(\nu_t)_{t\in [0,T]}$. Recall that for every $t\in [0,T]$, $\mathbb Q$-almost surely,  $\mu_t= \mathcal L^{\mathbb P} (z_t | \mathcal G_t^{\mu,W})$ and, by point $(2)$, $\nu_t= \mathcal L^{\mathbb R} (z_t | \mathcal G_t^{\mu,W})$. By~\eqref{dp_sur_dq} and~\eqref{dr_sur_dq}, 
\begin{align*}
\frac{\mathrm d\mathbb P}{\mathrm d\mathbb Q}
&= \exp \left( \int_0^T \widetilde{b}(z_s, \mu_s) \mathrm d\widetilde{\beta}_s -\frac{1}{2} \int_0^T |\widetilde{b}(z_s, \mu_s)|^2 \mathrm ds \right); \\
\frac{\mathrm d\mathbb R}{\mathrm d\mathbb Q}
&= \exp \left( \int_0^T \widetilde{b}(z_s, \nu_s) \mathrm d\widetilde{\beta}_s -\frac{1}{2} \int_0^T |\widetilde{b}(z_s, \nu_s)|^2 \mathrm ds \right).
\end{align*}
Let us apply the same computation as in the proof of Lemma~\ref{lemme:inegalite_lacker}. Recall that for every $t\in [0,T]$,  $\mu^t$ denotes $\mathcal L^{\mathbb P}(z_{\cdot \wedge t} | \mathcal G_t^{\mu,W})$ and $\nu^t:= \mathcal L^{\mathbb R}(z_{\cdot \wedge t} | \mathcal G_t^{\mu,W})$.
For every $t\in [0,T]$, 
\begin{align*}
H(\nu^t | \mu^t)
= \int_\R \ln \frac{\mathrm d\nu^t}{\mathrm d\mu^t} \mathrm d\nu^t
= \mathbb E^{\mathbb R} \left[ \ln \frac{\mathrm d\nu^t}{\mathrm d\mu^t} (z _{\cdot \wedge t}) \Big| \mathcal G_t^{\mu,W}  \right].
\end{align*}
We use the fact that under $\mathbb P$, $\beta$ is independent of $(\mu,w)$ in order to prove, exactly as in the proof of Lemma~\ref{lemme:inegalite_lacker}, that for every $t\in [0,T]$, $\mathbb E^{\mathbb P} \left[ \frac{\mathrm d \mathbb R}{\mathrm d \mathbb P} \Big| \mathcal G_t^{\mu,W} \right]=1$. Again by mimicking the proof of~\eqref{lem38a}, this leads to 
\begin{align*}
\frac{\mathrm d\nu^t}{\mathrm d\mu^t} (z _{\cdot \wedge t}) 
=\mathbb E^{\mathbb P} \left[ \frac{\mathrm d\mathbb R}{\mathrm d\mathbb P} \Big| \mathcal G_t^{z,\mu,W} \right].
\end{align*}
Therefore, we finally obtain
\begin{align*}
H(\nu^t |\mu^t)
=\frac{1}{2}\mathbb E^{\mathbb R} \left[ \int_0^t |\widetilde{b}(z_s, \nu_s)-\widetilde{b}(z_s, \mu_s)|^2 \mathrm ds \Big| \mathcal G_t^{\mu,W}  \right].
\end{align*}
Applying Pinsker's inequality and using the fact that $\widetilde{b}$ is Lipschitz-continuous with respect to the measure variable, 
we finally obtain for every $t\in [0,T]$, 
\begin{align*}
\mathbb E^{\mathbb Q}  \left[\dtv (\nu_t, \mu_t)^2 \right]
\leq \mathbb E^{\mathbb Q}  \left[\dtv (\nu^t, \mu^t)^2 \right] 
\leq C \int_0^T \mathbb E^{\mathbb Q}\left[\dtv (\nu_s, \mu_s)^2 \right] \mathrm ds. 
\end{align*}
Thus by Gronwall's inequality, we obtain that for every $t\in [0,T]$, $\mathbb E ^{\mathbb Q} \left[\dtv (\nu_t, \mu_t)^2 \right]=0$. In particular, $\mathbb Q$-almost surely, the two continuous processes $(\mu_t)_{t\in [0,T]}$ and $(\nu_t)_{t\in [0,T]}$ are equal. This completes the proof of the lemma. 
\end{proof}

\subsection{Resolution of the SDE when the drift is general}
\label{section:resolution_general}

Let us state the well-posedness result for the general case:
\begin{theo}
\label{theo:interpolation_lacker}
Let $\eta > 0$ and  $\delta \in [0,1]$ satisfy the inequality $\eta  > \frac{3}{2}(1-\delta)$.
Let $b:\R \times \mathcal P_2(\R) \to \R$ be of class $(H^\eta, \mathcal C^\delta)$. Let $f: \R \to \R$ be a function of order $\alpha \in (\frac{3}{2}, \frac{\eta}{1-\delta}]$. 

Then existence of a weak solution and uniqueness in law hold for equation~\eqref{sde_beta_drift_general}. 
\end{theo}

\textit{Note:} The assumption on $\alpha$ is the same as the one given by inequality~\eqref{assumption_alpha}. 

As a first step, let us show that a drift function $b$ satisfying the assumptions of Theorem~\ref{theo:interpolation_lacker} can be written as  a sum $\widetilde{b}+(b-\widetilde{b})$, where $\widetilde{b}$ satisfies the assumptions of Proposition~\ref{prop:drift_lipschitz_exist} and where $b-\widetilde{b}$ satisfies assumptions similar to Definition~\ref{def:B hyp}, and apply on $b-\widetilde{b}$ the same Fourier inversion as in Lemma~\ref{lemme:fourier_masseconstante}.

Recall that by Definition~\ref{def:class_HC}, $b$ can be written as
\begin{align}
\label{Bxmu}
b(x,\mu)
&=\int_\R \crochetk^{-\eta} \Big( \cos (kx) \lambda^\Re(k,\mu) + \sin(kx) \lambda^\Im(k, \mu)  \Big) \mathrm dk,
\end{align}
where $\lambda=\lambda^\Re + i\lambda^\Im$ satisfies for every $k \in \R$ and for every $\mu, \nu \in \mathcal P_2(\R)$, 
\begin{align}
\label{lambda_borne}
|\lambda(k,\mu)| & \leq \Lambda(k); \\
\label{lambda_holder}
|\lambda(k,\mu)-\lambda(k,\nu)| &\leq \Lambda(k) \dtv(\mu, \nu)^\delta,
\end{align}
and $\Lambda$ belongs to $L_1(\R) \cap L_2(\R)$.

\begin{lemme}
\label{lemme:regularisation_alpha}
Let $\theta:= \frac{\alpha-\eta}{\delta}$. There exists $\widetilde{\lambda}= \widetilde{\lambda}^\Re+i \widetilde{\lambda}^\Im$, where $\widetilde{\lambda}^\Re, \widetilde{\lambda}^\Im: \R \times \mathcal P_2(\R) \to \R$, such that for each $k \in \R$ and for each $\mu, \nu \in \mathcal P_2(\R)$, 
\begin{align}
\label{lambda_moins_lambdatilde}
|\lambda(k,\mu)-\widetilde{\lambda}(k,\mu)| &\leq \frac{C}{\crochetk^{\theta \delta}} \Lambda(k) ;\\
\label{lambdatilde_lipschitz}
|\widetilde{\lambda}(k,\mu)-\widetilde{\lambda}(k,\nu)| &\leq C \crochetk^{\theta(1-\delta)} \Lambda(k) \dtv(\mu,\nu),
\end{align}
where $C$ is independent of $k$, $\mu$, $\nu$ and $\theta$. 
\end{lemme}

\begin{proof}
Let us fix $k\in \R$. We will focus on the proof for the real part; the case of the imaginary part is identical. 

Let us define $u:\mathcal P_2(\R) \to \R$ by $u(\mu):= \frac{\lambda^\Re(k, \mu)}{\Lambda(k)}$. By~\eqref{lambda_borne} and~\eqref{lambda_holder}, for every $\mu, \nu \in \mathcal P_2(\R)$, $|u(\mu)| \leq 1$ and $|u(\mu)-u(\nu)| \leq \dtv(\mu,\nu)^\delta$. Let $(\Omega, \mathcal F, \mathbb P)$ be a Polish and atomless probability space. Let us define $v:L_2(\Omega) \to \R$ by $v(X):= u(\mathcal L(X))$.

The following approximation method is inspired by the inf-convolution techniques. 
Let $\eps>0$. Let us define $v^\eps:L_2(\Omega) \to \R$ by
\begin{align}
\label{def:veps}
v^\eps (X):= \inf_{Y \in L_2(\Omega \times [0,1])} \left\{  v(Y) +\frac{1}{2\eps} (\mathbb P \otimes \Leb_{[0,1]}) \left[X \neq Y\right]^2 \right\}. 
\end{align}
We consider here the infimum over random variables in a larger probability space  in order to be enseure the existence of a random variable~$Y$ independent of $X$. In~\eqref{def:veps}, the map $v$ is extended to $L_2(\Omega \in [0,1]) \times \R$ by $v(Y):= u (\mathcal L(Y))$. 
Let us prove that
\begin{itemize}
\item[$(i)$] $v^\eps(X)$ depends only on the law of $X$; thus we can define $u^\eps (\mu)$ by letting  $u^\eps (\mu):=v^\eps (X)$, whatever the choice of the random variable $X$ with distribution $\mu$. 
\item[$(ii)$] for every $\mu \in \mathcal P_2(\R)$, $|u^\eps (\mu)-u(\mu) | \leq C \eps^{\frac{\delta}{2-\delta}}$. 
\item[$(iii)$] for every $\mu,\nu \in \mathcal P_2(\R)$, $|u^\eps (\mu)-u^\eps (\nu) | \leq C \eps^{\frac{\delta-1}{2-\delta}} \dtv(\mu,\nu)$. 
\end{itemize}

\textbf{Proof of $(i)$.} Let $X,X' \in L_2(\Omega)$ with same law. We want to prove that $v^\eps (X)=v^\eps (X')$. Remark that by definition of $v$, $v(X)$ depends only on the law of $X$. 
Fix $\eta>0$. There is $Y^\eta \in L_2(\Omega \times [0,1])$ such that
\begin{align}
\label{v_Yeta}
v(Y^\eta)+\frac{1}{2\eps } (\mathbb P \otimes \Leb_{[0,1]}) \left[X \neq Y^\eta\right]^2 \leq v^\eps (X) +\eta. 
\end{align}
Let $\nu:\R \times \mathcal B(\R) \to \R$ be the conditional law of $Y^\eta$ given $X$; in other words, for every fixed $x\in \R$, $\nu(x,\cdot)$ belongs to $\mathcal P_2(\R)$, for every fixed $A \in \mathcal B(\R)$, $x \mapsto\nu (x,A)$ is measurable and for every $f:\R^2\to \R$ bounded and measurable, $(\mathbb E \otimes \Leb_{[0,1]}) \left[f(X,Y^\eta)\right]=(\mathbb E \otimes \Leb_{[0,1]}) \left[\int_\R f(X,y) \nu(X,\mathrm dy)\right]$. 

Furthermore, for every fixed $x\in \R$, let us denote by $u\in [0,1] \mapsto
g(x,u)$ the quantile function associated to the probability measure $\nu(x,\cdot)$. 
For every $t\in \R$ and for every $u\in [0,1]$, $\{ x: g(x,u) \leq t \}=\{ x: \nu(x,(-\infty,t]) \leq u \} \in \mathcal B(\R)$, so we deduce that for every $u\in [0,1]$, $x \mapsto g(x,u)$ is measurable. Moreover, $u\mapsto g(x,u)$ is a càdlàg function. It follows from~\cite[Proposition 1.13]{karatzasshreve91} that $(x,u) \mapsto g(x,u)$ is measurable. 

Let $U \in L_2([0,1])$ be a random variable with  uniform law on $[0,1]$; in particular, it is independent of  $X'$ (remark that we have considered a larger probability space in order to ensure the existence of $U$ independent of $X'$). Let $Y':= g(X',U)$. Then for every $f:\R^2\to \R$ bounded and measurable
\begin{align*}
(\mathbb E \otimes \Leb_{[0,1]}) \left[f(X',Y')\right]
&=(\mathbb E \otimes \Leb_{[0,1]}) \left[f(X', g(X',U))\right] \\
&=(\mathbb E \otimes \Leb_{[0,1]}) \left[\int_0^1 f(X', g(X',u)) \mathrm du\right] \\
&=(\mathbb E \otimes \Leb_{[0,1]}) \left[\int_\R f(X',y) \nu(X', \mathrm dy)\right].
\end{align*}
Since $X$ and $X'$ have same law, we deduce that
\begin{align*}
(\mathbb E \otimes \Leb_{[0,1]}) \left[f(X',Y')\right]=(\mathbb E \otimes \Leb_{[0,1]}) \left[\int_\R f(X,y) \nu(X, \mathrm dy)\right]= (\mathbb E \otimes \Leb_{[0,1]}) \left[f(X,Y^\eta)\right].
\end{align*}
Therefore, the pair $(X',Y')$ has same distribution as $(X, Y^\eta)$. 
It follows that 
\begin{align*}
(\mathbb P \otimes \Leb_{[0,1]}) \left[X \neq Y^\eta\right]=(\mathbb P \otimes \Leb_{[0,1]}) \left[X' \neq Y'\right]
\end{align*}
and $v(Y^\eta)=v(Y')$ since $v$ depends only on the law of the random variable. 
Thus by inequality~\eqref{v_Yeta}, 
\begin{align*}
v(Y')+\frac{1}{2\eps } (\mathbb P \otimes \Leb_{[0,1]}) \left[X' \neq Y'\right]^2 \leq v^\eps (X) +\eta. 
\end{align*}
By definition~\eqref{def:veps} of $v^\eps$, $v^\eps (X') \leq v(Y')+\frac{1}{2\eps } (\mathbb P \otimes \Leb_{[0,1]}) \left[X' \neq Y'\right]^2$, thus $v^\eps (X') \leq v^\eps (X) +\eta$. We proved that the inequality holds with every $\eta>0$, thus $v^\eps (X') \leq v^\eps (X)$. By symmetry, $v^\eps (X) \leq v^\eps (X')$, hence  the equality holds true.

\textbf{Proof of $(ii)$.}
Let us prove that for every $X \in L_2(\Omega)$, $|v^\eps (X) - v(X) | \leq C \eps^{\frac{\delta}{2-\delta}}$. 
Fix $X \in L_2(\Omega)$. By definition~\eqref{def:veps}, it is obvious that $v^\eps(X) \leq v(X)$. Thus it is sufficient to prove that $v(X)-v^\eps(X) \leq C \eps^{\frac{\delta}{2-\delta}}$.

Fix $\eta >0$. There exists $Y^\eta$ such that~\eqref{v_Yeta}. It follows that
\begin{align*}
v(Y^\eta)+\frac{1}{2\eps } (\mathbb P \otimes \Leb_{[0,1]}) \left[X \neq Y^\eta\right]^2 \leq v (X) +\eta. 
\end{align*}
By definition of $v$, $|v(X)-v(Y^\eta)| = |u(\mathcal L(X)) - u(\mathcal L(Y^\eta))| \leq \dtv(\mathcal L(X),\mathcal L(Y^\eta))^\delta$. 
Therefore, by~\eqref{def:dtv}, 
\begin{align}
\label{dtv_yeta_1}
\dtv(\mathcal L(X),\mathcal L(Y^\eta))^2
\leq 4 (\mathbb P \otimes \Leb_{[0,1]}) \left[X \neq Y^\eta\right]^2 
\leq 8\eps  \left[\dtv(\mathcal L(X),\mathcal L(Y^\eta))^\delta + \eta\right]. 
\end{align}
Let $l:=\limsup_{\eta \searrow 0} \dtv(\mathcal L(X),\mathcal L(Y^\eta))$. Thus $l^2 \leq 8\eps l^\delta$, hence we get $l^{2-\delta} \leq 8\eps$. 
It follows  that
\begin{align}
\label{dtv_yeta_2}
\limsup_{\eta \searrow 0} \dtv(\mathcal L(X),\mathcal L(Y^\eta))
\leq 2^{\frac{3}{2-\delta}} \eps^{\frac{1}{2-\delta}}. 
\end{align}
By inequality~\eqref{v_Yeta}, 
\begin{align*}
v(X)-v^\eps(X) 
&\leq v(X) - v(Y^\eta)- \frac{1}{2\eps } (\mathbb P \otimes \Leb_{[0,1]}) \left[X \neq Y^\eta\right]^2 +\eta \\
&\leq |v(X) - v(Y^\eta)| +\eta
\leq \dtv(\mathcal L(X),\mathcal L(Y^\eta))^\delta + \eta. 
\end{align*}
By passing to the limit $\eta \searrow 0$, we obtain $v(X)-v^\eps(X) \leq C \eps^{\frac{\delta}{2-\delta}}$, which completes the proof of $(ii)$. 

\textbf{Proof of $(iii)$.}
Let us first prove that $u^\eps$ is also $\delta$-Hölder continuous. 
Let $\mu, \nu \in \mathcal P_2(\R)$. Let $X$ and $X' \in L_2(\Omega)$ with respective distributions $\mu$ and $\nu$. 
Fix $\eta>0$. Let $Y^\eta$ satisfying~\eqref{v_Yeta}. Then
\begin{align*}
v^\eps(X')-v^\eps(X) 
&\leq v^\eps (X') -  v(Y^\eta)-\frac{1}{2\eps } (\mathbb P \otimes \Leb_{[0,1]}) \left[X \neq Y^\eta\right]^2+\eta \\
&\leq v(Y^\eta+ X'-X) + \frac{1}{2\eps } (\mathbb P \otimes \Leb_{[0,1]}) \left[X' \neq Y^\eta +X'-X\right]^2 \\
& \quad \quad -  v(Y^\eta)-\frac{1}{2\eps } (\mathbb P \otimes \Leb_{[0,1]}) \left[X \neq Y^\eta\right]^2+\eta \\
&\leq |v(Y^\eta+ X'-X) -  v(Y^\eta) |+\eta 
\leq \dtv(\mathcal L(Y^\eta+ X'-X),\mathcal L(Y^\eta))^\delta + \eta.
\end{align*}
By definition of the distance in total variation, $\dtv(\mathcal L(Y^\eta+ X'-X),\mathcal L(Y^\eta)) \leq 2 \P{X' \neq X}$. Thus for every $\eta>0$, $v^\eps(X')-v^\eps(X)  \leq 2^\delta \P{X' \neq X}^\delta + \eta$. By letting $\eta$ tend to zero and by symmetry, we deduce that there is $C>0$ depending only on $\delta$ such that
\begin{align*}
|u^\eps(\mu)-u^\eps(\nu)| = |v^\eps(X)-v^\eps(X') | \leq C \P{X' \neq X}^\delta.
\end{align*}
By taking the infimum over every coupling $(X,X')$ of $(\mu,\nu)$, we finally get
\begin{align}
\label{ue_holder}
|u^\eps(\mu)-u^\eps(\nu)| \leq C \dtv(\mu,\nu)^\delta. 
\end{align}
Therefore, $u^\eps$ is also $\delta$-Hölder continuous. 

Keep $X,X' \in L_2(\Omega)$ two random variables with laws $\mu$ and $\nu$. Let $(Y^\eta)_{\eta>0}$ satisfy~\eqref{v_Yeta}. 
It follows from~\eqref{dtv_yeta_1} and~\eqref{dtv_yeta_2} that
\begin{align*}
\limsup_{\eta \searrow 0}  (\mathbb P \otimes \Leb_{[0,1]}) \left[X \neq Y^\eta\right]
\leq \sqrt{2\eps 2^{\frac{3\delta}{2-\delta}} \eps^{\frac{\delta}{2-\delta}}}
= 2^{\frac{1+\delta}{2-\delta}} \eps^{\frac{1}{2-\delta}}. 
\end{align*}
For every $\eta>0$, let us define
\begin{align*}
S^\eta := \bigg\{ Y \in L_2(\Omega \times [0,1]) :  (\mathbb P \otimes \Leb_{[0,1]}) \left[X \neq Y\right] &\leq 2^{\frac{1+\delta}{2-\delta}} \eps^{\frac{1}{2-\delta}}+\eta \\ \text{or} \quad (\mathbb P \otimes \Leb_{[0,1]}) \left[X' \neq Y\right] &\leq 2^{\frac{1+\delta}{2-\delta}} \eps^{\frac{1}{2-\delta}}+\eta  \bigg\}.
\end{align*}
Fix $\eta>0$. Thus there is $Y^\eta \in S^\eta$ such that~\eqref{v_Yeta} holds true. We deduce that
\begin{align*}
v^\eps(X') - v^\eps (X) 
&\leq v(Y^\eta) + \frac{1}{2\eps} (\mathbb P \otimes \Leb_{[0,1]}) \left[X' \neq Y^\eta\right]^2  \\
&\quad \quad -v(Y^\eta) - \frac{1}{2\eps} (\mathbb P \otimes \Leb_{[0,1]}) \left[X \neq Y^\eta\right]^2 +\eta \\
&\leq \frac{1}{2\eps} \left(  (\mathbb P \otimes \Leb_{[0,1]}) \left[X' \neq Y^\eta\right]^2 - (\mathbb P \otimes \Leb_{[0,1]}) \left[X \neq Y^\eta\right]^2 \right) +\eta. 
\end{align*}
By symmetry, we deduce that
\begin{align*}
|v^\eps(X') - v^\eps (X) | 
\leq \frac{1}{2\eps} \sup_{Y \in S^\eta} \left|  (\mathbb P \otimes \Leb_{[0,1]}) \left[X' \neq Y\right]^2 - (\mathbb P \otimes \Leb_{[0,1]}) \left[X \neq Y\right]^2 \right| +\eta. 
\end{align*}
Moreover,  
\begin{multline*}
\left|  (\mathbb P \otimes \Leb_{[0,1]}) \left[X' \neq Y\right]^2 - (\mathbb P \otimes \Leb_{[0,1]}) \left[X \neq Y\right]^2 \right| \\
\begin{aligned}
 & \leq  \left|  (\mathbb P \otimes \Leb_{[0,1]}) \left[X' \neq Y\right]- (\mathbb P \otimes \Leb_{[0,1]}) \left[X \neq Y\right] \right| \\
&\quad \quad \cdot \left|  (\mathbb P \otimes \Leb_{[0,1]}) \left[X' \neq Y\right]+ (\mathbb P \otimes \Leb_{[0,1]}) \left[X \neq Y\right] \right|\\
 & \leq   \mathbb P \left[X' \neq X\right]
 \cdot \left|  (\mathbb P \otimes \Leb_{[0,1]}) \left[X' \neq Y\right]+ (\mathbb P \otimes \Leb_{[0,1]}) \left[X \neq Y\right] \right|.
\end{aligned}
\end{multline*}
For every $Y\in S^ \eta$, we have
\begin{align*}
\left|  (\mathbb P \otimes \Leb_{[0,1]}) \left[X' \neq Y\right]+ (\mathbb P \otimes \Leb_{[0,1]}) \left[X \neq Y\right] \right|
\leq \mathbb P  \left[X \neq X'\right] +2 \left(  2^{\frac{1+\delta}{2-\delta}} \eps^{\frac{1}{2-\delta}}+\eta \right). 
\end{align*}
By passing to the limit $\eta \searrow 0$ it follows that
there exists $C>0$ depending on $\delta$ such that for every $X,X' \in L_2(\Omega)$ with respective distributions $\mu$ and $\nu$, 
\begin{align*}
|u^\eps (\mu)-u^\eps (\nu)| 
=|v^\eps(X) - v^\eps (X') | 
\leq \frac{1}{2\eps}
\mathbb P  \left[X' \neq X\right] \left(\mathbb P  \left[X' \neq X\right] + C \eps^{\frac{1}{2-\delta}}  \right). 
\end{align*}
Let us distinguish two cases:
\begin{itemize}
\item if $\dtv(\mu,\nu) < \eps^{\frac{1}{2-\delta}}$: by definition~\eqref{def:dtv}, there exists a coupling $(X,X')$ of law $(\mu,\nu)$ such that $\P{X \neq X'} < \eps^{\frac{1}{2-\delta}}$. Thus
\begin{align*}
|u^\eps (\mu)-u^\eps (\nu)| 
\leq C \frac{\eps^{\frac{1}{2-\delta}}}{\eps} \P{X' \neq X}
\leq C \eps^{\frac{\delta-1}{2-\delta}} \dtv(\mu,\nu).
\end{align*}
\item if $\dtv(\mu,\nu) \geq \eps^{\frac{1}{2-\delta}}$: recall that $u^\eps$ is $\delta$-Hölder continuous (see~\eqref{ue_holder}). 
Thus
\begin{align*}
|u^\eps (\mu)-u^\eps (\nu)|  \leq C \dtv(\mu,\nu)^\delta
\leq C \frac{\dtv(\mu,\nu)}{\dtv(\mu,\nu)^{1-\delta}}
\leq C \eps^{\frac{\delta-1}{2-\delta}} \dtv(\mu,\nu).
\end{align*}
\end{itemize}
This completes the proof of $(iii)$.

Let us conclude the proof of Lemma~\ref{lemme:regularisation_alpha}. Let us define $\widetilde{\lambda}^\Re (k,\mu):= \Lambda(k) u^\eps (\mu)$, with $\eps=\frac{1}{\crochetk^{\theta(2-\delta)}}$. For every $\mu, \nu \in \mathcal P_2(\R)$, we have
\begin{align*}
|\widetilde{\lambda}^\Re (k,\mu) -\lambda^\Re (k,\mu)|
&\leq \Lambda(k)|u^\eps (\mu) - u(\mu) | 
\leq C \Lambda(k) \eps^{\frac{\delta}{2-\delta}}
\leq C \Lambda(k) \frac{1}{\crochetk^{\theta\delta}} ; \\
|\widetilde{\lambda}^\Re (k,\mu) -\widetilde{\lambda}^\Re (k,\nu)|
&\leq \Lambda(k)|u^\eps (\mu) - u^\eps (\nu) |
\leq  C \Lambda(k) \eps^{\frac{\delta-1}{2-\delta}} \dtv(\mu,\nu) 
\leq C \Lambda(k) \crochetk^{\theta(1-\delta)} \dtv(\mu,\nu). 
\end{align*}
 It completes the proof of~\eqref{lambda_moins_lambdatilde} and~\eqref{lambdatilde_lipschitz} for the case of the real part. The proof for the imaginary part is the same. 
\end{proof}
In particular, it follows from~\eqref{lambda_borne} and from~\eqref{lambda_moins_lambdatilde} that there is $C>0$ such that for each $k \in \R$, $| \widetilde{\lambda}(k, \cdot)| \leq C \Lambda(k)$. 

Let us define 
\begin{align}
\label{Btildxmu}
\widetilde{b}(x,\mu):= \int_\R \crochetk^{-\eta} \Big( \cos (kx) \widetilde{\lambda}^\Re(k,\mu) + \sin(kx) \widetilde{\lambda}^\Im(k, \mu)  \Big) \mathrm dk.
\end{align}
For every $x\in \R$ and $\mu \in \mathcal P_2(\R)$, 
\begin{align*}
| \widetilde{b}(x,\mu) |
\leq C \int_\R \crochetk^{-\eta} \big(|\widetilde{\lambda}^\Re(k,\mu)| + |\widetilde{\lambda}^\Im(k,\mu)|  \big) \mathrm dk
 \leq C \int_\R \crochetk^{-\eta} \Lambda(k) \mathrm dk
\leq C,
\end{align*}
since  $\eta  > 0$ and $\Lambda \in L_1(\R)$. 
Furthermore, by~\eqref{lambdatilde_lipschitz}, for every $x\in \R$ and for every $\mu, \nu \in \mathcal P_2(\R)$, 
\begin{align*}
| \widetilde{b}(x, \mu) - \widetilde{b}(x, \nu) |
&\leq \int_\R \crochetk^{-\eta} \Big(|\widetilde{\lambda}^\Re(k,\mu)-\widetilde{\lambda}^\Re(k,\nu)|  +|\widetilde{\lambda}^\Im(k,\mu)-\widetilde{\lambda}^\Im(k,\nu)|  \Big)\mathrm dk\\
&\leq C \int_\R \crochetk^{-\eta} \crochetk^{\theta(1-\delta)} \Lambda(k) \mathrm dk\; \dtv(\mu,\nu). 
\end{align*}
Moreover, $\eta-\theta(1-\delta) \geq 0$. Indeed, $\eta-\theta(1-\delta)=\eta+\theta \delta -\theta = \alpha-\frac{\alpha-\eta}{\delta}=\frac{\eta-\alpha(1-\delta)}{\delta} \geq 0$  by inequality~\eqref{assumption_alpha}.
Since $\Lambda$ belongs to $L_1(\R)$, it implies that $\int_\R \crochetk^{-\eta} \crochetk^{\theta(1-\delta)} \Lambda(k) \mathrm dk<+\infty$.
Therefore, the drift function $\widetilde{b}$ is uniformly bounded and uniformly Lipschitz-continuous in the measure variable.

\subsubsection{Existence of a weak solution to the SDE with drift function \texorpdfstring{$b$}{b}.}

Let us prove existence of a weak solution to equation~\eqref{sde_beta_drift_general}. 
We follow the same idea as in Theorem~\ref{theo:weak_wp_hypotheses_simplifiees}. 

\begin{proof}[Proof (Theorem~\ref{theo:interpolation_lacker}, existence part)]
Let $\mathbf{\Omega}$ be a weak solution to the SDE~\eqref{sde_beta_drift_lipschitz} with drift~$\widetilde{b}$ given by~\eqref{Btildxmu}. 
In particular, $\mathbb P$-almost surely and for every $t\in [0,T]$, 
\begin{align*}
z_t= \xi + \int_0^t \!\! \int_\R f(k) \Re( e^{-ikz_s} \mathrm dw(k,s)) +\beta_t+\int_0^t \widetilde{b}(z_s,\mu_s) \mathrm ds, 
\end{align*}
where for every $t\in [0,T]$, $\mu_t= \mathcal L^{\mathbb P} (z_t | \mathcal G_t^W)$ and $\mathcal G_t^W:= \sigma\{ w(k,s), k\in \R, s\leq t \}$. Recall that Proposition~\ref{prop:drift_lipschitz_uniqu} states that every weak solution has this form, \textit{i.e.} $(\mu_t)_{t\in [0,T]}$ is adapted to (the completion of) $(\mathcal G^W_t)_{t\in [0,T]}$. 

Let $(h_t)_{t\in [0,T]}=(h^\Re_t+i h^\Im_t)_{t\in [0,T]}$ be a  process with values in $L_2(\R, \C)$ satisfying for every $t\in [0,T]$ and for every $x\in \R$,  
\begin{align}
\label{fourier_b-btild}
(b-\widetilde{b})(x, \mu_t) = \int_\R f(k) \Big( \cos (kx) h_t^\Re(k) + \sin(kx) h_t^\Im(k)  \Big) \mathrm dk.
\end{align}
By~\eqref{Bxmu} and~\eqref{Btildxmu}, the unique solution to~\eqref{fourier_b-btild} is given,  for every $k \in \R$ and for every $t\in [0,T]$, by
\begin{align*}
h_t(k)=\frac{1}{f(k)} \crochetk^{-\eta} (\lambda(k,\mu_t) - \widetilde{\lambda}(k, \mu_t)).
\end{align*}
Since $\mu_t$ is a $(\mathcal G_t^W)_{t\in [0,T]}$-adapted process, the process $(h_t)_{t\in [0,T]}$ is also $(\mathcal G_t^W)_{t\in [0,T]}$-adapted. 
Furthermore, by~\eqref{lambda_moins_lambdatilde} and since $f$ is of order $\alpha$, 
\begin{align*}
\int_0^T \!\!\int_\R |h_t(k)|^2  \mathrm dk \mathrm dt
&\leq C \int_0^T \!\!\int_\R  \crochetk^{2\alpha -2\eta}  |\lambda(k,\mu_t) - \widetilde{\lambda}(k, \mu_t)|^2    \mathrm dk \mathrm dt \\
&\leq C \int_0^T \!\!\int_\R  \crochetk^{2\alpha -2\eta-2\theta \delta} \Lambda(k)^2   \mathrm dk \mathrm dt.
\end{align*}
Since $\alpha=\eta+\theta \delta$ and $\Lambda \in L_2(\R)$,  we deduce that $\int_0^T \!\!\int_\R |h_t(k)|^2  \mathrm dk \mathrm dt$ is bounded by a deterministic constant. 
Therefore, the  measure $\mathbb Q$ on $(\Omega, \mathcal G)$ with the following density with respect to~$\mathbb P$:
\begin{align*}
\frac{\mathrm d \mathbb Q}{\mathrm d \mathbb P}
=\exp \left( \int_0^T \!\! \int_\R h^\Re_t(k) \mathrm dw^\Re(k,t)+\int_0^T \!\! \int_\R h^\Im_t(k) \mathrm dw^\Im(k,t) - \frac{1}{2}  \int_0^T \!\! \int_\R |h_t(k)|^2 \mathrm dk \mathrm dt \right)
\end{align*}
is a probability measure. Let us define $\widetilde{w}(k,t)=\widetilde{w}^\Re(k,t)+ i \widetilde{w}^\Im(k,t)$, where
\begin{align*}
\widetilde{w}^\Re(k,t) &:=\wR(k,t)  - \int_0^t \!\!\int_0^k h^\Re_s(l) \; \mathrm dl \mathrm ds, \\
\widetilde{w}^\Im(k,t) &:=\wI(k,t)  - \int_0^t \!\!\int_0^k h^\Im_s(l) \; \mathrm dl \mathrm ds.
\end{align*}
By Girsanov's Theorem, $\mathcal L^{\mathbb Q} (\widetilde{w},\beta, \xi) = \mathcal L^{\mathbb P} (w,\beta, \xi)$
and for any $t \in [0,T]$, the $\sigma$-field $\sigma \{\widetilde w(k,t')-\widetilde w(k,t),\beta_{t'}- \beta_{t}, \ k \in \R, \ t' \in [t,T] \}$ is independent of ${\mathcal G}_{t}$ under ${\mathbb Q}$.  
Moreover, $\mathbb Q$-almost surely, the process $(z_t)_{t\in [0,T]}$ satisfies
\begin{align*}
z_t&= \xi + \int_0^t \!\! \int_\R f(k) \Re( e^{-ikz_s} \mathrm d\widetilde{w}(k,s)) 
+\int_0^t \!\! \int_\R f(k) \Big( \cos (kz_s) h_t^\Re (k) + \sin (kz_s) h_t^\Im (k) \Big) \mathrm dk \mathrm ds \\ 
&\quad \quad +\beta_t+\int_0^t \widetilde{b}(z_s,\mu_s) \mathrm ds \\
&= \xi + \int_0^t \!\! \int_\R f(k) \Re( e^{-ikz_s} \mathrm d\widetilde{w}(k,s)) 
+\int_0^t (b-\widetilde{b})(z_s, \mu_s) \mathrm ds  +\beta_t+\int_0^t \widetilde{b}(z_s,\mu_s) \mathrm ds \\
&= \xi + \int_0^t \!\! \int_\R f(k) \Re( e^{-ikz_s} \mathrm d\widetilde{w}(k,s)) +
\beta_t+\int_0^t b(z_s,\mu_s) \mathrm ds.
\end{align*}

Furthermore, recall that for every $t\in [0,T]$ $\mathbb P$-almost surely, $\mu_t= \mathcal L^{\mathbb P} (z_t | \mathcal G_t^W)$. We want to prove that for every $t\in [0,T]$ $\mathbb Q$-almost surely, $\mu_t= \mathcal L^{\mathbb Q} (z_t | \mathcal G_t^{\mu,\widetilde{W}})$, where the filtration $(\mathcal G_t^{\mu,\widetilde{W}})_{t\in [0,T]}$ is defined by  $\mathcal G_t^{\mu,\widetilde{W}}= \sigma\{ \widetilde{w}(k,s),\mu_s \;; k\in \R, s\leq t \}$. 
Let $\psi:\R \to \R$ and $\varphi: \mathcal C([0,T],\mathcal P_2(\R)) \times \mathcal C(\R \times [0,T] , \R^2) \to \R$ be bounded and measurable functions.
Fix $t\in [0,T]$. Then
\begin{multline*}
\mathbb E^{\mathbb Q} \left[ \psi(z_t) \; \varphi(\mu_{\cdot \wedge t},\widetilde{w}_{\cdot \wedge t}) \right] \\
= \mathbb E^{\mathbb P} \left[ \psi(z_t) \;\varphi(\mu_{\cdot \wedge t},\widetilde{w}_{\cdot \wedge t})\exp \Big(  \int_0^t \!\! \int_\R \Re \big( \overline{ h_t(k)} \mathrm dw(k,t)\big) - \frac{1}{2}  \int_0^t \!\! \int_\R |h_t(k)|^2 \mathrm dk \mathrm dt \Big) \right].
\end{multline*}
Recall that the process $(h_t)_{t\in [0,T]}$ is  $(\mathcal G_t^W)_{t\in [0,T]}$-adapted. It follows that the process $(\widetilde{w}_{\cdot \wedge t})_{t\in [0,T]}$ is also $(\mathcal G_t^W)_{t\in [0,T]}$-adapted, since $\widetilde{w}_{\cdot \wedge t} = w_{\cdot \wedge t} - \int_0^{\cdot\wedge t} \int_0^\cdot h_s(l) \mathrm dl \mathrm ds$. Thus
\begin{multline}
\mathbb E^{\mathbb Q} \left[ \psi(z_t) \; \varphi(\mu_{\cdot \wedge t},\widetilde{w}_{\cdot \wedge t}) \right] \\
\begin{aligned}
&= \mathbb E^{\mathbb P} \left[ \mathbb E^{\mathbb P} \left[ \psi(z_t)| \mathcal G_t^W \right] \varphi(\mu_{\cdot \wedge t},\widetilde{w}_{\cdot \wedge t}) \exp \Big(   \int_0^t \!\!  \int_\R \Re \big( \overline{ h_t(k)} \mathrm dw(k,t)\big) - \frac{1}{2}  \int_0^t \!\! \int_\R |h_t(k)|^2 \mathrm dk \mathrm dt \Big) \right]  \\
&= \mathbb E^{\mathbb P} \left[\int_\R \psi (x) \mathrm d\mu_t(x)  \; \varphi(\mu_{\cdot \wedge t},\widetilde{w}_{\cdot \wedge t}) \exp \Big(  \int_0^t \!\!  \int_\R \Re \big( \overline{ h_t(k)} \mathrm dw(k,t)\big) - \frac{1}{2}  \int_0^t  \!\!\int_\R |h_t(k)|^2 \mathrm dk \mathrm dt \Big) \right]  \\
&= \mathbb E^{\mathbb Q} \left[\int_\R \psi (x) \mathrm d\mu_t(x)  \; \varphi(\mu_{\cdot \wedge t},\widetilde{w}_{\cdot \wedge t}) \right].
\end{aligned}
\label{preuvetheo41}
\end{multline}
Therefore, for every $t\in [0,T]$, $\mathbb E^{\mathbb Q} \left[ \psi(z_t) | \mathcal G_t^{\mu,\widetilde{W}} \right] = \int_\R \psi(x) \mathrm d\mu_t(x)$. We deduce that for every $t\in [0,T]$, $\mathbb Q$-almost surely $\mu_t= \mathcal L^{\mathbb Q} (z_t | \mathcal G_t^{\mu,\widetilde{W}})$.

Furthermore, the pair $(\mu, \widetilde{w})$ is $\mathcal G^W$-measurable and, subsequently, 
$\frac{d {\mathbb Q}}{d {\mathbb P}}$ is also ${\mathcal G}^W$-measurable.
By independence of $(\xi,w,\beta)$ under ${\mathbb P}$, we deduce that 
$(\mu,\widetilde w)$ and $(\beta,\xi)$ are independent under ${\mathbb Q}$. 
By the same argument and by the compatibility property under ${\mathbb P}$, we deduce that, under ${\mathbb Q}$, for any $t \in [0,T]$, 
$(\xi,\widetilde w,\mu)$ and $\beta$ are conditionally independent given ${\mathcal G}_{t}$, which is the required compatibility 
condition.

Therefore $(\Omega, \mathcal G, (\mathcal G_t)_{t\in [0,T]}, \mathbb Q, z,\widetilde{w}, \beta, \xi)$ is a weak solution to~\eqref{sde_beta_drift_general}. 
This proves the first statement of Theorem~\ref{theo:interpolation_lacker}. 
\end{proof}

\begin{rem}
\label{rem:adaptness_bis}
In Remark~\ref{rem:adaptness}, we emphasized the importance of the filtration under which $(\mu_t)_{t\in [0,T]}$ is adapted. It makes sense in~\eqref{preuvetheo41}, because in order to identify $\int_\R \psi(x) \mathrm d\mu_t(x)$ with the conditional expectation $\mathbb E^{\mathbb Q} \left[ \psi(z_t) | \mathcal G_t^{\mu,\widetilde{W}} \right]$, we need to know that $\mu_t$ is $\mathcal G_t^{\mu,\widetilde{W}}$-measurable. This is obviously true, but it is not necessarily true with $\mathcal G_t^{\widetilde{W}}$ instead of $\mathcal G_t^{\mu,\widetilde{W}}$. 
\end{rem}

\subsubsection{Uniqueness in law for the SDE with drift function \texorpdfstring{$b$}{b}.}

Let us conclude the proof of Theorem~\ref{theo:interpolation_lacker} by showing uniqueness in law for equation~\eqref{sde_beta_drift_general}. 
We follow the same idea as in Theorem~\ref{theo:weak_wp_hypotheses_simplifiees}.

\begin{proof}[Proof (Theorem~\ref{theo:interpolation_lacker}, uniqueness part)]
Let $\mathbf{\Omega^1}$ and $\mathbf{\Omega^2}$ be two weak solutions to~\eqref{sde_beta_drift_general}. We want to prove that $\mathcal L^{\mathbb P^1} (z^1)=\mathcal L^{\mathbb P^2} (z^2)$. 
In particular, for $n=1,2$, $\mathbb P^n$-almost surely, the process $(z_t^n)_{t\in [0,T]}$ satisfies
\begin{align*}
z^n_t= \xi^n + \int_0^t \!\! \int_\R f(k) \Re( e^{-ikz^n_s} \mathrm dw^n(k,s)) +\beta^n_t+\int_0^t b(z^n_s,\mu^n_s) \mathrm ds, 
\end{align*}
where for every $t\in [0,T]$, $\mu^n_t= \mathcal L^{\mathbb P^n} (z_t^n | \mathcal G_t^{\mu^n, W^n})$, $\mathcal G_t^{\mu^n, W^n}:= \sigma\{ w^n(k,s), \mu^n_s \; ;k\in \R, s\leq t \}$ and $(\mu^n,w^n)$ is independent of $(\beta^n,\xi^n)$ under $\mathbb P^n$. 

For $n=1,2$, define the process $(h_t^n)_{t\in [0,T]}$ by $h_t^n(k):= \frac{1}{f(k)} \crochetk^{-\eta} (\lambda(k,\mu^n_t)-\widetilde{\lambda}(k,\mu^n_t))$ for every $k \in \R$ and for every $t\in [0,T]$.
It is $(\mathcal G_t^{\mu^n})_{t\in [0,T]}$-adapted, $\int_0^T \int_\R |h^n_t(k)|^2 \mathrm dk \mathrm dt$ is bounded and $(h_t^n)_{t\in [0,T]}$ satisfies for every $x\in \R$ and for every $t \in [0,T]$
\begin{align*}
(b-\widetilde{b})(x, \mu^n_t) = \int_\R f(k) \Big( \cos (kx) h_t^{\Re,i}(k) + \sin(kx) h_t^{\Im,i}(k)  \Big) \mathrm dk.
\end{align*}
Let us define $\mathbb Q^n$ as the absolutely continuous probability measure with respect to $\mathbb P^n$ with density
\begin{align*}
\frac{\mathrm d \mathbb Q^n}{\mathrm d \mathbb P^n}
=\exp \left(- \int_0^T \!\! \int_\R \Re \big( \overline{h^n_t(k)} \mathrm dw^n(k,t) \big) - \frac{1}{2}  \int_0^T \!\! \int_\R |h^n_t(k)|^2 \mathrm dk \mathrm dt \right).
\end{align*}
Let us denote $\mathrm d \widetilde{w}^n(k,t)=\mathrm dw^n(k,t) + h_t^n(k) \mathrm dk \mathrm dt$. It follows from Girsanov's Theorem that $\mathcal L^{\mathbb Q^n} (\widetilde{w}^n, \beta^n, \xi^n) = \mathcal L^{\mathbb P^n} (w^n, \beta^n, \xi^n)$
and that, for any $t \in [0,T]$, the $\sigma$-field $\sigma \{\widetilde w^n(k,t')-\widetilde w^n(k,t),$ $\beta^n_{t'}- \beta^n_{t}, \ k \in \R, \ t' \in [t,T] \}$ is independent of ${\mathcal G}_{t}^n$ under ${\mathbb Q}^n$. Moreover, $(z^n_t)_{t\in [0,T]}$ satisfies
\begin{align*}
z^n_t= \xi^n + \int_0^t \!\! \int_\R f(k) \Re( e^{-ikz^n_s} \mathrm d\widetilde{w}^n(k,s)) +
\beta^n_t+\int_0^t \widetilde{b}(z^n_s,\mu^n_s) \mathrm ds.
\end{align*}
Let us remark that $\Big(\exp (- \int_0^t  \int_\R \Re \big( \overline{h^n_s(k)} \mathrm dw^n(k,s) \big) - \frac{1}{2}  \int_0^t \int_\R |h^n_s(k)|^2 \mathrm dk \mathrm ds )\Big)_{t\in [0,T]}$ and $\widetilde{w}^n_{\cdot \wedge t} = w^n_{\cdot \wedge t} - \int_0^{\cdot\wedge t} \int_0^\cdot h^n_s(l) \mathrm dl \mathrm ds$
are $(\mathcal G_t^{\mu^n,W^n})_{t\in [0,T]}$-adapted.
Let us consider the same function $\varphi$ and $\psi$ as in equality~\eqref{preuvetheo41}. We obtain by a similar computation:
\begin{multline*}
\mathbb E^{\mathbb Q^n} \left[ \psi(z_t^n) \; \varphi(\mu^n_{\cdot \wedge t},\widetilde{w}^n_{\cdot \wedge t}) \right] \\
\begin{aligned}
&= \mathbb E^{\mathbb P^n} \left[ \psi(z^n_t) \; \varphi(\mu^n_{\cdot \wedge t},\widetilde{w}^n_{\cdot \wedge t}) 
\exp \Big( -  \int_0^t  \!\! \int_\R \Re \big( \overline{h^n_s(k)} \mathrm dw^n(k,s) \big) - \frac{1}{2}  \int_0^t \!\! \int_\R |h_s^n(k)|^2 \mathrm dk \mathrm ds \Big) \right]  \\
&= \mathbb E^{\mathbb P^n} \bigg[ \mathbb E^{\mathbb P^n} \left[ \psi(z_t^n)| \mathcal G_t^{\mu^n,W^n} \right] \varphi(\mu^n_{\cdot \wedge t},\widetilde{w}^n_{\cdot \wedge t})  \\ 
&\quad \quad  \quad 
\cdot \exp \Big(\! - \!\! \int_0^t \!\! \int_\R \Re \big( \overline{h^n_s(k)} \mathrm dw^n(k,s) \big) - \frac{1}{2}  \int_0^t\!\!  \int_\R |h_s^n(k)|^2 \mathrm dk \mathrm ds \Big) \bigg] \\
&= \mathbb E^{\mathbb P^n} \left[\int_\R \psi (x) \mathrm d\mu^n_t(x)  \; \varphi(\mu^n_{\cdot \wedge t},\widetilde{w}^n_{\cdot \wedge t}) \exp \Big( \! - \int_0^t\!\!   \int_\R \Re \big( \overline{h^n_s(k)} \mathrm dw^n(k,s) \big) - \frac{1}{2}  \int_0^t \!\! \int_\R |h_s^n(k)|^2 \mathrm dk \mathrm ds \Big) \right]  \\
&= \mathbb E^{\mathbb Q^n} \left[\int_\R \psi (x) \mathrm d\mu^n_t(x)  \; \varphi(\mu^n_{\cdot \wedge t},\widetilde{w}^n_{\cdot \wedge t}) \right],
\end{aligned}
\end{multline*}
and thus for every $t\in [0,T]$, $\mathbb Q^n$-almost surely, $\mu^n_t= \mathcal L^{\mathbb Q^n} (z_t^n |\mathcal G_t^{\mu^n,\widetilde{W}^n} )$.

Moreover $(\mu^n, \widetilde{w}^n)$ and $\frac{\mathrm d \mathbb Q^n}{\mathrm d \mathbb P^n}$ are $\mathcal G^{\mu^n,W^n}$-measurable and under $\mathbb P^n$, $(\mu^n,w^n)$ is independent of $(\beta^n,\xi^n)$. 
Thus for any bounded and measurable functions   $g:\mathcal C([0,T],\R) \times \R \to \R$ and $f:  \mathcal C([0,T],\mathcal P_2(\R)) \times \mathcal C(\R \times [0,T] , \R^2) \to \R$, we have
\begin{align*}
\mathbb E^{\mathbb Q^n} \left[ f(\mu^n, \widetilde{w}^n) \; g(\beta^n, \xi^n) \right]
&=\mathbb E^{\mathbb P^n} \left[ f(\mu^n, \widetilde{w}^n)  \; \frac{\mathrm d \mathbb Q^n}{\mathrm d \mathbb P^n}\; g(\beta^n, \xi^n) \right] \\
&=\mathbb E^{\mathbb P^n} \left[ f(\mu^n, \widetilde{w}^n)  \; \frac{\mathrm d \mathbb Q^n}{\mathrm d \mathbb P^n}\right] \cdot \mathbb E^{\mathbb P^n} \left[ g(\beta^n, \xi^n) \right] \\
&=\mathbb E^{\mathbb Q^n} \left[ f(\mu^n, \widetilde{w}^n)  \right] \cdot \mathbb E^{\mathbb Q^n} \left[ g(\beta^n, \xi^n) \right].
\end{align*}
Thus under $\mathbb Q^n$, $(\mu^n, \widetilde{w}^n)$ is independent of $(\beta^n,\xi^n)$. 
By the same argument and by the compatibility property under $\mathbb P^n$,  we get that, under~${\mathbb Q}^n$, for any $t \in [0,T]$, 
$(\xi^n,\widetilde w^n,\mu^n)$ and $\beta^n$ are conditionally independent given ${\mathcal G}_{t}^n$, which 
proves compatibility 
under ${\mathbb Q}^n$.

Thus we deduce that for $n=1,2$, $(\Omega^n, \mathcal G^n, (\mathcal G^n_t)_{t\in [0,T]}, \mathbb Q^n, z^n,\widetilde{w}^n, \beta^n, \xi^n)$ are  weak solutions to the  SDE~\eqref{sde_beta_drift_lipschitz} with drift $\widetilde{b}$. 
By Proposition~\ref{prop:drift_lipschitz_uniqu}, it follows that $\mathcal L^{\mathbb Q^1} (z^1, \widetilde{w}^1)=\mathcal L^{\mathbb Q^2} (z^2, \widetilde{w}^2)$ and that for every $t\in [0,T]$, $\mu^n_t= \mathcal L^{\mathbb Q^n} (z_t^n |\mathcal G_t^{\widetilde{W}^n} )$.
Then, we apply the same computation as~\eqref{loi_de_zi_wi}: for each bounded and measurable $\phi:\mathcal C( [0,T], \R) \to \R$, 
\begin{align*}
\mathbb E^{\mathbb P^n} \left[ \phi(z^n) \right]
= \mathbb E^{\mathbb Q^n} \left[ \phi(z^n) \exp \Big(
\int_0^T \!\!\int_\R  \Re \big(\overline{ h_t^n(k) }\mathrm d\widetilde{w}^n(k,t)\big) - \frac{1}{2} \int_0^T \!\!\int_\R |h_t^n(k)|^2 \mathrm dk \mathrm dt
\Big) \right].
\end{align*}
Recall that $h_t^n(k)= \frac{1}{f(k)} \crochetk^{-\eta} (\lambda(k,\mu^n_t)-\widetilde{\lambda}(k,\mu^n_t))$ and that $\mu^n_t= \mathcal L^{\mathbb Q^n} (z_t^n |\mathcal G_t^{\widetilde{W}^n} )$. Hence the process $(h_t^n)_{t\in [0,T]}$ is $(\mathcal G_t^{\widetilde{W}^n})_{t\in [0,T]}$-progressively measurable. It follows that  there is a measurable map $\psi: \mathcal C(\R \times [0,T] , \R^2) \to \R$, independent of $n$, such that $ \mathbb E^{\mathbb Q^n} \left[| \psi(\widetilde{w}^n)| \right] <+\infty$ and
\begin{align*}
\mathbb E^{\mathbb P^1} \left[ \phi(z^1) \right]
= \mathbb E^{\mathbb Q^1} \left[\phi(z^1)\psi(\widetilde{w}^1) \right]
= \mathbb E^{\mathbb Q^2} \left[\phi(z^2)\psi(\widetilde{w}^2) \right]
=\mathbb E^{\mathbb P^2} \left[ \phi(z^2) \right].
\end{align*}
 We conclude that $\mathcal L^{\mathbb P^1} (z^1)=\mathcal L^{\mathbb P^2} (z^2)$. 
This completes the proof of Theorem~\ref{theo:interpolation_lacker}. 
\end{proof}

\begin{rem}
As the last computation right above highlights it, we have in fact that 
${\mathcal L}^{\mathbb P^1}(z^1,\widetilde w^1)=
{\mathcal L}^{\mathbb P^2}(z^2,\widetilde w^2)
$ and then 
${\mathcal L}^{\mathbb P^1}(z^1,\mu^1)=
{\mathcal L}^{\mathbb P^2}(z^2,\mu^2)
$.
\end{rem}

\bibliographystyle{alpha}

\bibliography{ref_victor} 
\end{document}